\documentclass[11pt,reqno,final]{amsart}

\usepackage{amsmath}
\usepackage{amscd}
\usepackage{amssymb}
\usepackage{latexsym}
\usepackage{stmaryrd} 
\usepackage{enumerate}
\usepackage[table]{xcolor}

\usepackage{graphicx} 
\usepackage{showkeys}

\usepackage[arrow,curve,matrix,tips,2cell]{xy}
  \SelectTips{eu}{10} \UseTips
  \UseAllTwocells
\definecolor{lgray}{rgb}{0.9,0.9,0.9}
\newcommand{\nsp}[1]{\cellcolor{lgray}{#1}}

\newtheorem{theorem}{Theorem}[section]
\newtheorem*{theorem*}{Theorem}
\newtheorem{lemma}[theorem]{Lemma}
\newtheorem*{lemma*}{Lemma}
\newtheorem{corollary}[theorem]{Corollary}
\newtheorem*{corollary*}{Corollary}
\newtheorem{proposition}[theorem]{Proposition}

\newtheorem{remark}[theorem]{Remark}
\newtheorem{definition}[theorem]{Definition}

\newcommand{\bgl}{\begin{equation}} 
\newcommand{\egl}{\end{equation}}
\newcommand{\bgloz}{\begin{equation*}} 
\newcommand{\egloz}{\end{equation*}}
\newcommand{\bgln}{\begin{eqnarray}} 
\newcommand{\egln}{\end{eqnarray}}
\newcommand{\bglnoz}{\begin{eqnarray*}} 
\newcommand{\eglnoz}{\end{eqnarray*}}
\newcommand{\btheo}{\begin{theorem}}
\newcommand{\etheo}{\end{theorem}}
\newcommand{\btheooz}{\begin{theorem*}}
\newcommand{\etheooz}{\end{theorem*}}
\newcommand{\blemma}{\begin{lemma}}
\newcommand{\elemma}{\end{lemma}}
\newcommand{\blemmaoz}{\begin{lemma*}}
\newcommand{\elemmaoz}{\end{lemma*}}
\newcommand{\bproof}{\begin{proof}}
\newcommand{\eproof}{\end{proof}}
\newcommand{\bbew}{\begin{beweis}}
\newcommand{\ebew}{\end{beweis}}
\newcommand{\bremark}{\begin{remark}\em}
\newcommand{\eremark}{\end{remark}}
\newcommand{\bdefin}{\begin{definition}}
\newcommand{\edefin}{\end{definition}}
\newcommand{\bprop}{\begin{proposition}}
\newcommand{\eprop}{\end{proposition}}
\newcommand{\bcor}{\begin{corollary}}
\newcommand{\ecor}{\end{corollary}}
\newcommand{\bcoroz}{\begin{corollary*}}
\newcommand{\ecoroz}{\end{corollary*}}
\newcommand{\bfa}{\begin{cases}} 
\newcommand{\efa}{\end{cases}}
\newcommand{\Ext}{\operatorname{Ext}}
\newcommand{\cK}{\mathcal K}
\newcommand{\cL}{\mathcal L}

\newcommand{\cO}{\mathcal O}

\newcommand{\cT}{\mathcal T}

\def\Cz{\mathbb{C}}

\def\Nz{\mathbb{N}}

\def\Tz{\mathbb{T}}
\def\Zz{\mathbb{Z}}

\def\1z{\mathbb{1}}
\newcommand{\an}[1]{``#1''} 

\newcommand{\ma}{\mapsto} 
\newcommand\onto{\twoheadrightarrow} 
\newcommand\into{\hookrightarrow} 
\newcommand{\Rarr}{\Rightarrow} 
\newcommand{\Larr}{\Leftarrow} 

\newcommand{\sgn}{\operatorname{sgn}}
\newcommand{\ve}{\varepsilon}

\def\SEMI{\mbox{$\times\kern-2pt\vrule height5pt width.6pt \kern3pt $}}

\newcommand{\Prim}{{\rm Prim\,}} 

\newcommand{\id}{{\rm id}}

\renewcommand{\ker}{{\rm ker}\,}

\newcommand{\abs}[1]{\lvert#1\rvert} 
\newcommand{\norm}[1]{\left\|#1\right\|} 
\newcommand{\defeq}{\mathrel{:=}} 

\newcommand{\dop}{\text{: }} 

\newcommand{\ftn}[3]{ #1 \colon #2 \rightarrow #3 }
\newcommand{\kk}{\operatorname{KK}}
\newcommand{\ksix}{\operatorname{K}_{\mathrm{six}}^{+}}
\newcommand{\lge}{\left\{} 
\newcommand{\rge}{\right\}} 
\newcommand{\lru}{\left(} 
\newcommand{\rru}{\right)} 
\newcommand{\leck}{\left[} 
\newcommand{\reck}{\right]} 
\newcommand{\lsp}{\left\langle} 
\newcommand{\rsp}{\right\rangle} 
\newcommand{\rukl}[1]{\lru #1 \rru} 
\newcommand{\eckl}[1]{\leck #1 \reck} 
\newcommand{\gekl}[1]{\lge #1 \rge} 
\newcommand{\spkl}[1]{\lsp #1 \rsp} 
\newcommand{\menge}[2]{\gekl{ #1 \dop #2 }} 
\newcommand{\FK}{FK_+}

\begin{document}

\title[The isomorphism problem for C*-algebras of Artin monoids]{The isomorphism problem for semigroup C*-algebras of right-angled Artin monoids}
\begin{abstract}
Semigroup C*-algebras for right-angled Artin monoids were introduced and studied by Crisp and Laca. In the paper at hand, we are able to present  the complete answer to their question of when such C*-algebras are isomorphic. The answer to this question is presented both in terms of properties of the graph defining the Artin monoids as well as in terms of classification by $K$-theory, and is obtained using recent results from classification of non-simple C*-algebras.

Moreover, we are able to answer another natural question: Which of these semigroup C*-algebras for right-angled Artin monoids are isomorphic to graph algebras? We give a complete answer, and note the consequence that many of the C*-algebras under study are semiprojective.
\end{abstract}

\author{S{\o}ren Eilers}
\address{Department of Mathematical Sciences \\
        University of Copenhagen\\
        Universitetsparken~5 \\
        DK-2100 Copenhagen, Denmark}
        \email{eilers@math.ku.dk }

\author{Xin Li}
\address{School of Mathematical Sciences\\
Queen Mary University of London\\
Mile End Road\\
London E1 4NS}
\email{xin.li@qmul.ac.uk}
\author{Efren Ruiz}
\address{Department of Mathematics\\University of Hawaii,
Hilo\\200 W. Kawili St.\\
Hilo, Hawaii\\
96720-4091 USA}
        \email{ruize@hawaii.edu}




\maketitle


\setlength{\parindent}{0pt} \setlength{\parskip}{0.5cm}

\section{Introduction}

Semigroup C*-algebras for right-angled Artin monoids were introduced and studied by Crisp and Laca in \cite{jcml:tarfag} and \cite{jcml:bqitaag}. In \cite{jcml:bqitaag}, the authors ask how to classify these semigroup C*-algebras up to *-isomorphism. We now present the complete answer to their question.

The Artin monoids studied here are given by countable, symmetric and antireflexive graphs $\Gamma=(V,E)$ as
\bgloz
  A_{\Gamma}^+ \defeq \spkl{\menge{\sigma_v}{v \in V} \vert \sigma_v \sigma_w = \sigma_w \sigma_v \text{ if } (v,w) \in E}^+.
\egloz
The corresponding right-angled Artin groups, defined by the same generators and relations, are special cases of Artin groups, which form an important class of examples of groups. We refer the reader to \cite{{jcml:tarfag},{jcml:bqitaag}} and the references therein for more details.

Semigroup C*-algebras of left cancellative semigroups, generated by the left regular representation of the semigroup, have been studied for a long time. Recently, there has been a renewed interested in this topic (see \cite{{xl:scas},{xl:nscca}} and the references therein). By \cite{jcml:tarfag}, the semigroup C*-algebras $C^*(  A_{\Gamma}^+)$ attached to right-angled Artin monoids are given as the universal C*-algebras for 
\[
\left\langle\menge{s_v}{v \in V} \left|
  \begin{array}{c}\eckl{s_v,s_w} = \eckl{s_v,s_w^*} = 0 \ \text{if} \ (v,w) \in E \\
  s_v^* s_w  =\delta_{v,w}\ \text{if} \ (v,w) \not\in E \end{array}\right\rangle\right.
\]
We answer the question of when two graphs $\Gamma,\Lambda$ produce C*-algebras that are isomorphic. Although we emphasize that our results cover the full range of such graphs, it is instructive to state our main results in the case of finite graphs. This is a specialization of the combination of Theorems \ref{general-theo} and \ref{K-theory-main}.

\begin{theorem}
\label{theo_finite-graphs}
Let $\Gamma$ and $\Lambda$ be finite undirected graphs with no loops. The following are equivalent
\begin{enumerate}
\item $C^*(A_{\Gamma}^+)\cong C^*(A_{\Lambda}^+)$
\item \begin{enumerate}
\item $t(\Gamma)=t(\Lambda)$
\item $N_k(\Gamma)+N_{-k}(\Gamma)=N_k(\Lambda)+N_{-k}(\Lambda)$ for all $k\in \Zz$
\item $N_0(\Gamma)>0$ or 
\[
\sum_{k>0}N_k(\Gamma)\equiv \sum_{k>0}N_k(\Lambda) \mod 2
\]
\end{enumerate}
\item $[\FK(C^*(A_{\Gamma}^+)),[1_{C^*(A_{\Gamma}^+)}]]\cong [\FK(C^*(A_{\Lambda}^+)),[1_{C^*(A_{\Lambda}^+)}]]$
\end{enumerate}
\end{theorem}

In this result, the invariant mentioned in (3)  is the standard ordered filtered $K$-theory -- implicitly containing the primitive ideal space -- which has been conjectured in \cite{segrer:gclil} to be a complete invariant for a large and important class of C*-algebras. This conjecture is still open, but has been confirmed in a multitude of situations partially overlapping with the case at hand. But the main strength of our result is in fact that the \emph{ad hoc} invariant of (2) is extremely easy to compute for $\Gamma$ and $\Lambda$. Indeed, as we shall detail below, the numbers $t(\Gamma)$ and $N_k(\Gamma)$ are obtained by dividing $\Gamma$ into co-irreducible components and then counting how many of these are singletons, yielding $t(\Gamma)$, and counting how many of the remaining co-irreducible components have Euler characteristic $k$, yielding $N_k(\Gamma)$. In Figure \ref{fivecases} this process has been completed for all 34 graphs with five vertices, and we conclude that they define 18 different C*-algebras.

\begin{figure}
\newcommand{\getp}[1]{\includegraphics[scale=.08]{penta#1}}
\begin{center}
\begin{tabular}{|c|c|c|c|c|c|c|}\hline
\getp{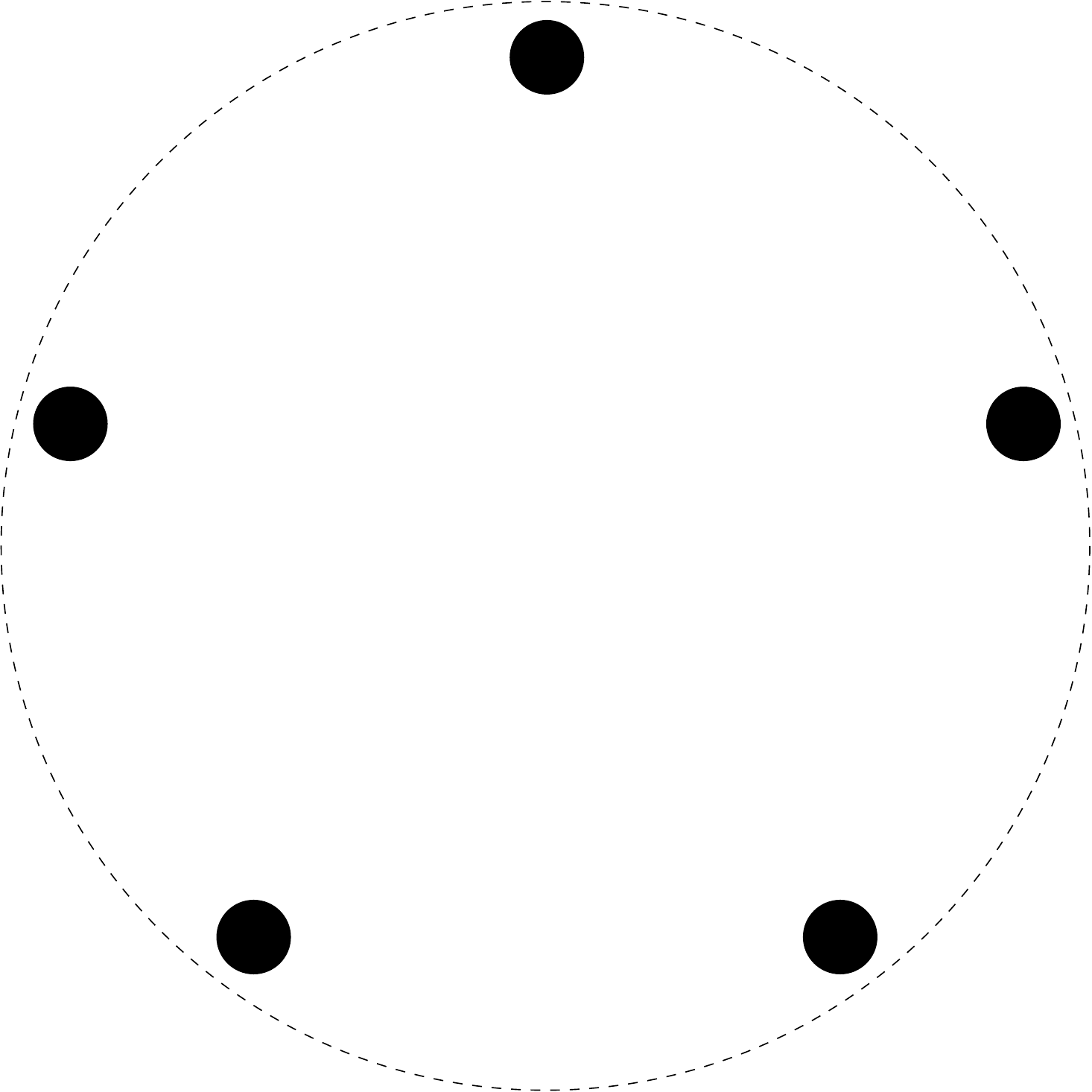}&\getp{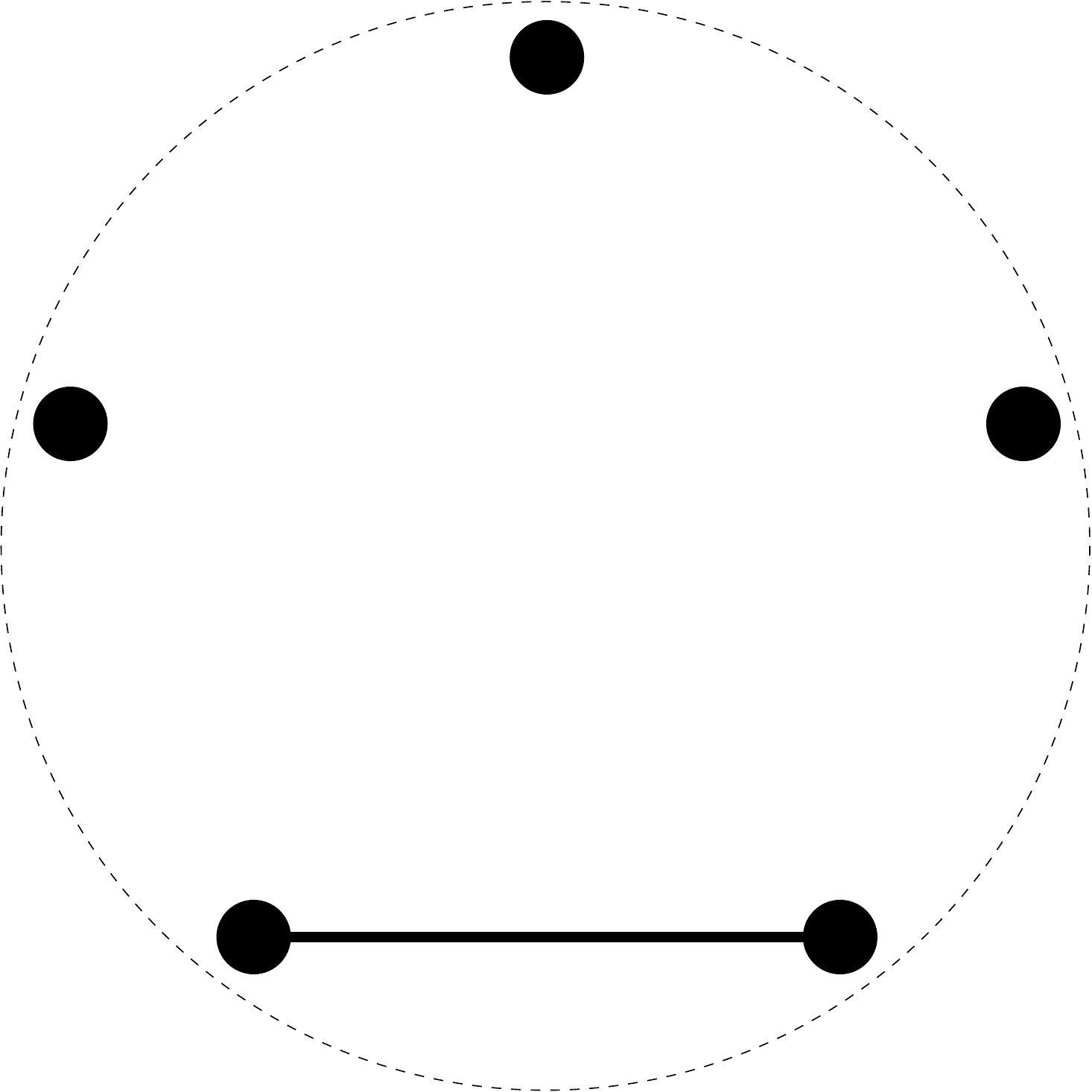}&\getp{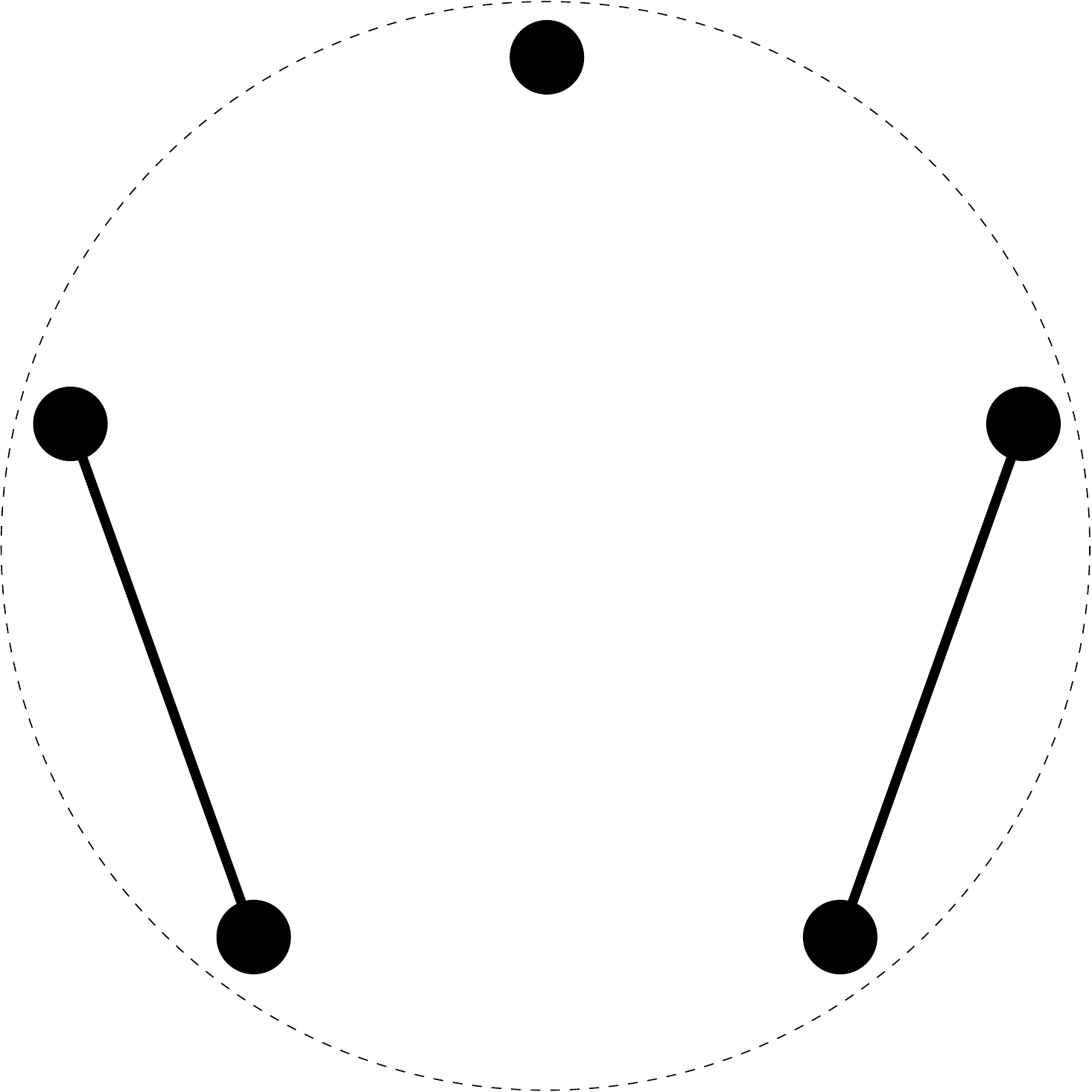}&\getp{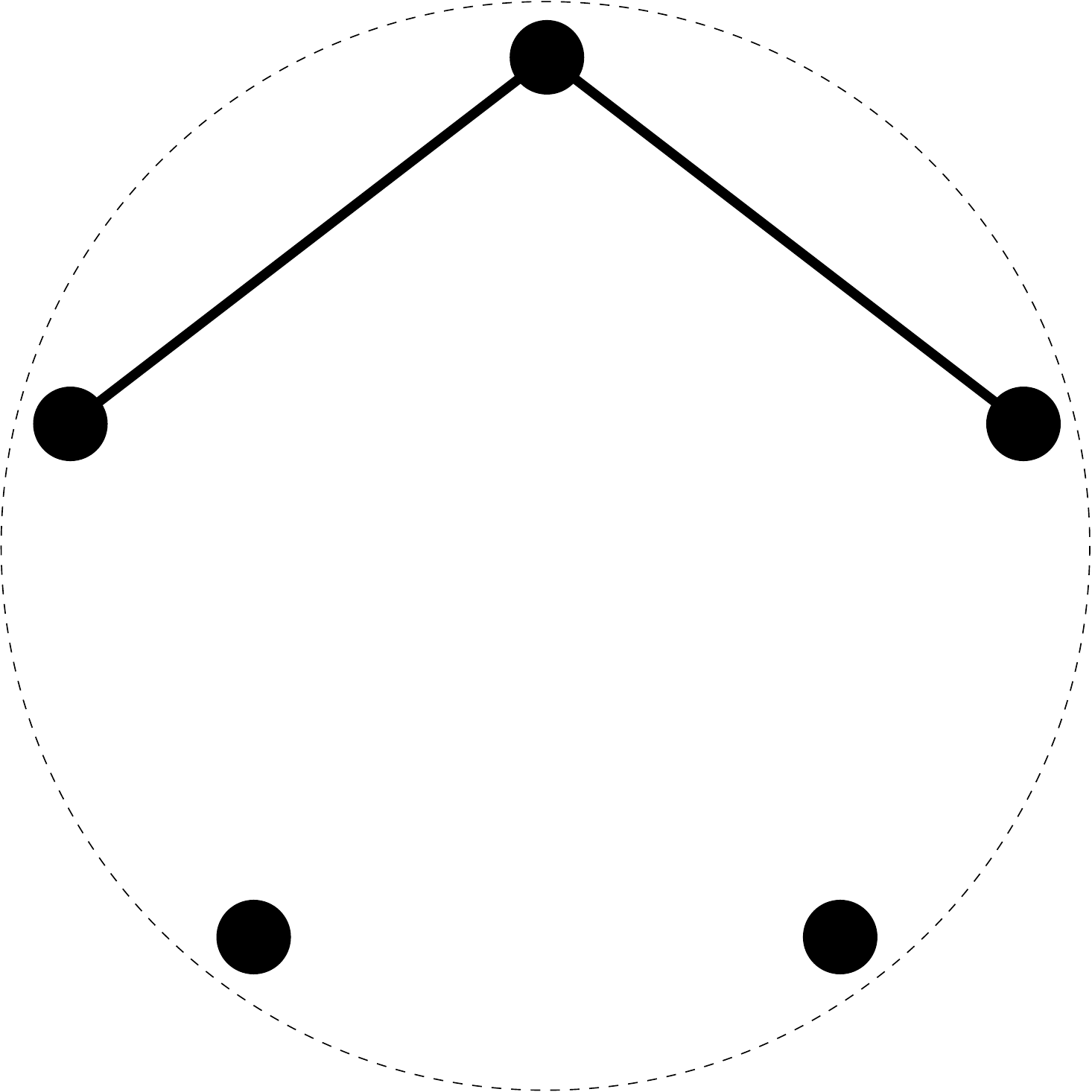}&\getp{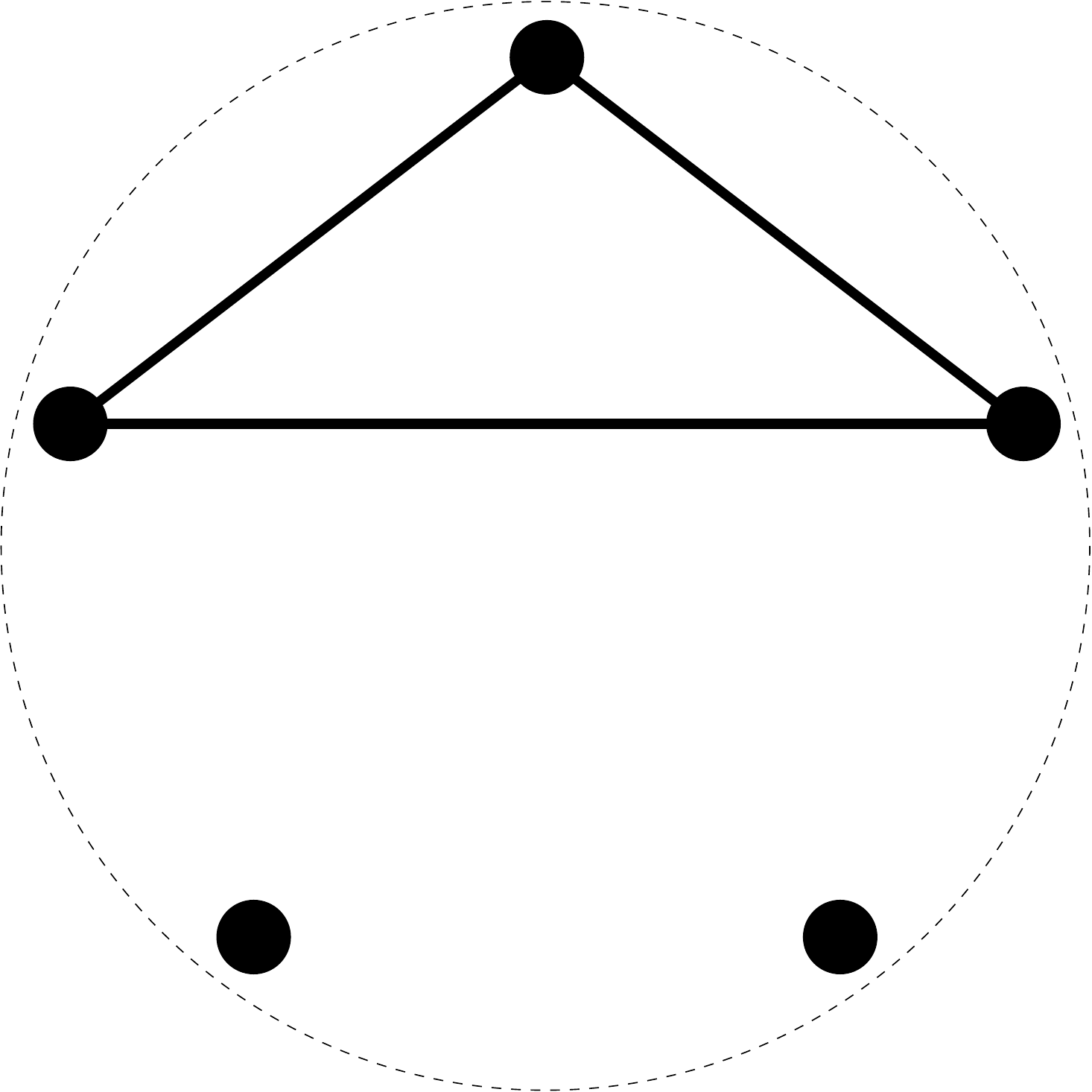}&\getp{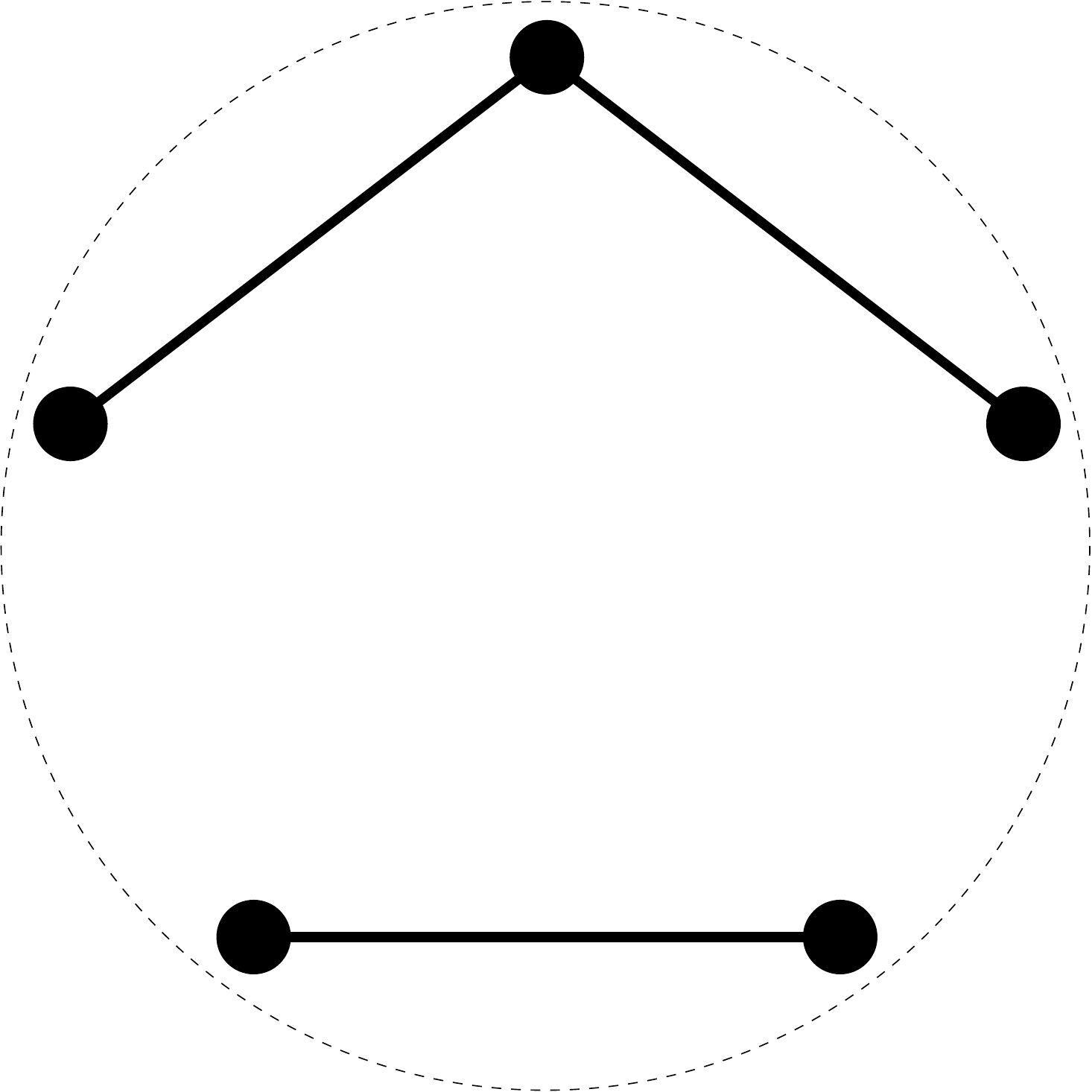}&\getp{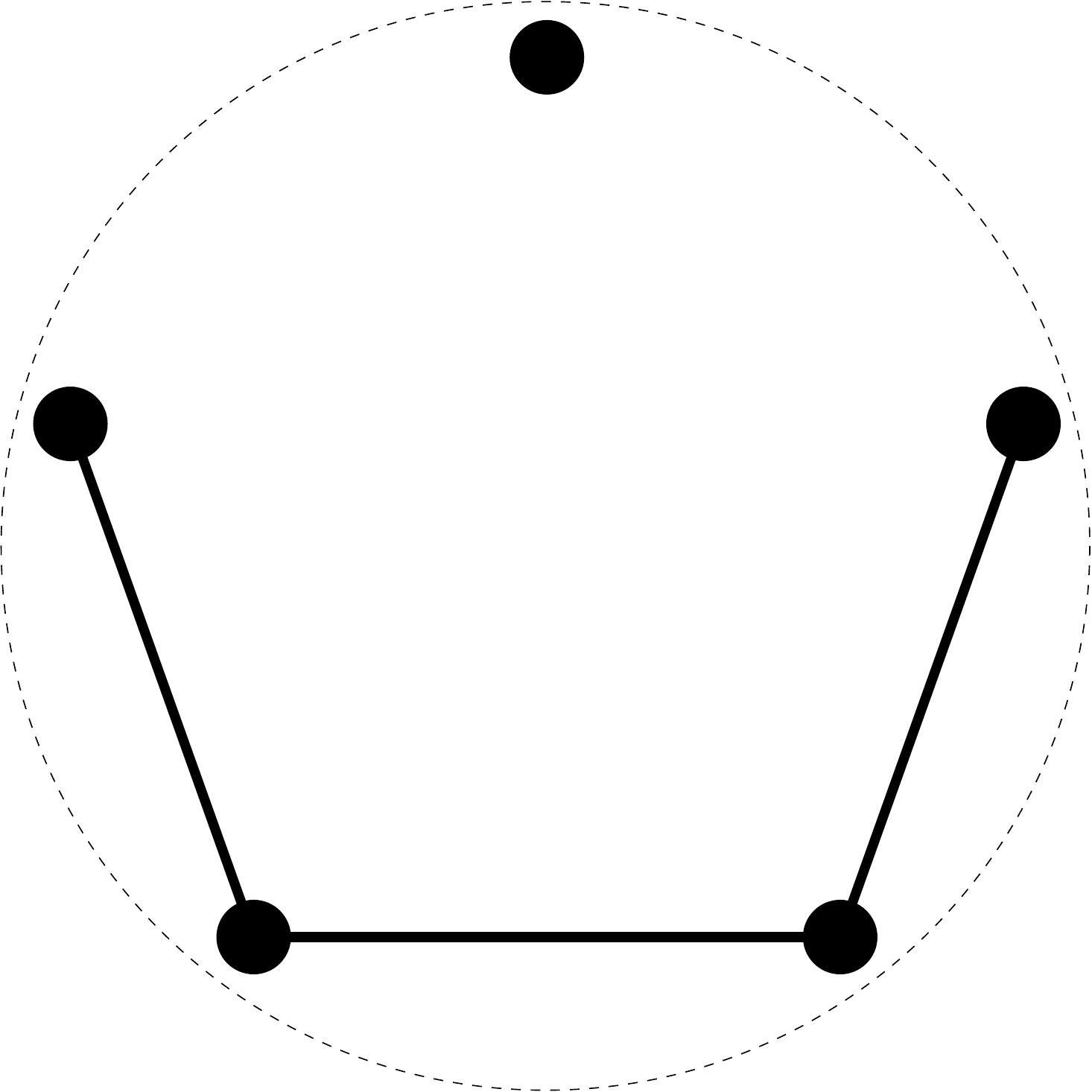}\\
$N_{-4}=1$&$N_{-3}=1$&$N_{-2}=1$&$N_{-2}=1$&$N_{-2}=1$&$N_{-1}=1$&$N_{-1}=1$\\\hline
\getp{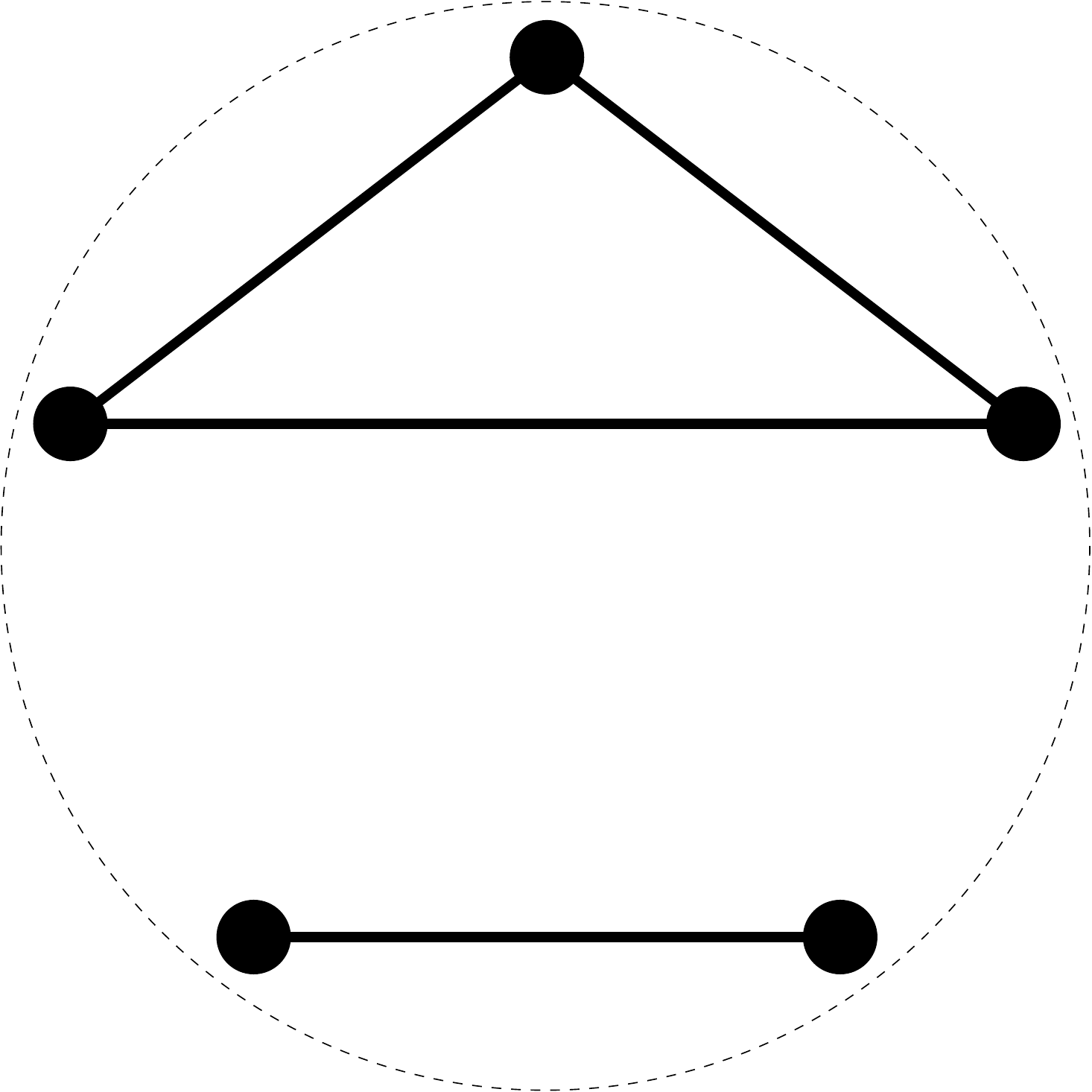}&\getp{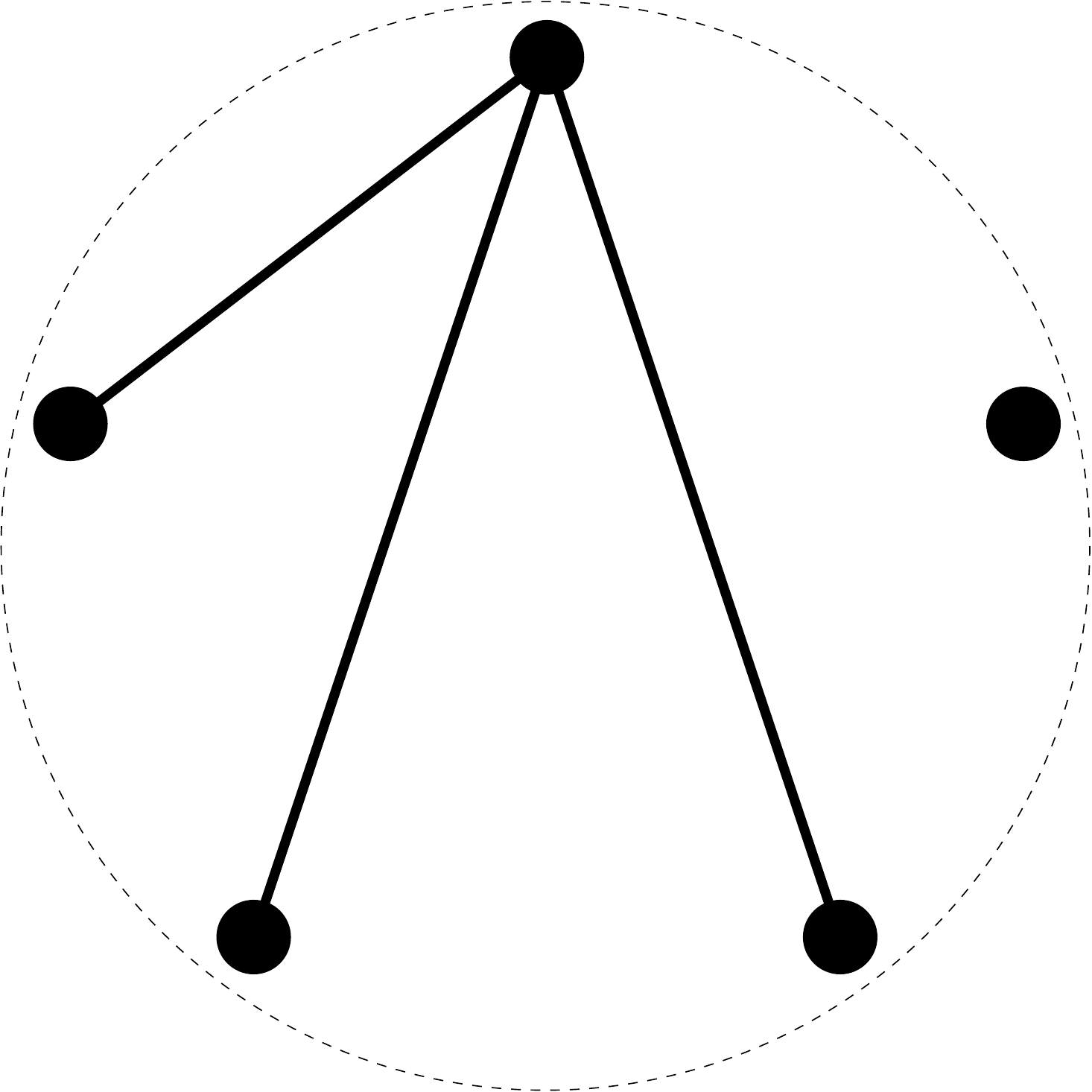}&\getp{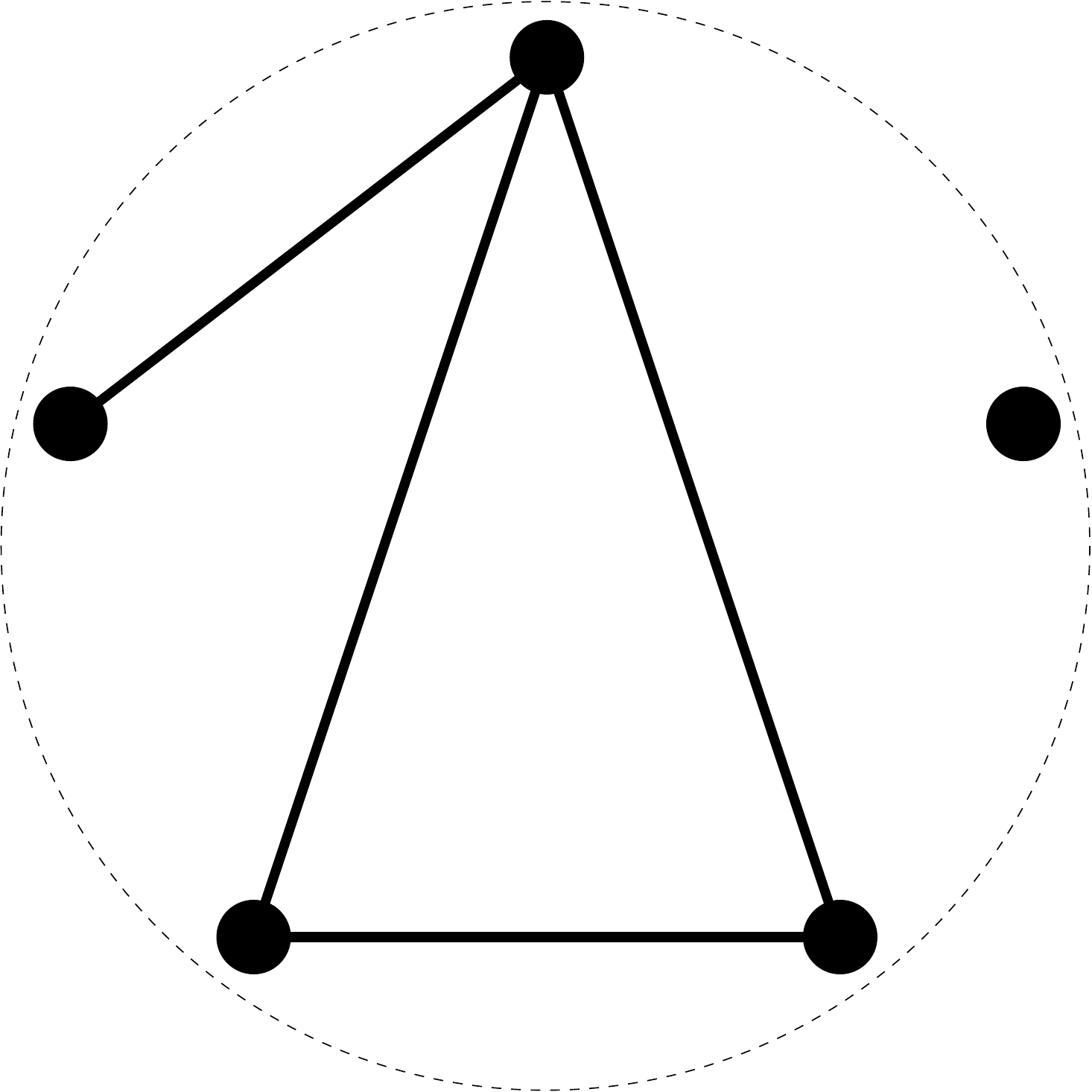}&\getp{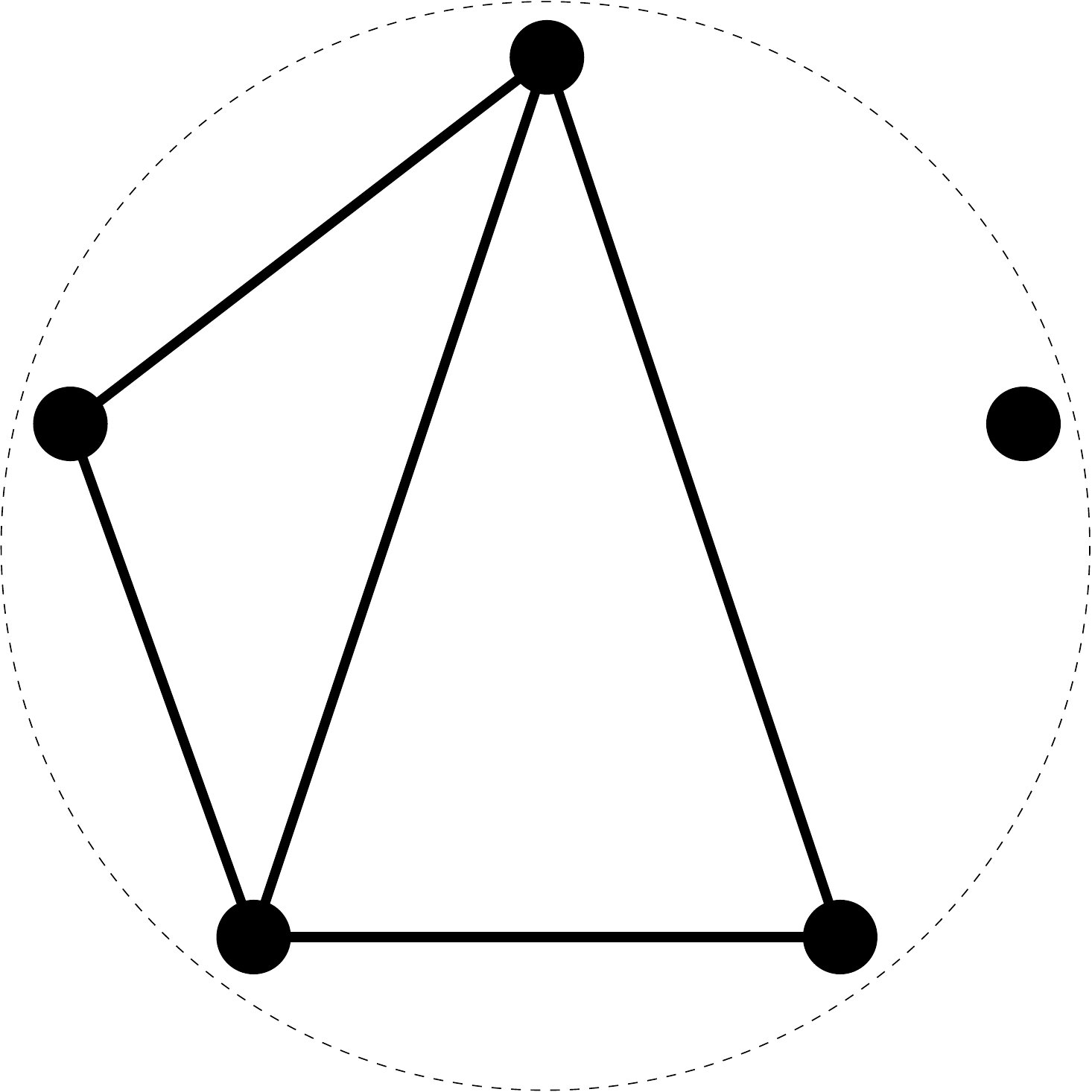}&\getp{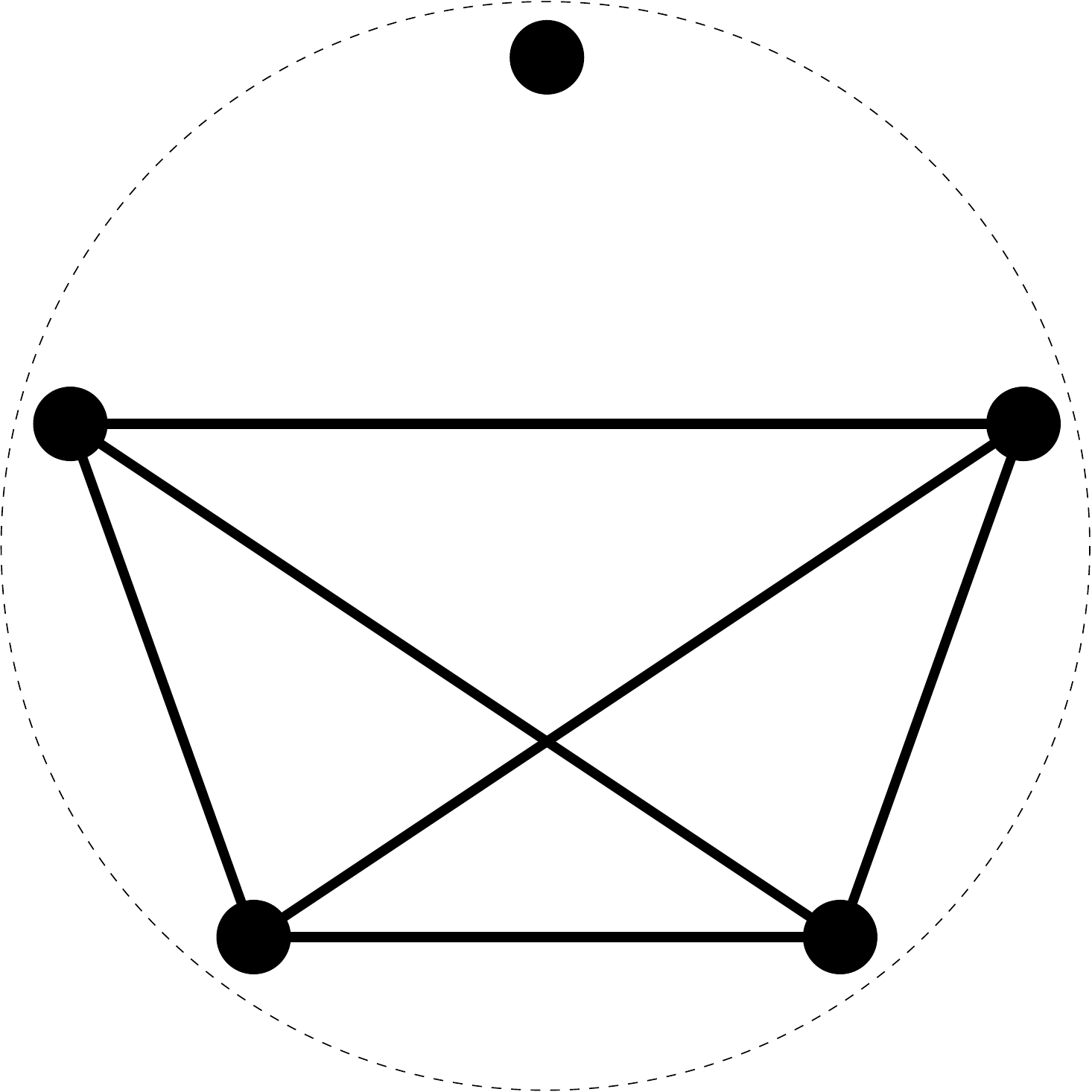}&\getp{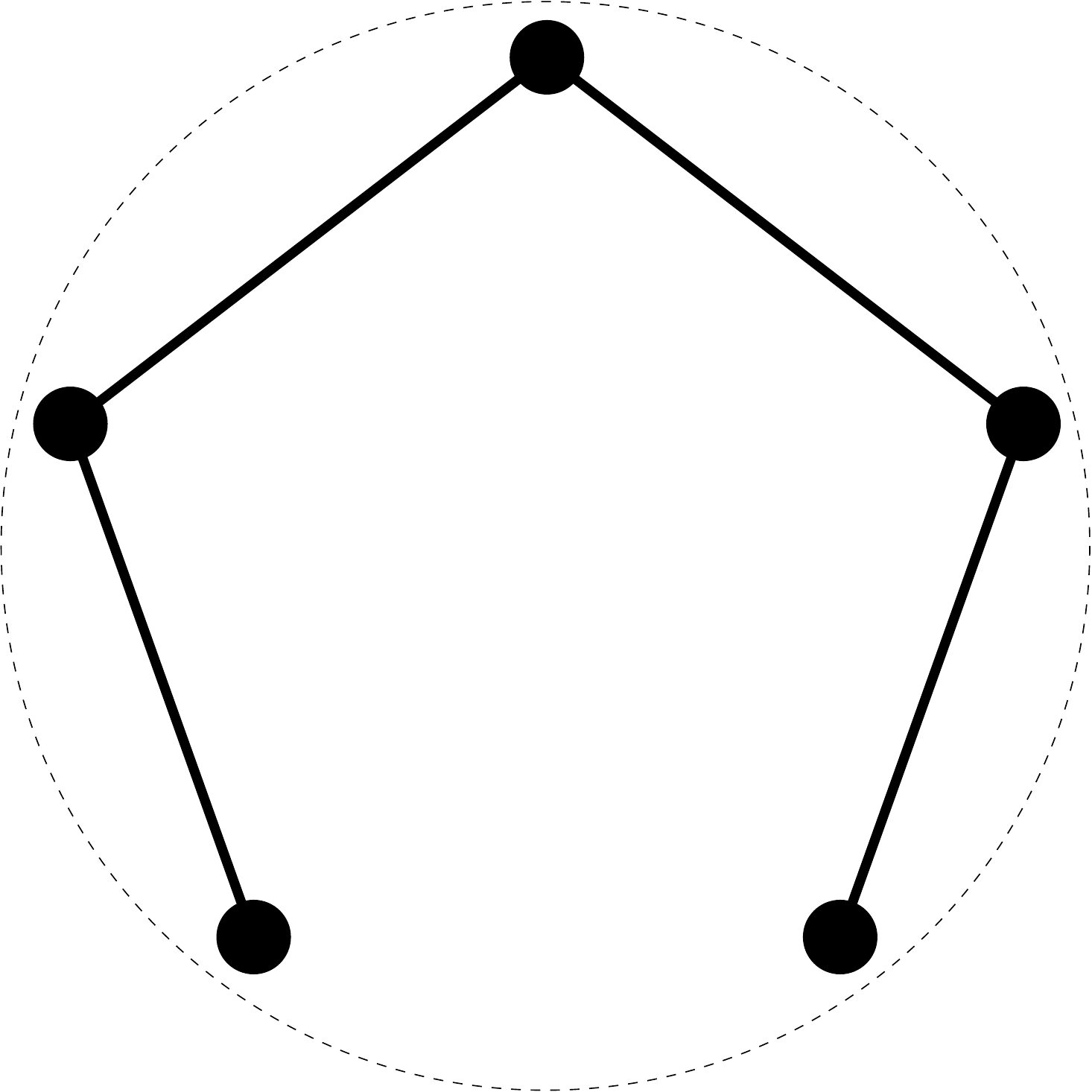}&\getp{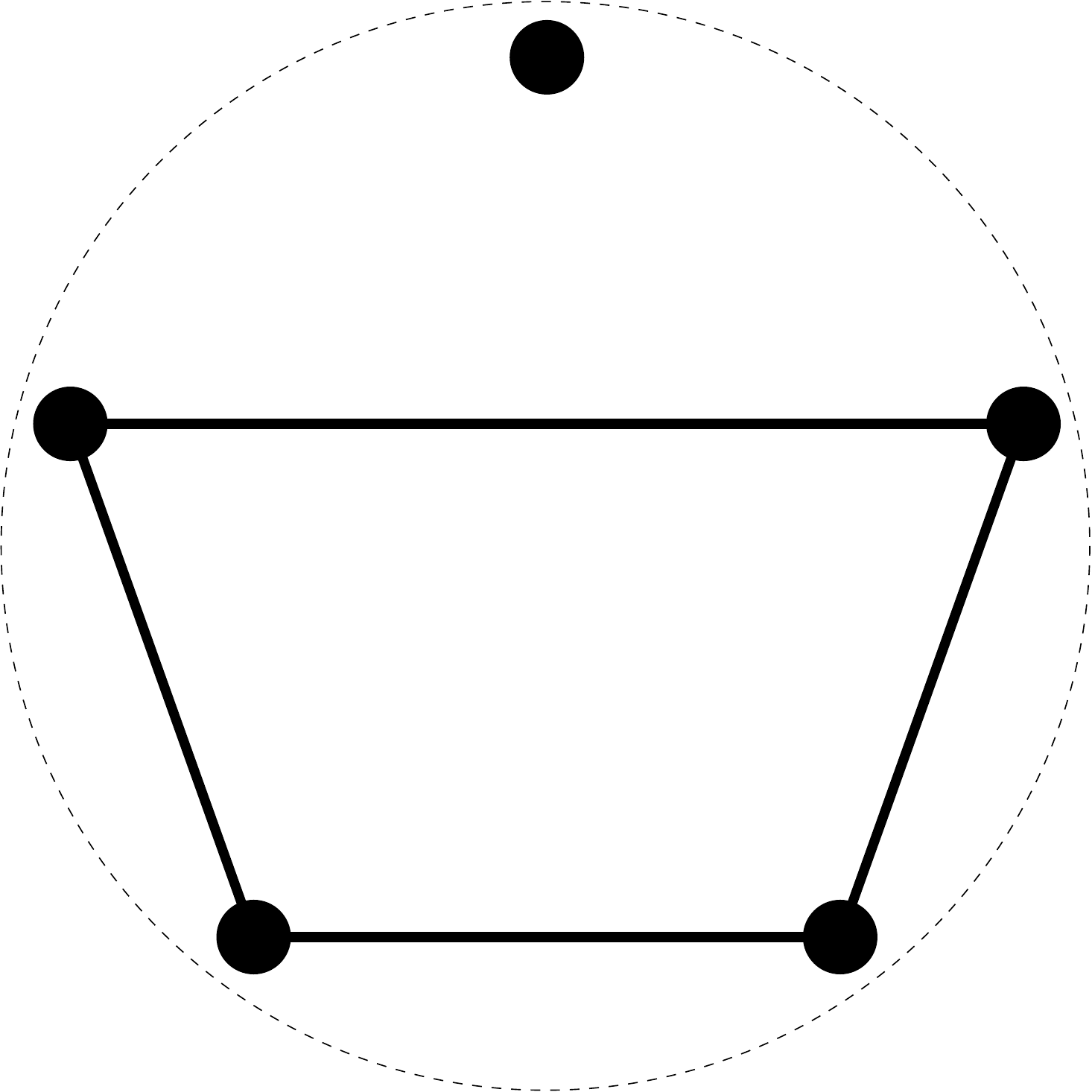}\\
$N_{-1}=1$&$N_{-1}=1$&$N_{-1}=1$&$N_{-1}=1$&$N_{-1}=1$&$N_0=1$&$N_0=1$\\\hline
\getp{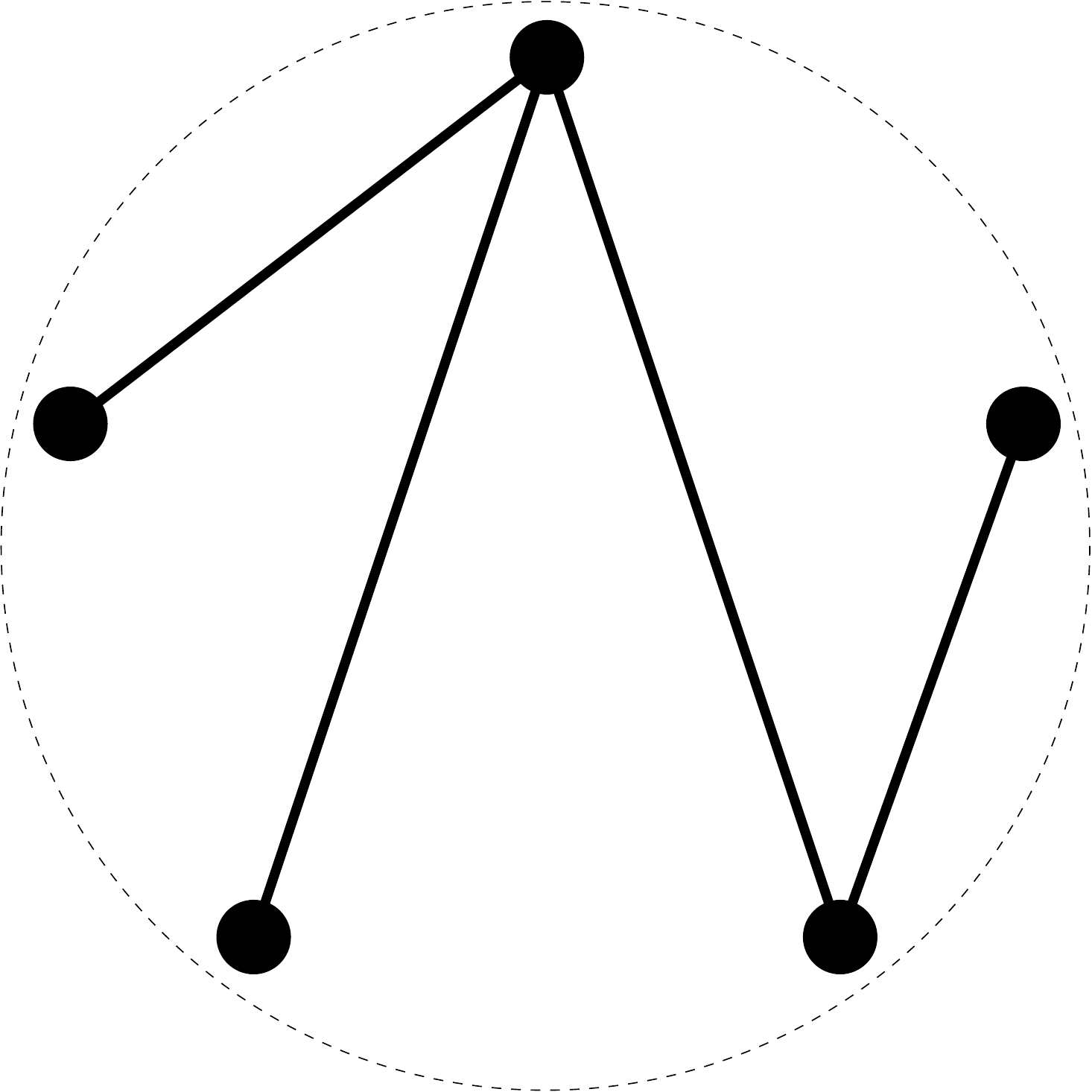}&\getp{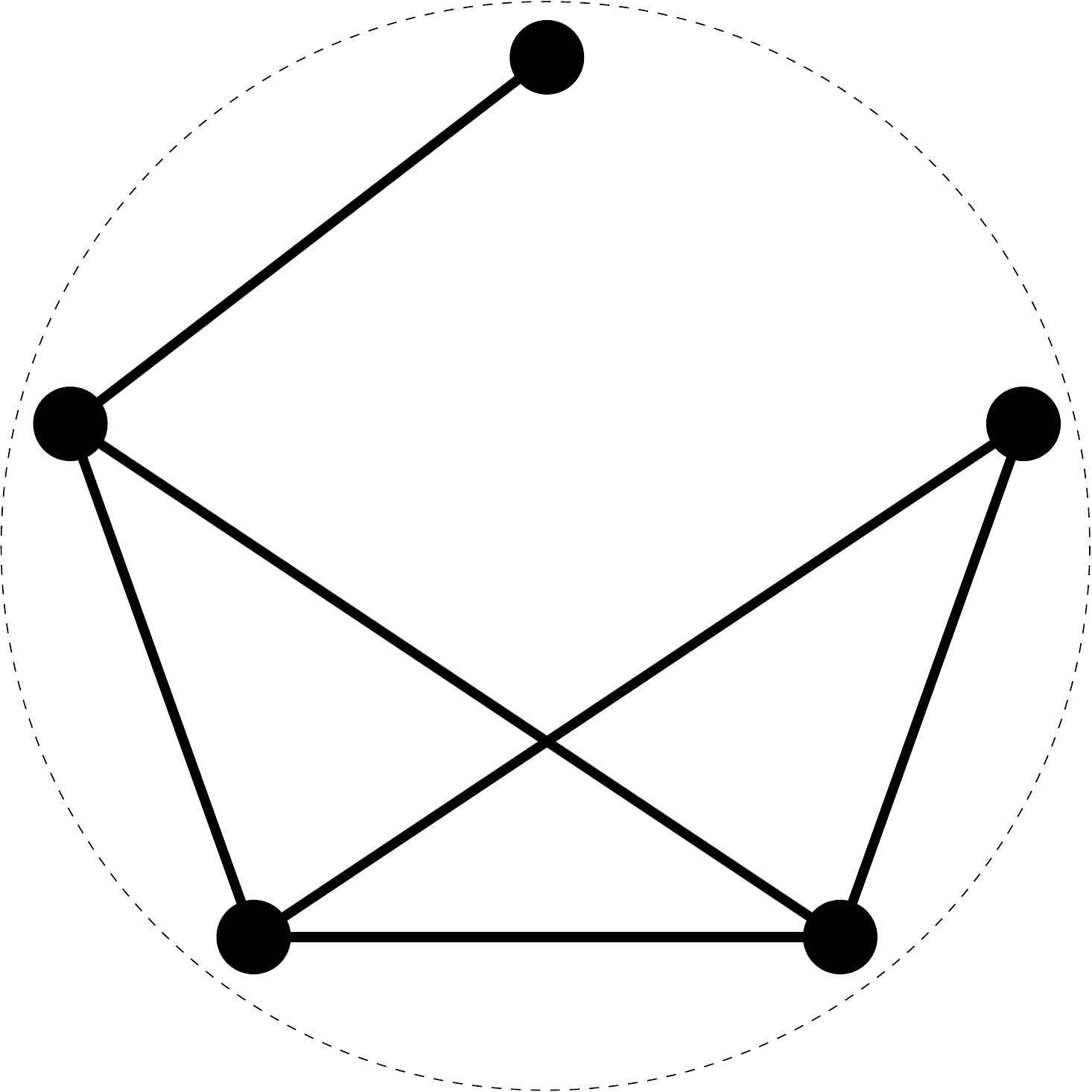}&\getp{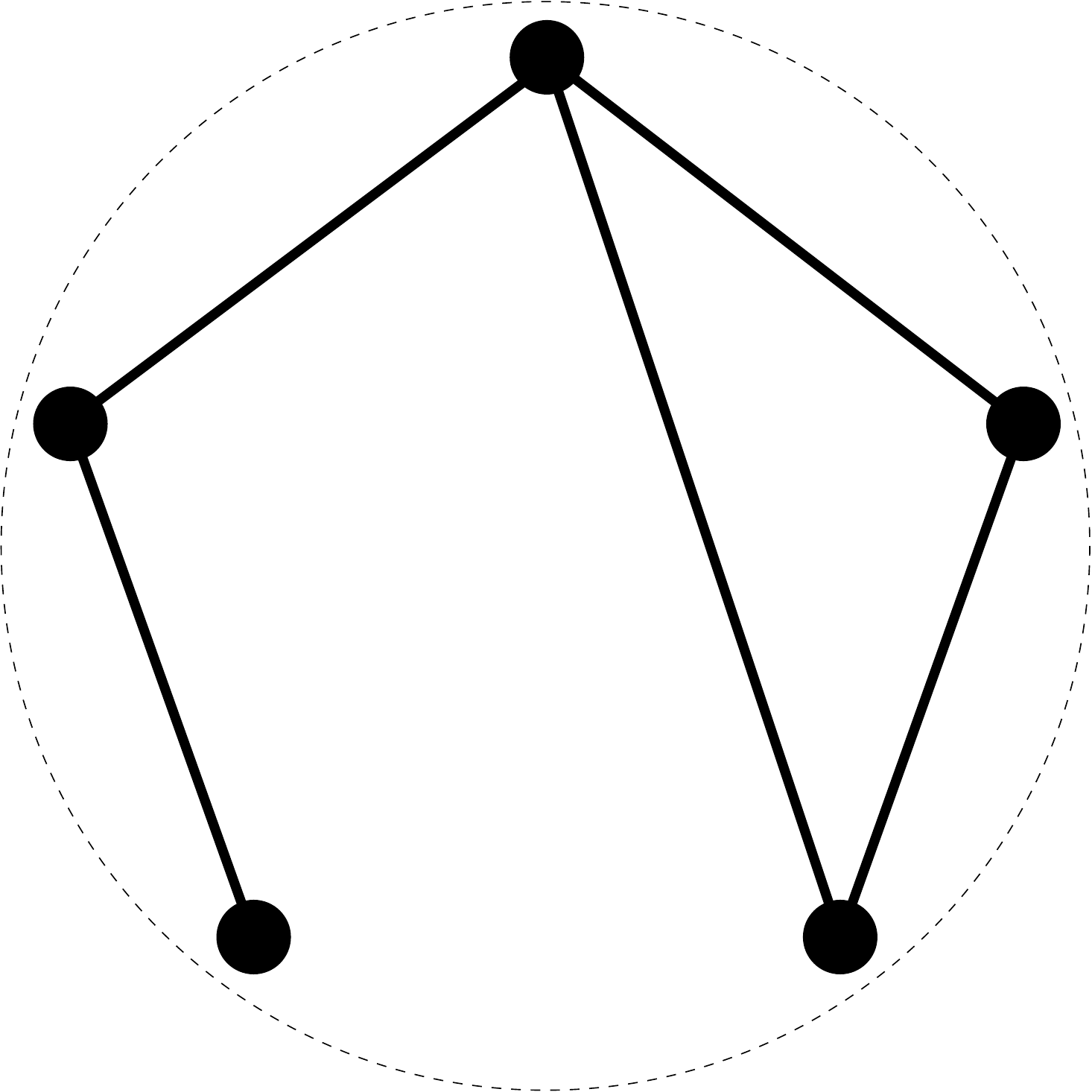}&\getp{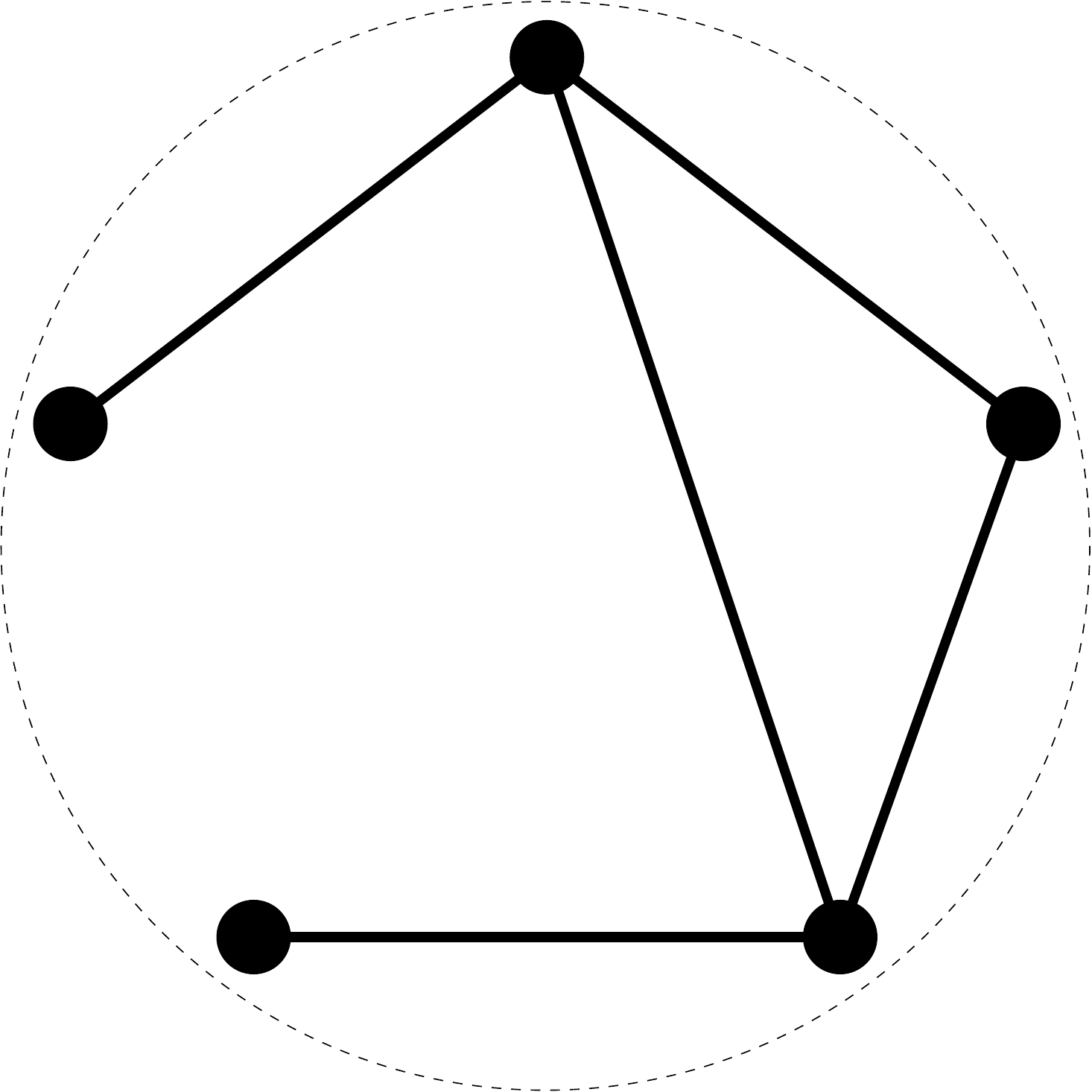}&\getp{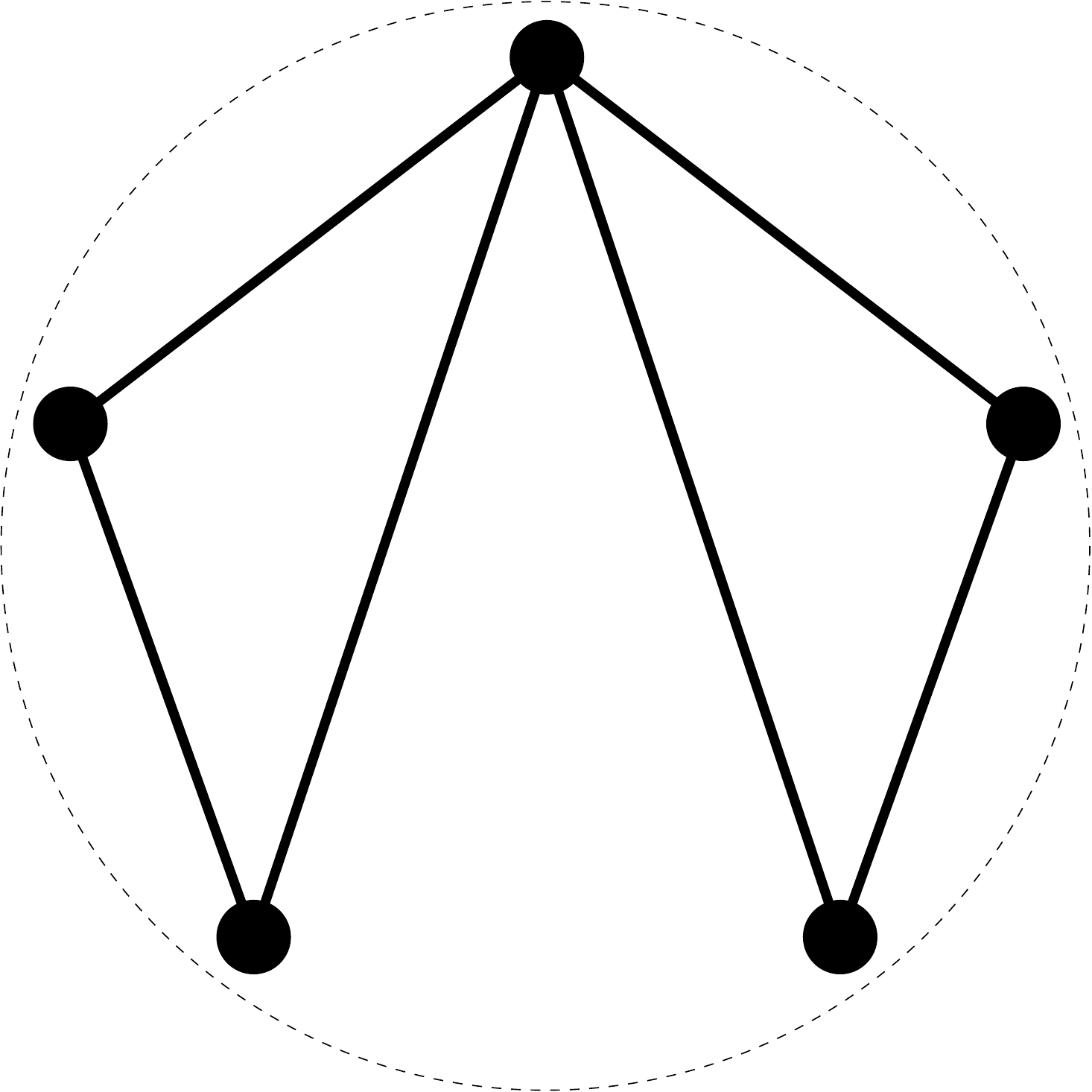}&\getp{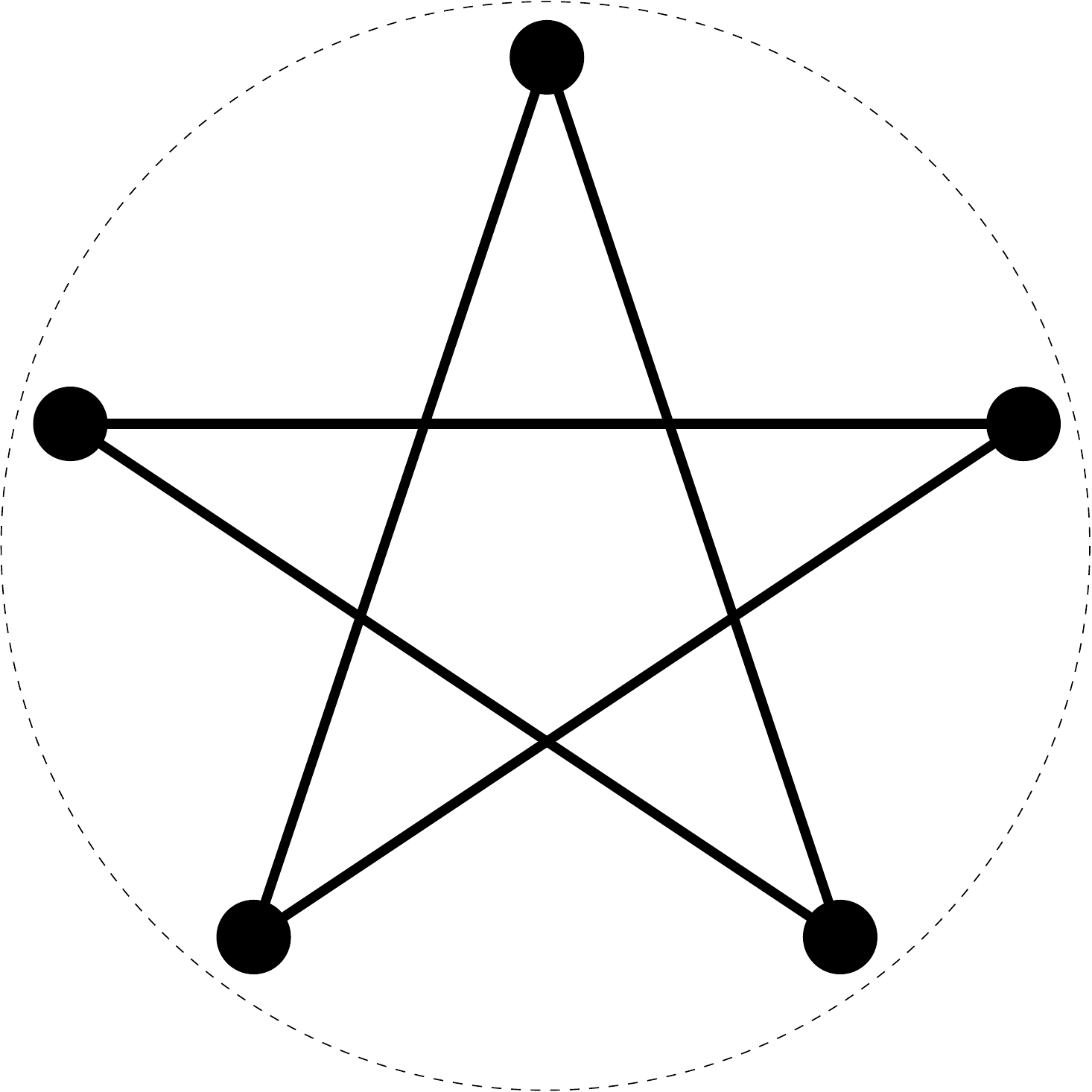}&\getp{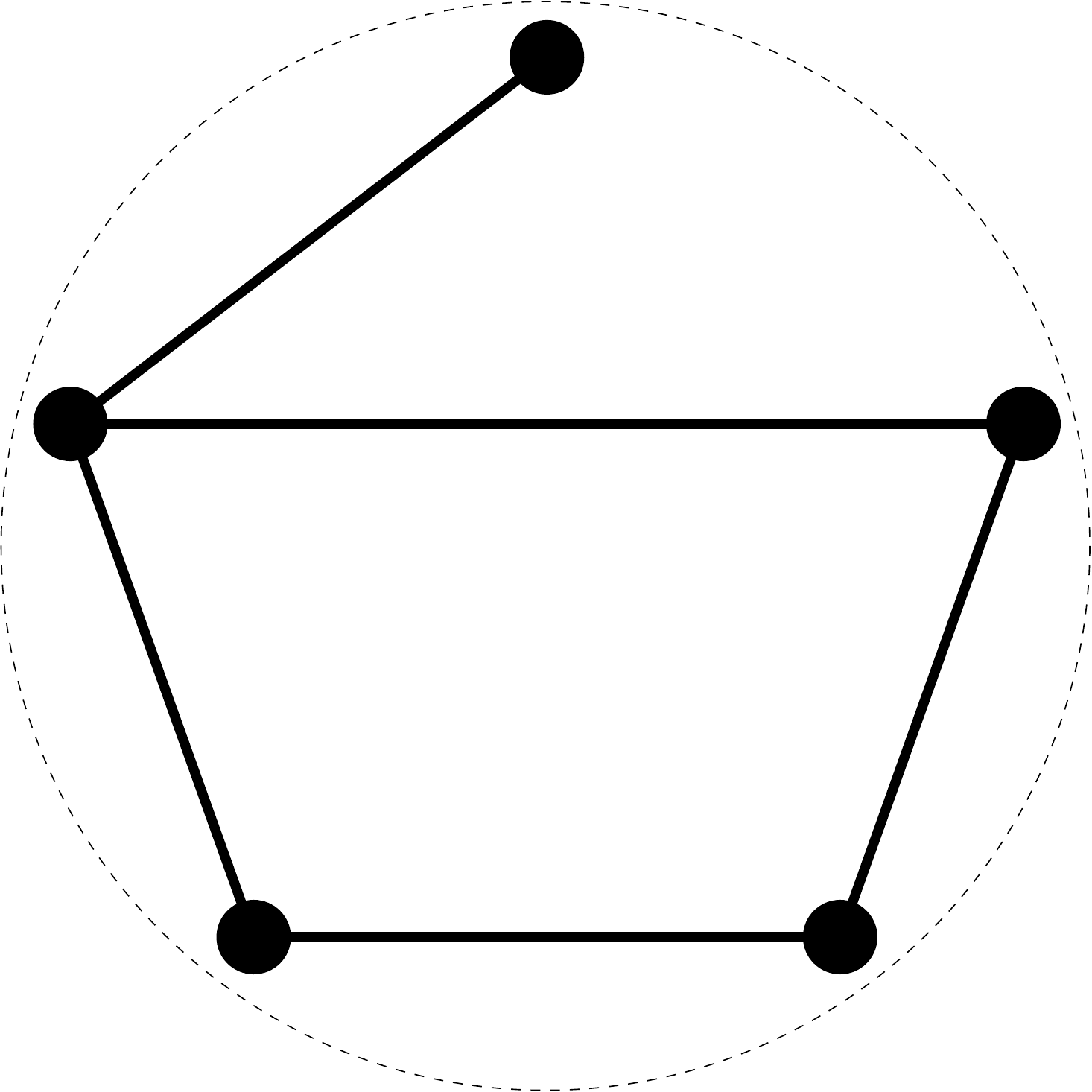}\\
$N_{0}=1$&$N_{0}=1$&$N_{0}=1$&$N_{0}=1$&$N_0=1$&$N_1=1$&$N_1=1$\\\hline
\nsp{\getp{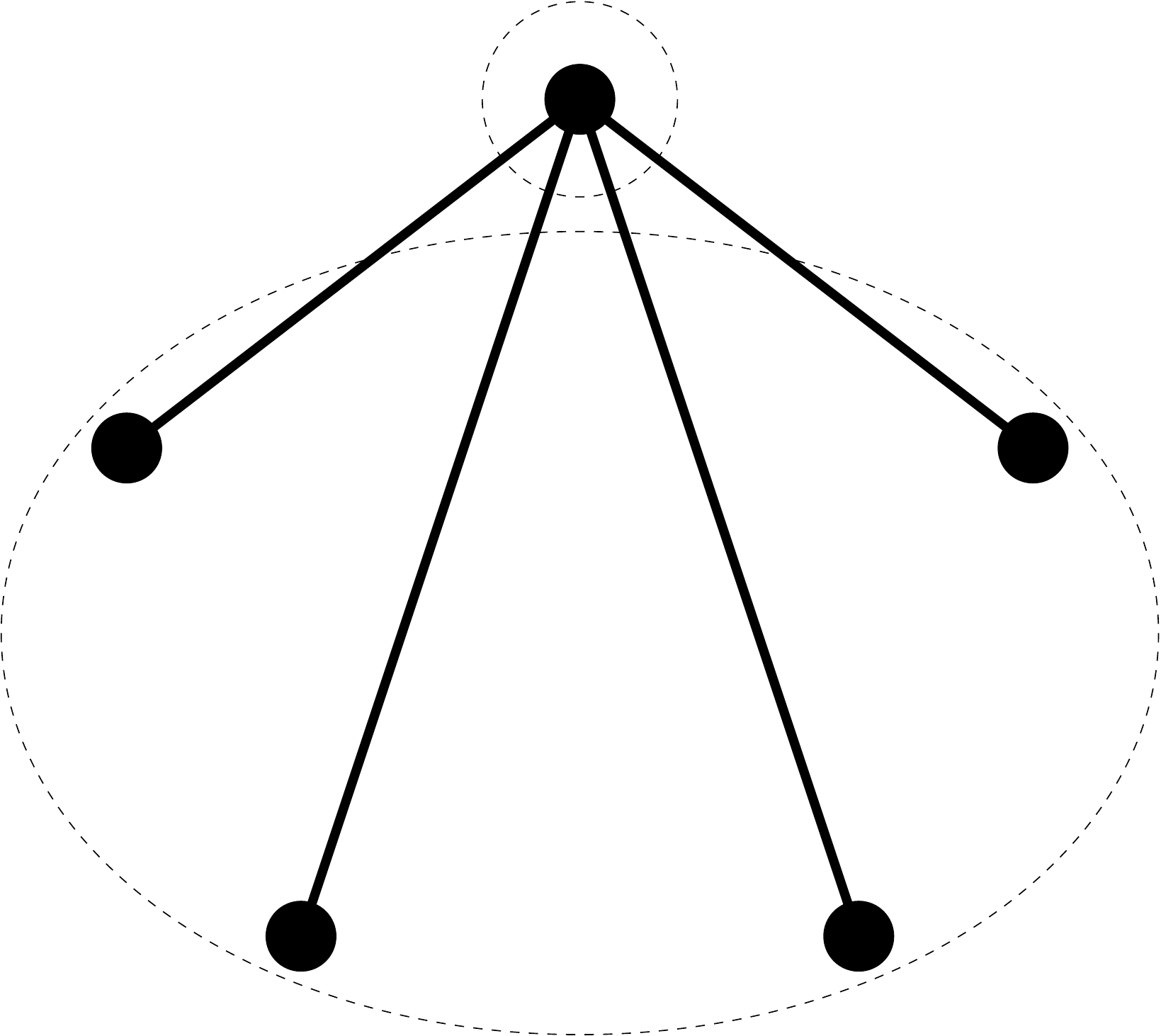}}&\getp{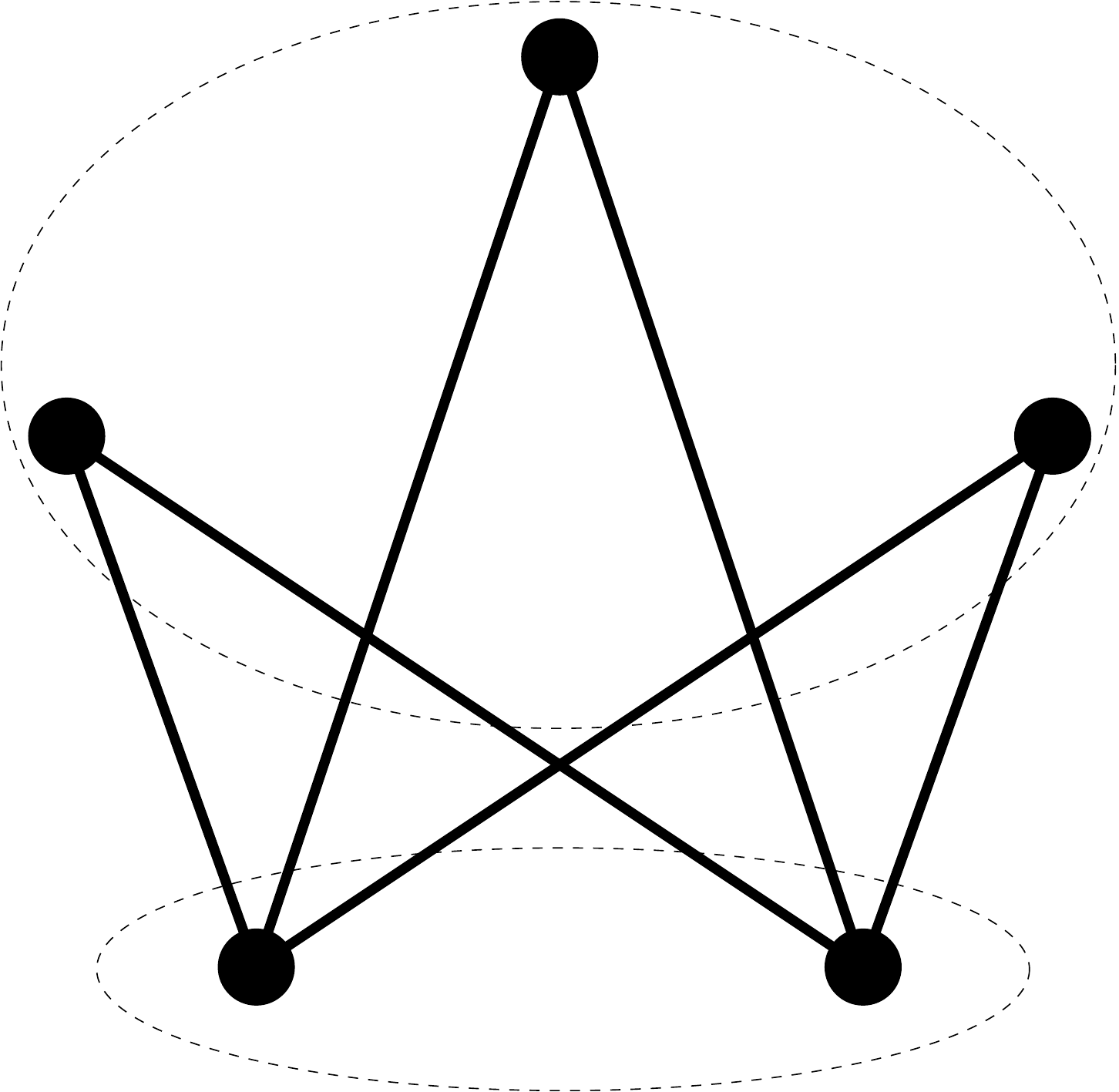}&\nsp{\getp{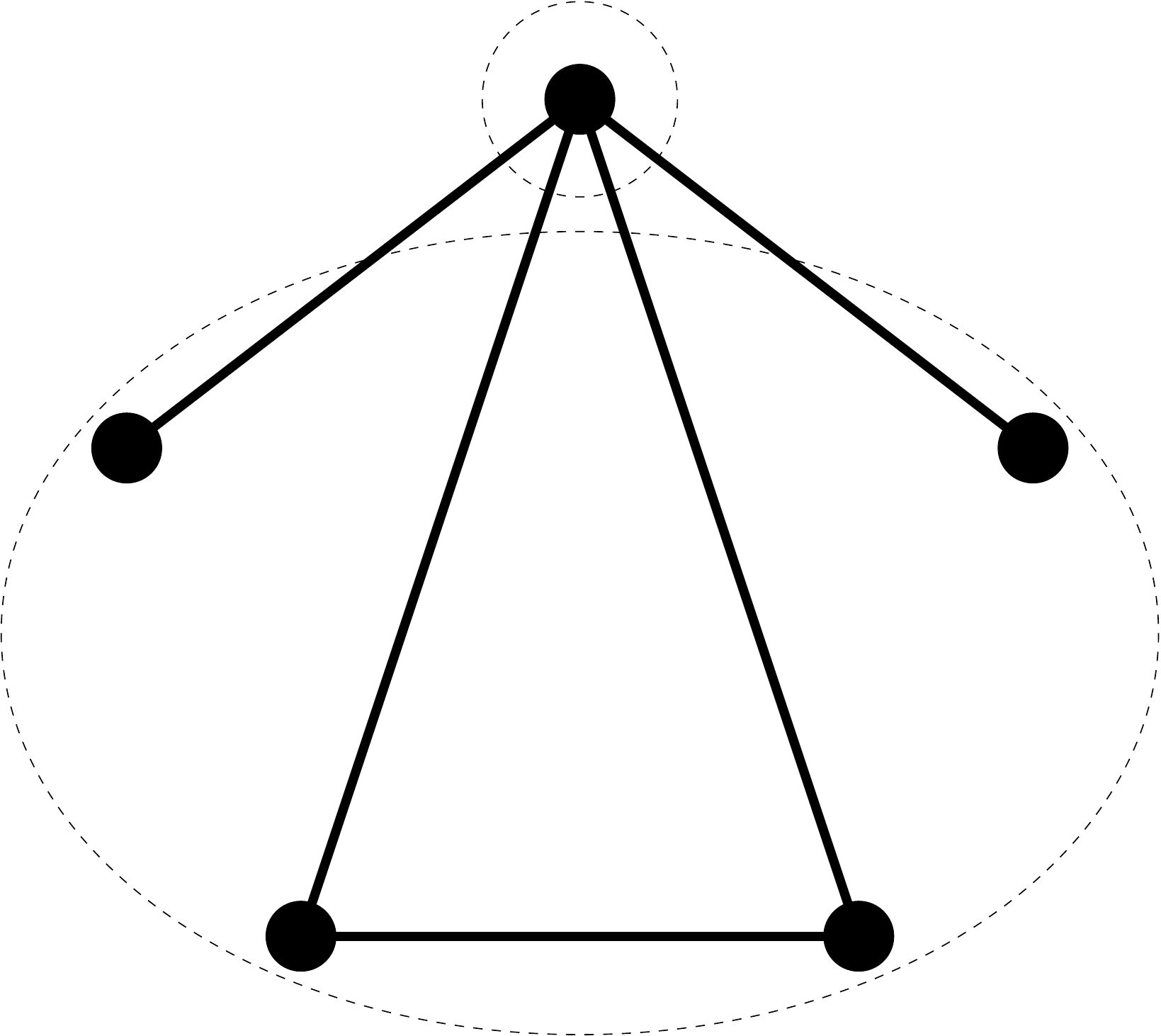}}&\getp{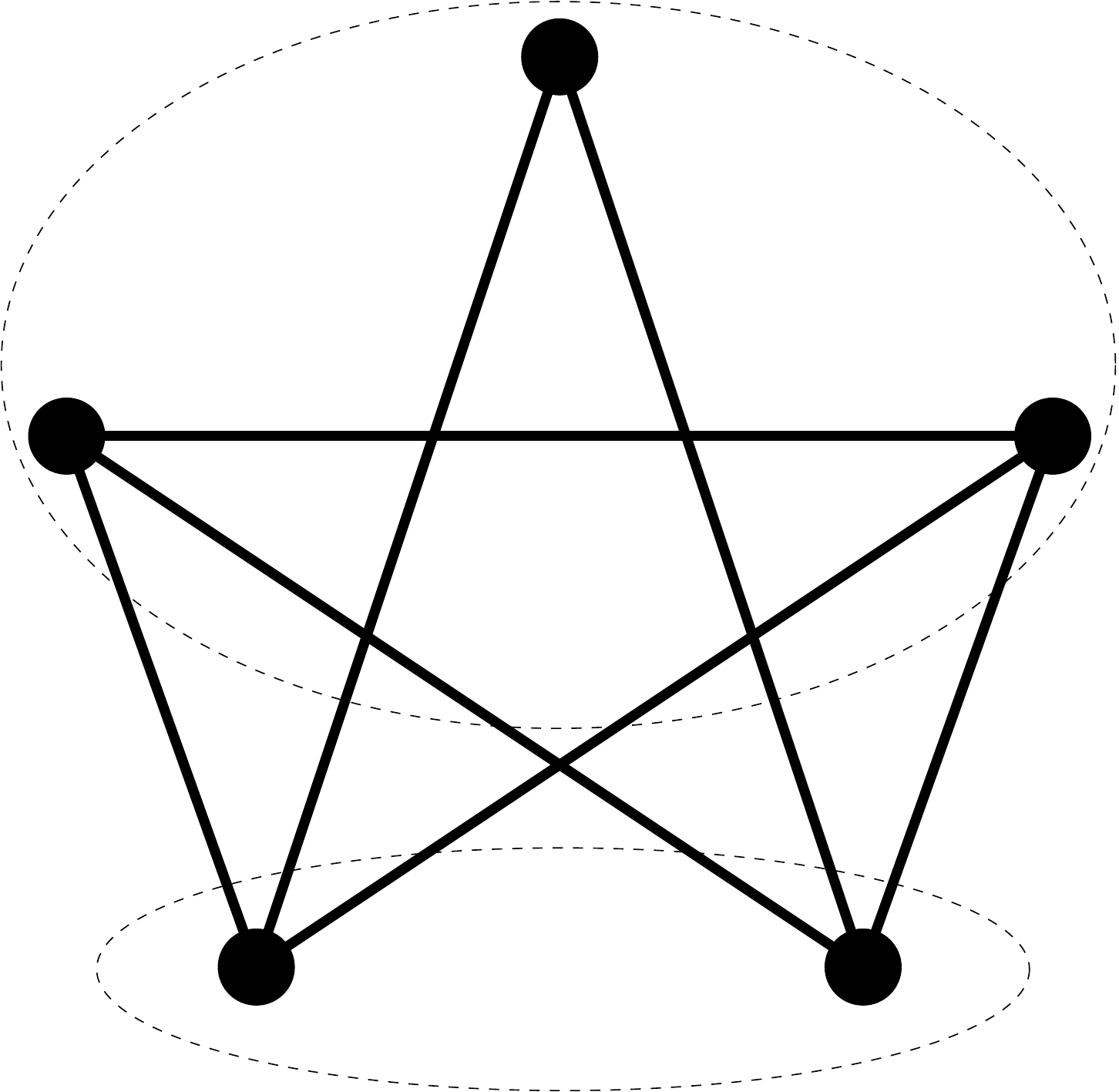}&\nsp{\getp{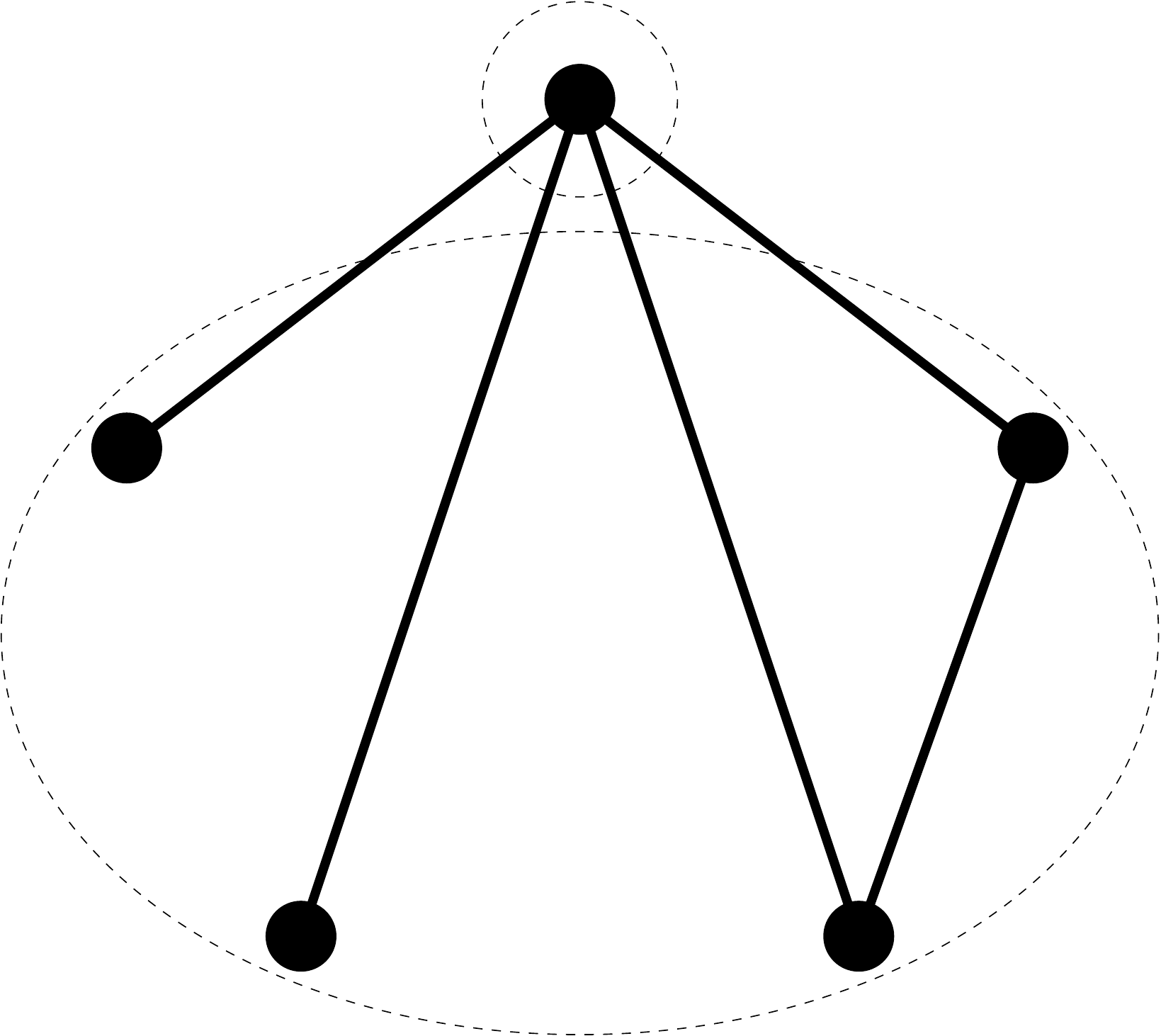}}&\nsp{\getp{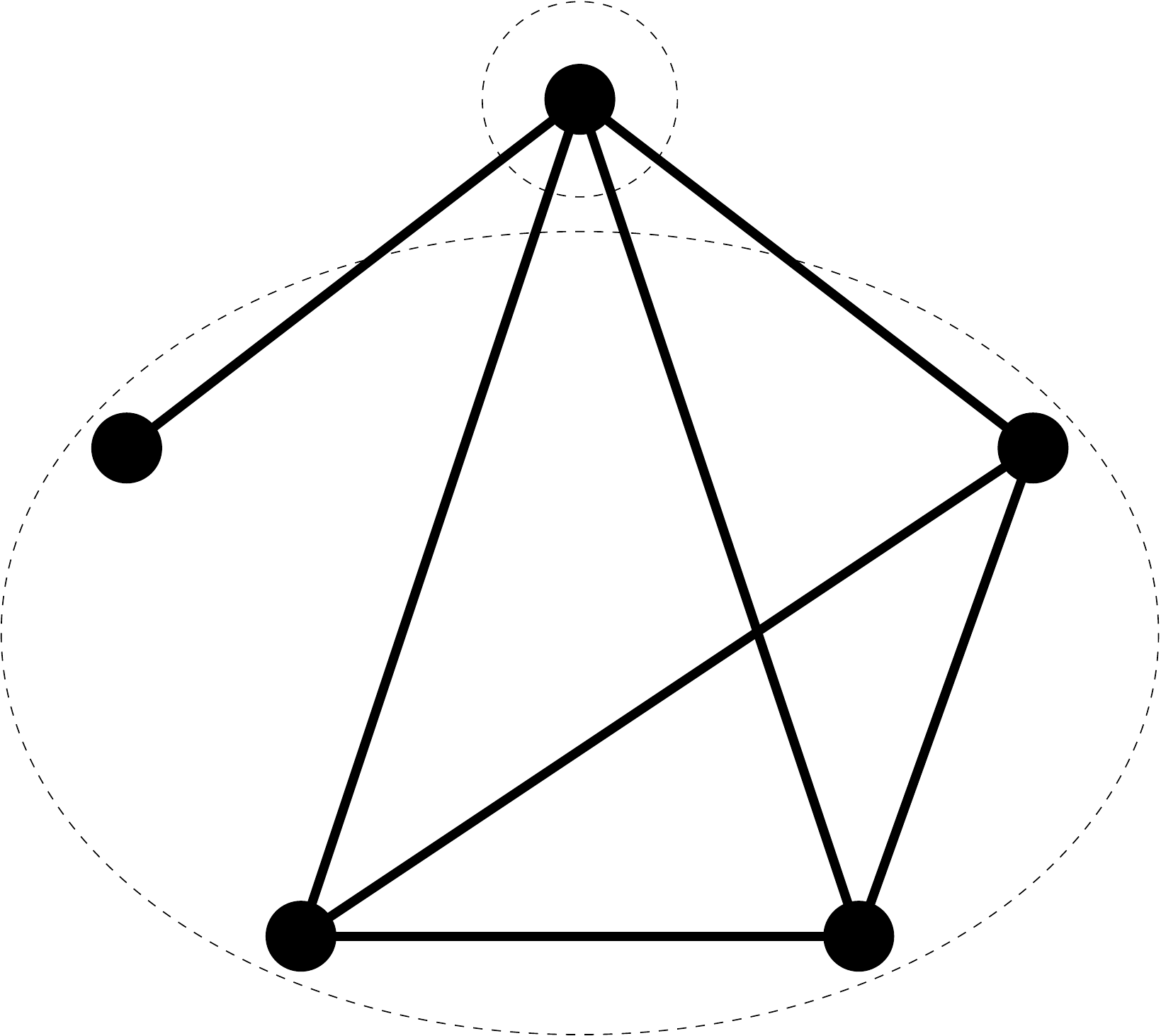}}&\nsp{\getp{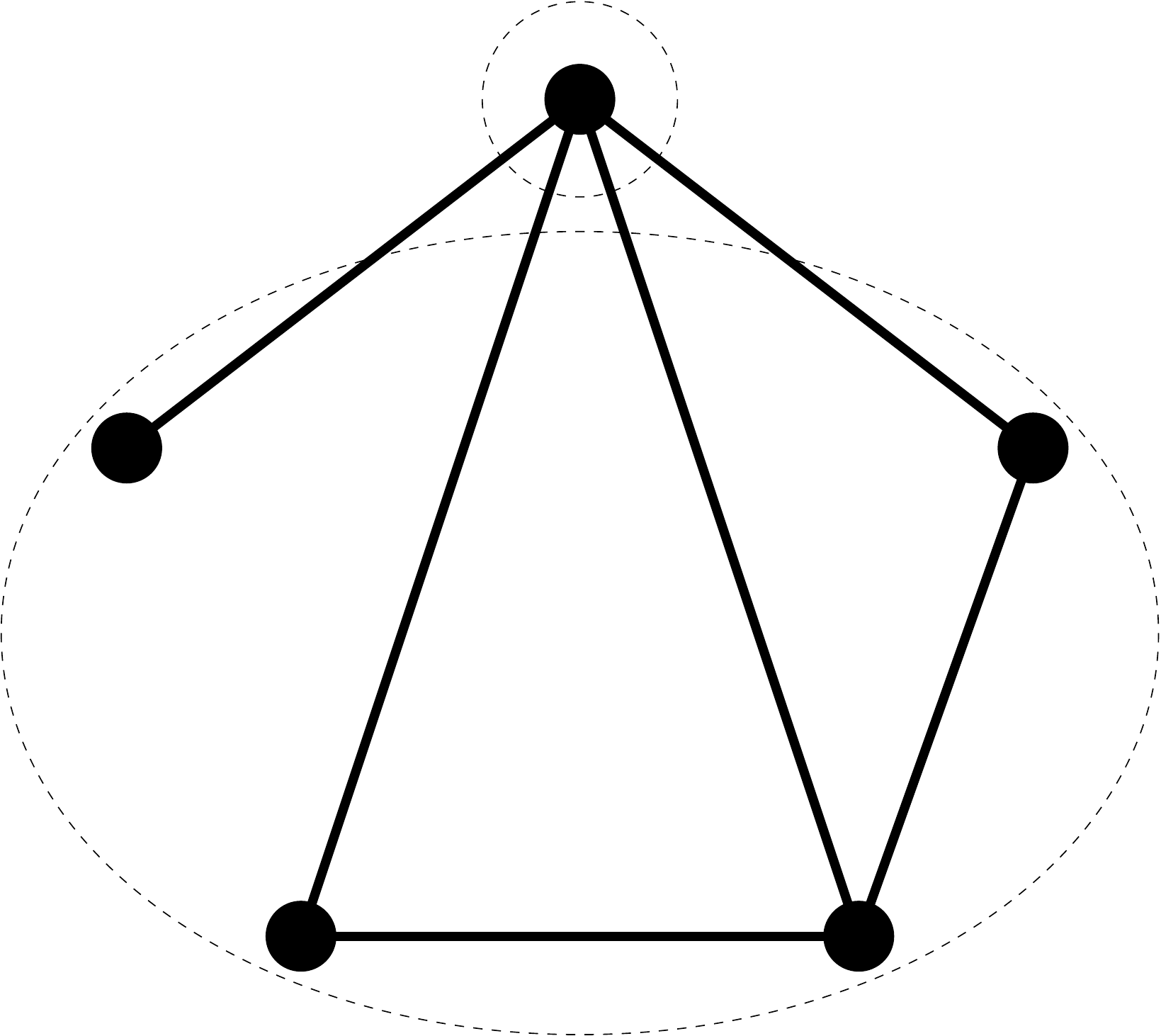}}\\
\nsp{$N_{-3}=1$}&$N_{-2}=1$&$\nsp{N_{-2}=1}$&$N_{-1}=2$&\nsp{$N_{-1}=1$}&\nsp{$N_{-1}=1$}&\nsp{$N_1=1$}\\
\nsp{$t=1$}&$N_{-1}=1$&\nsp{$t=1$}&&\nsp{$t=1$}&\nsp{$t=1$}&\nsp{$t=1$}\\\hline
\nsp{\getp{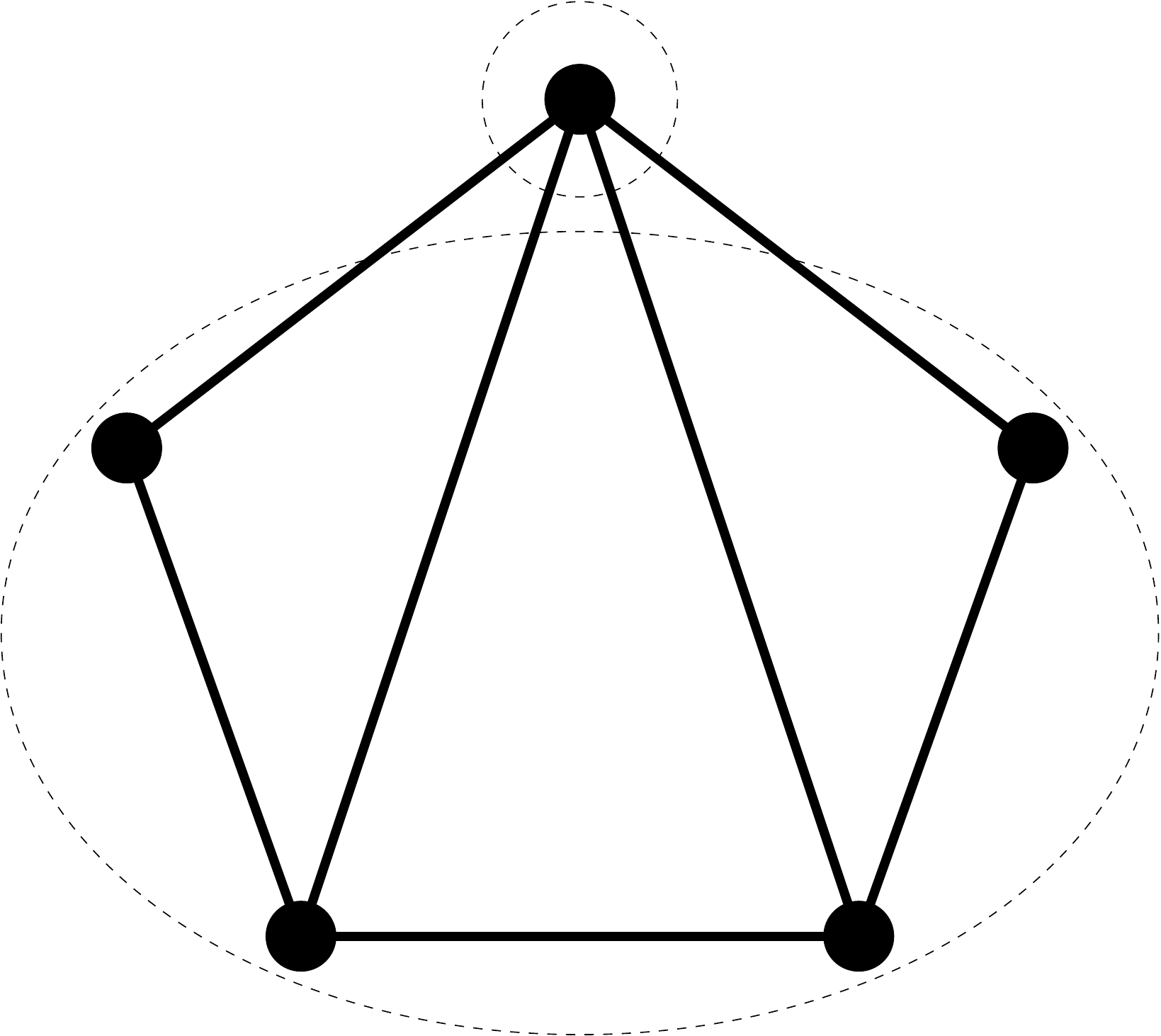}}&\nsp{\getp{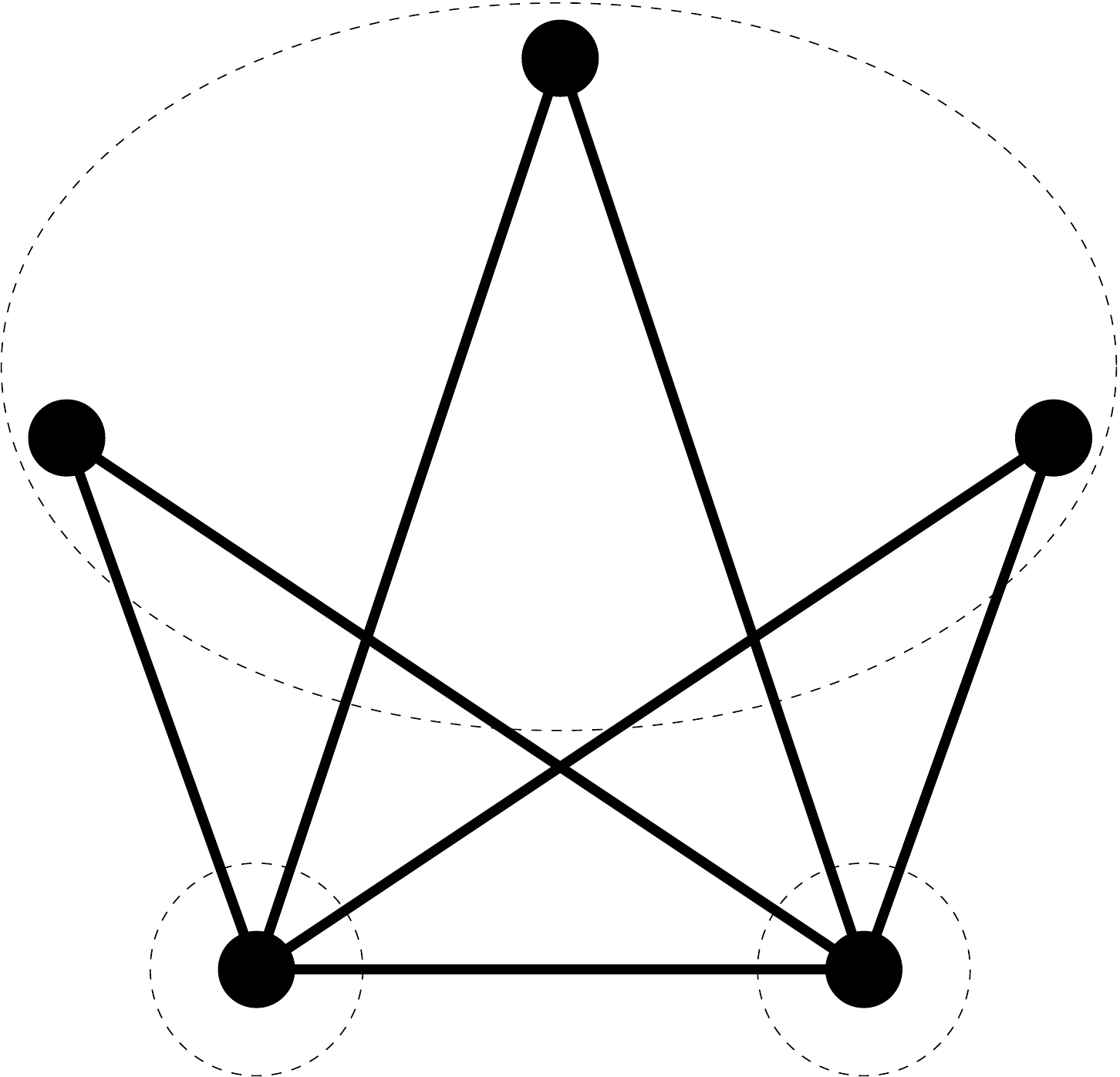}}&\nsp{\getp{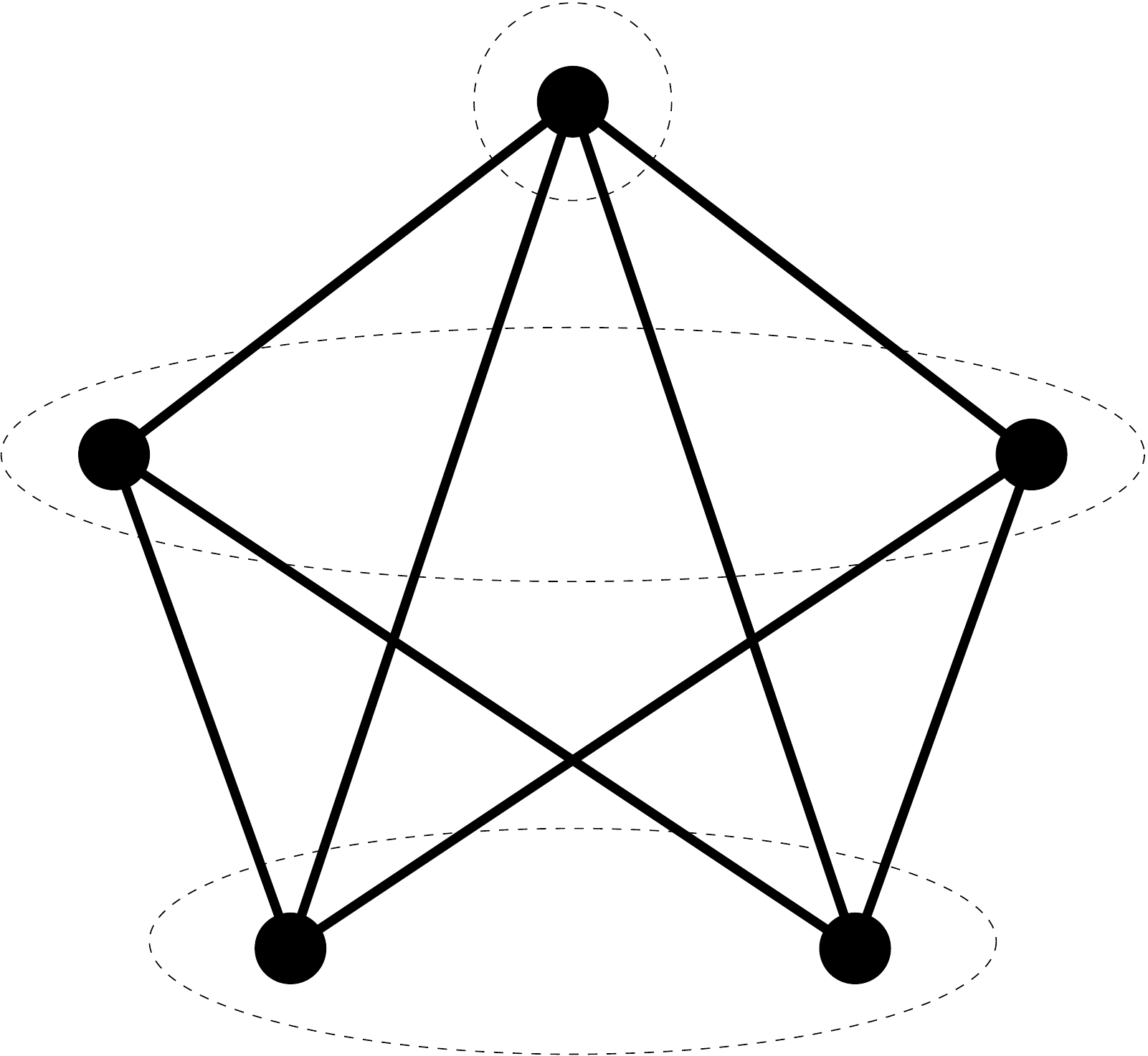}}&\nsp{\getp{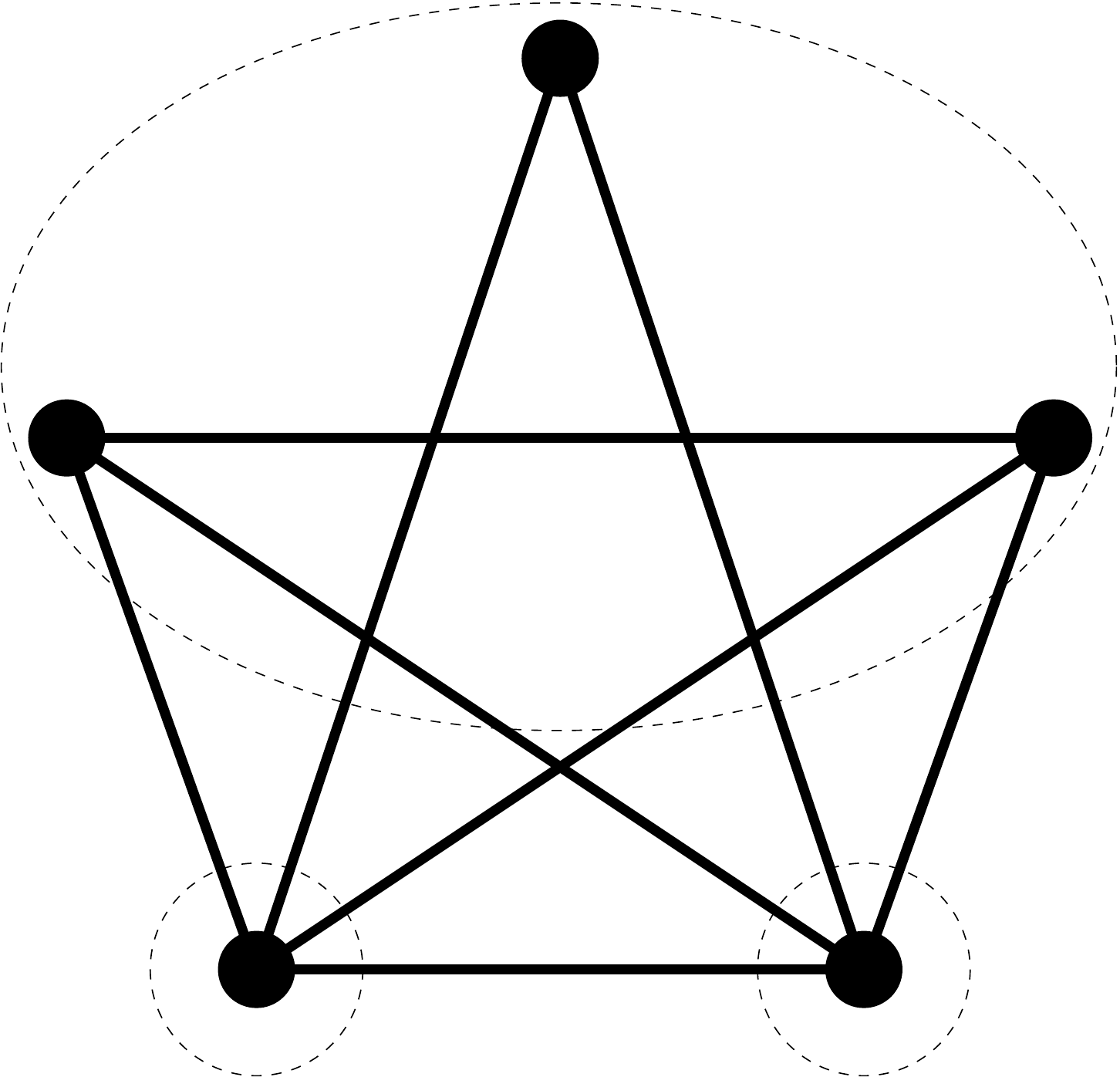}}&\nsp{\getp{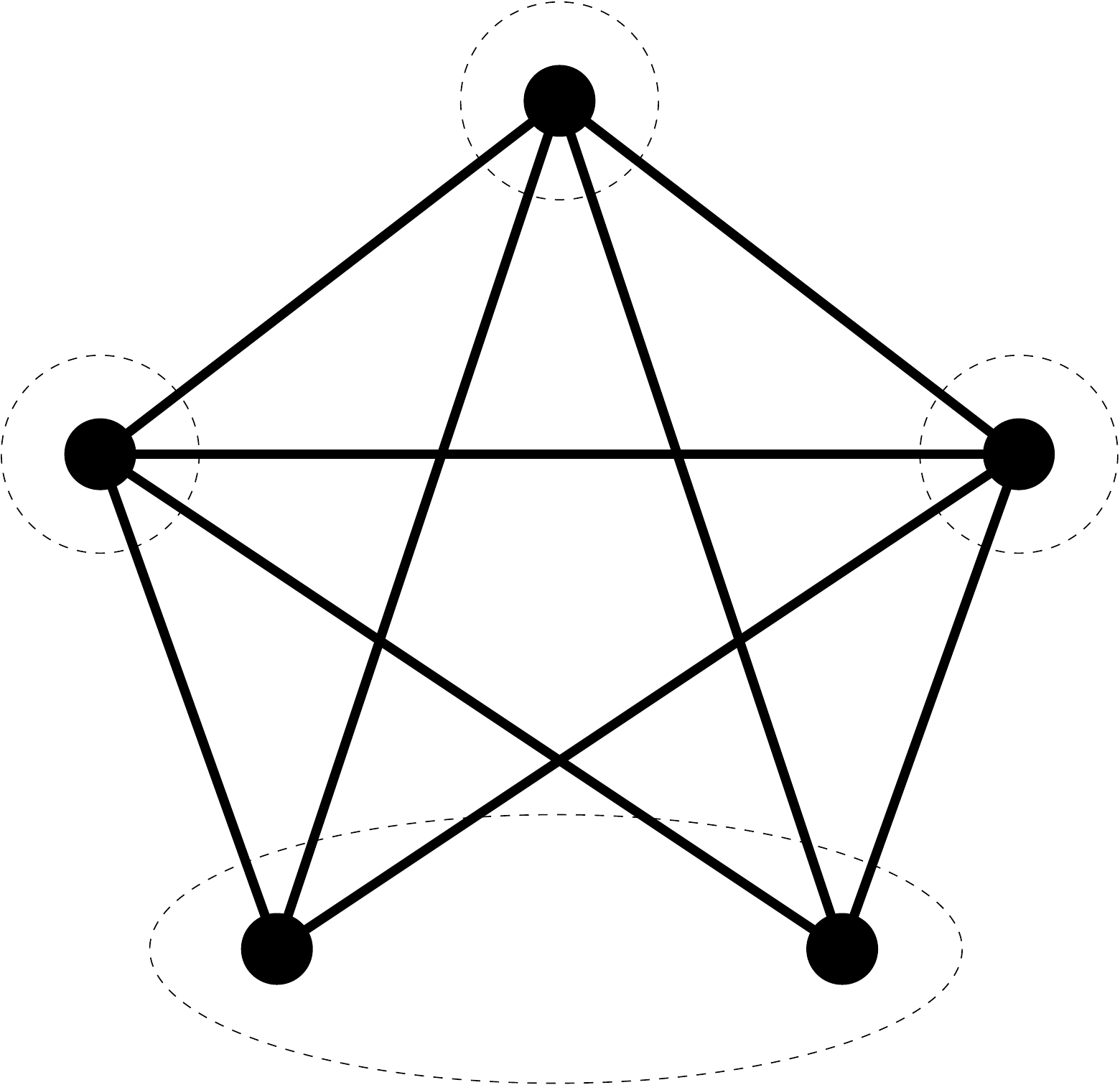}}&\nsp{\getp{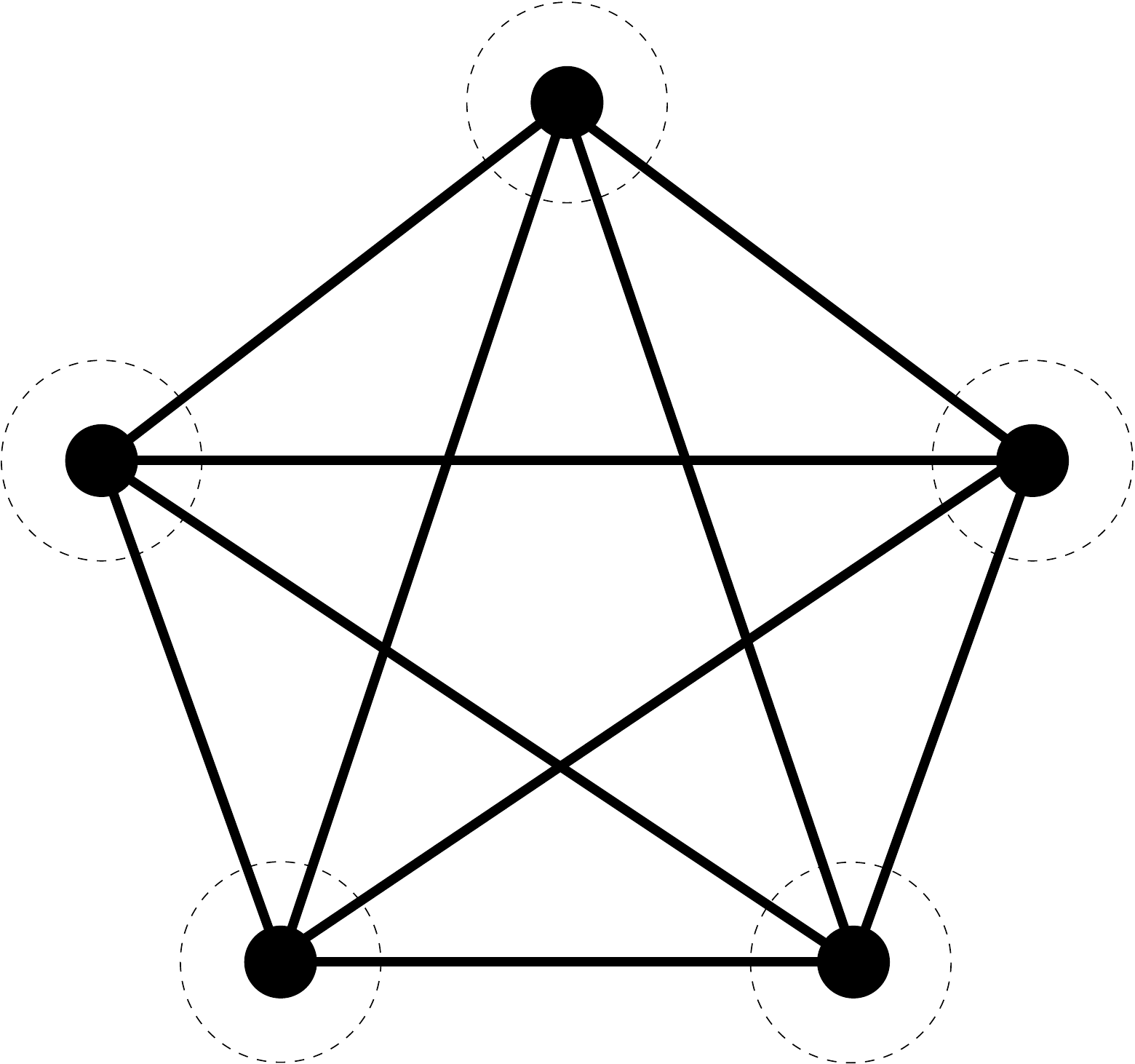}}&\\
\nsp{$N_{0}=1$}&\nsp{$N_{-2}=1$}&\nsp{$N_{-1}=2$}&\nsp{$N_{-1}=1$}&\nsp{$N_{-1}=1$}&\nsp{$t=5$}&\\
\nsp{$t=1$}&\nsp{$t=2$}&\nsp{$t=1$}&\nsp{$t=2$}&\nsp{$t=3$}&\nsp{}&\\\hline
\end{tabular}
\end{center}
\caption{Invariants for all graphs with 5 vertices. Any quantity not mentioned equals zero.}\label{fivecases}
\end{figure}

When the number of vertices increases, it is possible for two graphs to have different sets of invariants, yet define the same C*-algebra. For instance, defining a graph $\Gamma'$ with 10 vertices having its co-irreducible components chosen among those given in Figure \ref{fivecases} so that
\[
N_{-1}(\Gamma')=2
\]
and $\Lambda'$ similarly defined so that
\[
N_1(\Lambda')=2,
\]
then $C^*(A_{\Gamma'}^+)$ and $C^*(A_{\Lambda'}^+)$ will be isomorphic.
Similarly, we may define $\Gamma''$ and $\Lambda''$ with 15 vertices each so that
\begin{gather*}
N_{-1}(\Gamma'')=1,N_0(\Gamma'')=1,N_1(\Gamma'')=1\\
N_{-1}(\Lambda'')=2,N_0(\Lambda'')=1
\end{gather*}
obtaining that $C^*(A_{\Gamma''}^+)\cong C^*(A_{\Lambda''}^+)$.

In the general case of possibly infinite graphs, an additional quantity $o(\Gamma)$ must be introduced to count those co-irreducible components which have an infinite number of vertices, and to address the possibility of having an infinite number of co-irreducible components, but the necessary condition replacing (2) in this general case is not much more complicated than the one given above.

We note that the C*-algebras associated via semigroups to undirected and loop-free graphs are not always graph C*-algebras in the usual sense, not only because graph C*-algebras are defined using directed graphs. We provide a complete description of when $C^*(A_{\Gamma}^+)$ is in fact a graph C*-algebra, and note that there is a rather complicated relation between $\Gamma$ and the graph $G_\Gamma$ when in fact $C^*(A_{\Gamma}^+)\cong C^*(G_\Gamma)$. In this case, $G_\Gamma$ is not unique.

Our results have surprising consequences for the issue of \emph{stable relations} (cf.\ \cite{tal:lsppc}) among sets of isometries of separable Hilbert spaces, subject to commutativity or orthogonality relations as given by the graph $\Gamma$, or, which is nearly the same, for the issue of \emph{semiprojectivity} (cf. \cite{bb:stc}) of the C*-algebras $C^*(A_{\Gamma}^+)$. Indeed, it is easy to see by spectral theory that $C^*(A_{\Gamma}^+)$ is semiprojective when $\Gamma$ is a finite graph with no edges, corresponding to a family of isometries having orthogonal range projections. Similarly, it follows e.g.\ from considering the celebrated Voiculescu matrices (\cite{dv:acfruoca},\cite{retal:iacu}) that when $\Gamma$ is a complete graph with more than one vertex, $C^*(A_{\Gamma}^+)$ cannot have this property.

In fact, it is a question attracting a lot of attention (see e.g. \cite{bb:ssc}) to what extent it is possible to obtain stable relations for commuting sets of stable relations, or to what extent tensor products of semiprojective C*-algebras can themselves be semiprojective, the most intriguing open problem being the case of $\cO_3\otimes \cO_3$. In our setting, because we have found that many settings in which some isometries are required to be orthogonal, and others to commute, give the same C*-algebras as the ones where all are required to be orthogonal, we immediately see that many such settings --- for instance the first 12 listed in Figure \ref{fivecases} --- provide for stable relations.  Involving the notion of graph algebras as outlined above, we will show in Theorem \ref{spnsp} semiprojectivity and nonsemiprojectivity for many of the C*-algebras under study, and it follows from our results that exactly those C*-algebras arising from the graphs in Figure \ref{fivecases} in the non-shaded entries are semiprojective. We have not been able to resolve the issue completely as indeed it is related to the Blackadar conjecture mentioned above, the first open case having six vertices and two co-irreducible components each with Euler characteristic $-2$.

\section{Preliminaries}\label{preliminaries}

\subsection{Semigroup C*-algebras for right-angled Artin monoids}
\label{sgpC-RAAM}

Let $\Gamma$ be a countable graph. $\Gamma = (V,E)$ is given by a countable set of vertices $V$ and a set of edges $E$. We only consider unoriented edges, and given two vertices, there is at most one edge joining these two vertices. In other words, we can think of $E$ as a symmetric subset of $V \times V$ not containing elements of the diagonal.

Given such a graph $\Gamma = (V,E)$, the right-angled Artin group $A_{\Gamma}$ is defined as follows:
\bgloz
  A_{\Gamma} \defeq \spkl{\menge{\sigma_v}{v \in V} \vert \sigma_v \sigma_w = \sigma_w \sigma_v \ \text{if} \ (v,w) \in E}.
\egloz
Similarly, the right-angled Artin monoid $A_{\Gamma}^+$ is defined as follows:
\bgloz
  A_{\Gamma}^+ \defeq \spkl{\menge{\sigma_v}{v \in V} \vert \sigma_v \sigma_w = \sigma_w \sigma_v \ \text{if} \ (v,w) \in E}^+.
\egloz
It turns out that the canonical semigroup homomorphism $A_{\Gamma}^+ \to A_{\Gamma}$ is injective (see \cite{lp:amig}). Moreover, it is shown in \cite{jcml:tarfag} that $A_{\Gamma}^+ \subseteq A_{\Gamma}$ is quasi-lattice ordered. This means that for every $g$ in $A_{\Gamma}$, either $(g A_{\Gamma}^+) \cap A_{\Gamma}^+ = \emptyset$ or there exists $p \in A_{\Gamma}^+$ with $(g A_{\Gamma}^+) \cap A_{\Gamma}^+ = p A_{\Gamma}^+$.

The (left) reduced semigroup C*-algebra of $A_{\Gamma}^+$ is given by
$$
C^*_{\lambda}(A_{\Gamma}^+) = C^*\spkl{\menge{S_v}{v \in V}} \subseteq \cL(\ell^2(A_{\Gamma}^+)),
$$
where $S_v$ is the isometry on $\ell^2(A_{\Gamma}^+)$ acting on the canonical orthonormal basis $\{e_x\}_{x \in A_{\Gamma}^+}$ by $S_v e_x = e_{\sigma_v x}$. The full semigroup C*-algebra of $A_{\Gamma}^+$ is defined as
$$
  C^*(A_{\Gamma}^+) = 
C^*_{\text{univ}}\left\langle\menge{s_v}{v \in V} \left|
  \begin{array}{c}\eckl{s_v,s_w} = \eckl{s_v,s_w^*} = 0 \ \text{if} \ (v,w) \in E \\
  s_v^* s_w  =\delta_{v,w}\ \text{if} \ (v,w) \not\in E \end{array}\right\rangle\right.$$
The canonical homomorphism $C^*(A_{\Gamma}^+) \to C^*_{\lambda}(A_{\Gamma}^+)$ is an isomorphism by \cite{jcml:tarfag}. Hence we do not distinguish between reduced and full versions and simply write $C^*(A_{\Gamma}^+)$ for the semigroup C*-algebra of $A_{\Gamma}^+$.

\subsection{Co-irreducible components}
\label{coirrco}

The graph $\Gamma$ is called co-reducible if there exist non-empty subsets $V_1$ and $V_2$ of $V$ with $V = V_1 \sqcup V_2$ such that $V_1 \times V_2 \subseteq E$. $\Gamma$ is called co-irreducible if $\Gamma$ is not co-reducible. In general, we can always decompose $\Gamma$ into co-irreducible components. This means that there exist co-irreducible graphs $\Gamma_i = (V_i,E_i)$ such that $A_{\Gamma}^+ \cong \bigoplus_i A_{\Gamma_i}^+$ (and also $A_{\Gamma} \cong \bigoplus_i A_{\Gamma_i}$). As explained in \cite{jcml:bqitaag}, these co-irreducible components are found by looking at the opposite graph of $\Gamma$. For the semigroup C*-algebra, we get $C^*(A_{\Gamma}^+) \cong \bigotimes_i C^*(A_{\Gamma_i}^+)$. Note that if there are (necessarily countably) infinitely many co-irreducible components, the tensor product is defined as an inductive limit of finite tensor products with respect to the canonical unital embeddings as tensor factors.

It is shown in \cite{jcml:bqitaag} that for a co-irreducible graph $\Gamma = (V,E)$ with $1 < \abs{V} < \infty$, $C^*(A_{\Gamma}^+)$ has a unique non-trivial ideal isomorphic to the compact operators. It is easy to see the compact operators in the description of $C^*(A_{\Gamma}^+)$ as a concrete C*-algebra on $\ell^2(A_{\Gamma}^+)$: We just have to observe that $1 - \bigvee_{v \in V} S_v S_v^*$ is the orthogonal projection onto the one-dimensional subspace of $\ell^2(A_{\Gamma}^+)$ corresponding to the identity element of $A_{\Gamma}^+$. This projection then generates the ideal of compact operators. The corresponding quotient $C^*_Q (A_{\Gamma}^+)$ is a (unital) Kirchberg algebra satisfying the UCT. However, if our co-irreducible graph has infinitely many vertices, then $C^*(A_{\Gamma}^+)$ itself is a (unital) Kirchberg algebra satisfying the UCT. That we obtain UCT Kirchberg algebras follows also from \cite[Corollary~7.23]{xl:nscca}. The case where $\Gamma$ consists of only one vertex is easy to understand; in that case, $C^*(A_{\Gamma}^+)$ is canonically isomorphic to the Toeplitz algebra $\cT$.

\subsection{Primitive ideal space}
\label{prim}

We can now describe the primitive ideal space of $A_{\Gamma}^+$ for arbitrary $\Gamma$. Let $\Gamma_i = (V_i,E_i)$ be the co-irreducible components of $\Gamma$. Then by \cite[Theorem~4.9]{bb:itpc}, we have an identification
$$
  \Prim(C^*(A_{\Gamma}^+)) \cong \prod_i \Prim(C^*(A_{\Gamma_i}^+)).
$$
Under this identification, an element $(I_i)$ of the space on the right hand side corresponds to the primitive ideal $I$ of $C^*(A_{\Gamma}^+)$ which is generated by $\{\bigotimes_j J_{ij}\}_i$, where $J_{ij} = C^*(A_{\Gamma_j}^+)$ if $j \neq i$ and $J_{ii} = I_i$. Since each of the $\Gamma_i$ is co-irreducible, the primitive ideal space $\Prim(C^*(A_{\Gamma_i}^+))$ is easy to describe because of the results summarized above:
\begin{itemize}
\item If $\Gamma_i$ just consists of one point, then $\Prim(C^*(A_{\Gamma_i}^+))$ is homeomorphic to the primitive ideal space of the Toeplitz algebra. This means that as a set, $\Prim(C^*(A_{\Gamma_i}^+))$ is the disjoint union of a point and a circle, and the non-empty open sets are given by unions of the point and open subsets in the usual topology of the circle.
\item If $\Gamma_i$ has more than one, but finitely many vertices, $\Prim(C^*(A_{\Gamma_i}^+))$ consists of two points, one of which is closed (the corresponding primitive ideal is the ideal of compact operators) and the other one is dense.
\item If $\Gamma_i$ has infinitely many vertices, then $\Prim(C^*(A_{\Gamma_i}^+))$ consists of only one point.
\end{itemize}

\subsection{K-theory}
\label{K}

K-theory for $C^*(A_{\Gamma}^+)$ and the quotients $C^*_Q (A_{\Gamma}^+)$ has been computed in \cite{ni:ktcrag} in an ad hoc way, and can also be computed using \cite{jcsexl:kcpasa}. Let us explain the computation via the latter route. First of all, we need the Euler characteristic of a graph $\Gamma$. We view $\Gamma$ as a simplicial complex by defining for every $n = 0, 1, 2, \dotsc$ the set of $n$-simplices by
$$
  K_n \defeq \menge{\gekl{v_0, \dotsc, v_n} \subseteq V}{(v_i,v_j) \in E \ {\rm for} \ {\rm all} \ i, \, j \in \gekl{0, \dotsc, n}, \, i \neq j}.
$$ 
Then we set for a graph $\Gamma$ with finitely many vertices $\chi(\Gamma) \defeq 1 - \sum_{n=0}^{\infty} (-1)^n \abs{K_n}$. 

\begin{remark}\label{anyZ}
It is easy to see that there are co-irreducible graphs attaining any integer as its Euler characteristic. Indeed, letting $\Gamma_{-m}$ denote the graph with $m+1$ vertices and no edges, we clearly have 
\[
\chi(\Gamma_{-m})=-m.
\] 
Systematically generating positive characteristics is harder; one option is to let $\Gamma_{n^2-1}$ denote the graph with $2n+2$ vertices obtained by deleting one edge from the complete bipartite graph $K_{n+1,n+1}$ and note that
\[
\chi(\Gamma_{n^2-1})=n^2-1
\]
To obtain positive characteristics in $\{(n-1)^2,\dots, n^2-1\}$ one may simply add a suitable number of isolated vertices to $\Gamma_{n^2-1}$. 
\end{remark}

Now, by \cite[Theorem~5.2]{jcsexl:kcpasa}, we know that we always have $K_*(C^*(A_{\Gamma}^+)) \cong K_*(\Cz)$, and $K_0(C^*(A_{\Gamma}^+)) \cong \Zz$ is generated by the class of the unit $[1]$. Here we use that right-angled Artin groups satisfy the Baum-Connes conjecture with coefficients because these groups have the Haagerup property (see \cite{gnlr:gaccc}, and also \cite{yadd:hpsgp}). To compute K-theory for the quotient $C^*_Q (A_{\Gamma}^+)$ in the case that $\Gamma$ has (more than one and) finitely many vertices, we consider the short exact sequence $0 \to \cK \to C^*(A_{\Gamma}^+) \to C^*_Q (A_{\Gamma}^+) \to 0$ and its six term exact sequence in K-theory:
\begin{align}
\label{6tes1}
\scalebox{1}{$\vcenter{\xymatrix{
  K_0(\cK) \ar[r] & K_0(C^*(A_{\Gamma}^+)) \ar[r] & K_0(C^*_Q(A_{\Gamma}^+)) \ar[d] \\
  K_1(C^*_Q(A_{\Gamma}^+)) \ar[u] & \ar[l] K_1(C^*(A_{\Gamma}^+)) & \ar[l] K_1(\cK)
  }}$}
\end{align}
Since both $K_1(\cK)$ and $K_1(C^*(A_{\Gamma}^+))$ vanish, all we have to do is to compute the homomorphism $K_0(\cK) \cong \Zz \to \Zz \cong K_0(C^*(A_{\Gamma}^+))$. $K_0(\cK) \cong \Zz$ is generated by the class of any minimal projection. So we can take $e = 1 - \bigvee_{v \in V} S_v S_v^*$. It is easy to see that in $K_0$, $[e] \in K_0(\cK)$ is sent to $\chi(\Gamma) [1] \in K_0(C^*(A_{\Gamma}^+))$. Therefore, by exactness of \eqref{6tes1}, we conclude that $K_0(C^*_Q (A_{\Gamma}^+)) \cong \Zz / \abs{\chi(\Gamma)} \Zz$ and
\[
K_1(C^*_Q (A_{\Gamma}^+)) \cong
\bfa
  \gekl{0} \ \text{if} \ \chi(\Gamma) \neq 0 \\
  \Zz \ \text{if} \ \chi(\Gamma) = 0
\efa.
\]

\section{Extension algebras}\label{coirr}

We now discuss the C*-algebras associated to co-irreducible graphs and see how they are all isomorphic to either the Toeplitz algebra, the Cuntz algebra $\cO_\infty$, or an \emph{extension algebra} as specified below.

\begin{theorem}\label{basicunique}
Consider the following properties for a unital C*-algebra $A$:
\begin{enumerate}[(1)]
\item $A$ contains $\cK$ as an ideal, and $A/\cK$ is a Kirchberg algebra satisfying the UCT,
\item $K_0(A)=\Zz$ with $[1_A]=1$.
\end{enumerate}
For each $k\in \Zz\backslash\{0\}$ there is a unique C*-algebra satisfying (1), (2) and
\begin{enumerate}[(1)]\addtocounter{enumi}{2}
\item The six-term exact sequence for $\cK$ and $A$ is given by
\[
\xymatrix{
\Zz \ar[r]^{k} & \Zz \ar[r] & \Zz/k\Zz \ar[d] \\
  0 \ar[u] & \ar[l] {0} & \ar[l] {0}}
\]
\end{enumerate}
There is also a unique 
 C*-algebra satisfying (1), (2) and
\begin{enumerate}[(1')]\addtocounter{enumi}{2}
\item The six-term exact sequence for $\cK$ and $A$ is given by
\[
\xymatrix{
\Zz \ar[r]^{0} & \Zz \ar@{=}[r] & \Zz\ar[d] \\
  \Zz \ar@{=}[u] & \ar[l] {0} & \ar[l] {0}}
\]
\end{enumerate}
\end{theorem}
\begin{proof}
Note that $\cK$ is an essential ideal of $A$ (i.e., every nonzero ideal of $A$ has a nontrivial intersection with $\cK$) since $A$ is unital and $A / \cK$ is simple.  Uniqueness follows from \cite[Corollary 4.22]{segrer:scecc}. For existence, we note that when $\Gamma$ is a finite and co-irreducible graph with $|\Gamma|>1$ and $\chi(\Gamma)=k$, all properties are met as noted in Section \ref{preliminaries}.
\end{proof}

When specifying the map $K_0(\cK)\to K_0(A)$ above we let the unit of the leftmost copy of $\Zz$ denote the class of a minimal projection of $\Zz$.

\begin{definition}\label{first}
The unique C*-algebras satisfying (1),(2) and (3) are denoted $E_{|k|+1}^{\sgn (k)}$. The unique 
C*-algebra satisfying (1),(2) and (3') is denoted $E_{1}^{0}$. The quotient $E_1^0/\cK$ is denoted $\cO_1$.
\end{definition}

Our notation has been chosen to fit the notation  $E_n^k$ for the {extension algebras} of $\cO_n$ studied in \cite{xfsl:eaca}. With our name $\cO_1$ for the appropriately chosen Kirchberg algebra, we have
\[
\xymatrix{
0\ar[r]&\cK\ar[r]^-{\iota}& E^k_n\ar[r]^-{\pi}&\cO_n\ar[r]&0}
\]
for any $k\in\{-1,0,1\}$ and $n\in\Nz$, provided $k=0$ precisely when $n=1$.

\begin{lemma}\label{notconfuse}
$E^k_n\cong E^{k'}_{n'}$ only when $n=n'$ and $k=k'$. 
$E^k_n\otimes\cK \cong E^{k'}_{n'}\otimes \cK$ precisely when $n=n'$.
\end{lemma}
\begin{proof}
Since the six-term exact sequences are as specified in (3) or (3') of Theorem \ref{basicunique}, stable isomorphism can only occur if $n=n'$, and hence we only need to check that for $n>0$, we have
$E^1_n\not \cong E^{-1}_{n}$, yet $E^1_n\otimes\cK\cong E^{-1}_{n}\otimes\cK$.

We note that the only two options for an isomorphism among the six-term exact sequences in this case are given as
\[
\xymatrix{
{\Zz}\ar^-{n}[r]\ar[d]^-{\pm1}&{\Zz}\ar[r]\ar[d]^-{\mp1}&\Zz/n\Zz\ar[d]\\
{\Zz}\ar_-{-n}[r]&{\Zz}\ar[r]&\Zz/n\Zz}
\]
and that we must choose $+1$ as the left most isomorphism to preserve the positive cone of $K_0(\cK)$. Thus, an isomorphism is ruled out as it would fail to send the class of the unit of $E^1_n$ to the unit of $E^{-1}_n$, but an isomorphism after stabilization is guaranteed by, e.g., \cite{segrer:cecc}.
\end{proof}

The following result follows directly from \S~\ref{coirrco}, \S~\ref{K} and Theorem~\ref{basicunique}.
\begin{theorem}\label{list}
When $\Gamma$ is a co-irreducible graph, $C^*(A_{\Gamma}^+)$ is one of the C*-algebras
$
\cT,E^k_n,\cO_\infty
$
according to
\begin{enumerate}
\item If $|\Gamma|=1$, $C^*(A_{\Gamma}^+)\cong \cT$
\item If $1<|\Gamma|<\infty$, $C^*(A_{\Gamma}^+)\cong E_{1+|\chi(\Gamma)|}^{\sgn\chi(\Gamma)}$
\item If $|\Gamma|=\infty$, $C^*(A_{\Gamma}^+)\cong \cO_\infty$
\end{enumerate}
\end{theorem}

We note that by the information already noted on the ideal structures in combination with Lemma \ref{notconfuse}, the C*-algebras appearing are not mutually isomorphic, and hence we have a complete classification by the cardinality of $\Gamma$ and the Euler characteristic in the co-irreducible case.

 In preparation for the general case we now study isomorphisms between various tensor products amongst the relevant extension algebras and some of their quotients. For this, we will need:
 
\setlength{\parskip}{0cm}
\btheo
\label{class_pi-orthideals}
Let $A_i$, $i=1,2$, be unital C*-algebras whose proper ideals are precisely given by $(0)$, $I_i$, $J_i$ and $I_i \oplus J_i$. We assume that $I_i$ and $J_i$ are UCT Kirchberg algebras, and the quotients $A_i / (I_i \oplus J_i)$ are also UCT Kirchberg algebras.

Let $\alpha_I: \ K_*(I_1) \cong K_*(I_2)$, $\alpha_J: \ K_*(J_1) \cong K_*(J_2)$, $\alpha_{I \oplus J}: \ K_*(I_1 \oplus J_1) \cong K_*(I_2 \oplus J_2)$, $\beta: \ K_*(A_1) \cong K_*(A_2)$, $\gamma_I: \ K_*(A / I_1) \cong K_*(A / I_2)$, $\gamma_J: \ K_*(A / J_1) \cong K_*(A / J_2)$, and $\gamma_{I \oplus J}: \ K_*(A_1 / (I_1 \oplus J_1)) \cong K_*(A_2 / (I_2 \oplus J_2))$ be isomorphisms, with $\beta$ preserving the $K_0$-classes of the units and $\alpha_{I \oplus J} = \alpha_I \oplus \alpha_J$ (under the canonical isomorphism $K_*(I_i \oplus J_i) \cong K_*(I_i) \oplus K_*(J_i)$).

Furthermore, we assume that these isomorphism are compatible with the K-theoretic six term exact sequences attached to
\begin{align*}
0 \to I_i \to A_i \to A_i / I_i \to 0, \qquad & 0 \to J_i \to A_i \to A_i / J_i \to 0, \\
0\to I_i \oplus J_i \to A_i \to A_i / (I_i \oplus J_i) \to 0,\qquad  & 0 \to J_i \to A_i / I_i \to A_i / (I_i \oplus J_i) \to 0 \\
\text{and} \qquad & 
\end{align*} 
\[
0 \to I_i \to A_i / J_i \to A_i / (I_i \oplus J_i) \to 0.
\]

Then there exists an isomorphism $\varphi: \ A_1 \cong A_2$ which induces $\alpha_I$, $\alpha_J$, $\alpha_{I \oplus J}$, $\beta$, $\gamma_I$, $\gamma_J$ and $\gamma_{I \oplus J}$ in K-theory. 
\etheo
\setlength{\parskip}{0.5cm}

\begin{proof}
Combine \cite[Folgerung~4.3]{ek:nkmkna} and \cite[Theorem~1.3]{rbmk:uctcfts} with \cite[Theorem~2.1]{grer:rccconiII} or \cite[Theorem~3.3]{segrer:scecc}.
\end{proof}

\blemma
\label{OinfxE}
For every $n \geq 2$, we have $\cO_{\infty} \otimes E_n^{+1} \cong \cO_{\infty} \otimes E_n^{-1}$.
\elemma
\bproof
Both $\cO_{\infty} \otimes E_n^{+1}$ and $\cO_{\infty} \otimes E_n^{-1}$ are unital C*-algebras with unique ideal isomorphic to $\cO_{\infty} \otimes \cK$ and corresponding quotient isomorphic to $\cO_{\infty} \otimes \cO_n \cong \cO_n$. The K-theoretic six term exact sequences for $0 \to \cO_{\infty} \otimes \cK \to \cO_{\infty} \otimes E_n^{+1} \to \cO_n \to 0$ and $0 \to \cO_{\infty} \otimes \cK \to \cO_{\infty} \otimes E_n^{-1} \to \cO_n \to 0$ look as follows:
\bgloz
\xymatrix{
  \Zz \ar[r] & \Zz \ar[r] & \Zz / (n-1) \Zz \ar[d] \\
  {0} \ar[u] & \ar[l] {0} & \ar[l] {0}
  }
\egloz
where $\Zz \cong K_0(\cO_{\infty} \otimes \cK)$ is generated by $[1 \otimes e]$ for a minimal projection $e \in \cK$ and the unit $1$ of $\cO_{\infty}$, and $\Zz \cong K_0(\cO_{\infty} \otimes E_n^{\pm 1})$ is generated by the class of the unit. The only difference is that for $E_n^{+1}$, the homomorphism $\Zz \to \Zz$ is given by $z [1 \otimes e] \ma (n-1) z [1]$, whereas for $E_n^{-1}$, the homomorphism $\Zz \to \Zz$ is given by $z [1 \otimes e] \ma - (n-1) z [1]$ (for $z \in \Zz$). We now apply \cite[Theorem 2.2]{grer:rccconiII} to $I_i = \cO_{\infty} \otimes \cK$, $A_1 = \cO_{\infty} \otimes E_n^{+1}$, $A_2 = \cO_{\infty} \otimes E_n^{-1}$, $Q_i = \cO_n$ and the homomorphisms $\alpha  = - \id_{K_0(\cO_{\infty} \otimes \cK)}$, $\beta: \ K_0(\cO_{\infty} \otimes E_n^{+1}) \to K_0(\cO_{\infty} \otimes E_n^{-1}), \ z [1] \ma z [1]$ (for $z \in \Zz$), $\gamma = \id_{K_0(\cO_n)}$. It is then obvious that all the assumptions in \cite{grer:rccconiII} are satisfied, and we conclude that $\cO_{\infty} \otimes E_n^{+1} \cong \cO_{\infty} \otimes E_n^{-1}$.
\eproof

Now recall that we have introduced the extension algebra $E_1^0$ in Theorem \ref{basicunique}. The C*-algebra $E_1^0 \otimes E_n^{+1}$ ($n \geq 2$) contains the ideal $\cK \otimes \cK \cong \cK$, and we denote the corresponding quotient by $Q^+$. Obviously, the primitive ideals of $Q^+$ are given by $\cK \otimes \cO_n$, $\cO_1 \otimes \cK$ and $\cK \otimes \cO_n \oplus \cO_1 \otimes \cK$. From the six term exact sequence in K-theory for $0 \to \cK \to E_1^0 \otimes E_n^{+1} \to Q^+ \to 0$, we obtain $K_0(Q^+) \cong \Zz \cong K_1(Q^+)$, where $K_{0} ( Q^{+} )$ is generated by $[ 1_{ Q^{+} } ]$. All this also holds for the quotient $Q^-$ of $E_1^0 \otimes E_n^{-1}$ by the ideal $\cK \otimes \cK \cong \cK$. 
\blemma
\label{Q+Q-}
$Q^+$ and $Q^-$ are isomorphic. Moreover, there exists an automorphism of $Q^+$ which induces $\id_{\Zz}$ on $K_0$ and $- \id_{\Zz}$ on $K_1$.
\elemma
\bproof
The first statement is an application of Theorem~\ref{class_pi-orthideals} to $A_1 = Q^+$, $A_2 = Q^-$, $I_1 = \cK \otimes \cO_n \triangleleft \ Q^+$, $J_1 = \cO_1 \otimes \cK \triangleleft \ Q^+$, $I_2 = \cK \otimes \cO_n \ \triangleleft \ Q^-$, $J_2 = \cO_1 \otimes \cK \ \triangleleft \ Q^-$. Namely, it is straightforward to check that it is possible to choose $\alpha_I$, $\alpha_J$, $\alpha_{I \oplus J}$, $\beta$, $\gamma_I$, $\gamma_J$, and $\gamma_{I \oplus J}$ with all the desired properties in Theorem~\ref{class_pi-orthideals}.

The second statement follows in a similar way by applying Theorem~\ref{class_pi-orthideals} to $A_1 = A_2 = Q^+$, $I_1 = I_2 = \cK \otimes \cO_n \triangleleft \ Q^+$, $J_1 = J_2 = \cO_1 \otimes \cK \triangleleft \ Q^+$.
\eproof

\blemma
\label{E10xE}
For every $n \geq 2$, we have $E_1^0 \otimes E_n^{+1} \cong E_1^0 \otimes E_n^{-1}$.
\elemma
\bproof
By the previous lemma, we can identify $Q^+$ and $Q^-$ (we use the same notation as in the previous lemma) so that we can view $E_1^0 \otimes E_n^{+1}$ and $E_1^0 \otimes E_n^{-1}$ as extensions of $Q^+$:
\bgln
\label{ExtE+}
  && 0 \to \cK \to E_1^0 \otimes E_n^{+1} \to Q^+ \to 0 \\
\label{ExtE-}
  && 0 \to \cK \to E_1^0 \otimes E_n^{-1} \to Q^+ \to 0.
\egln
Again by the previous lemma, we can choose the identification $Q^+ \cong Q^-$ in such a way that for a fixed choice of isomorphisms $K_1(Q^+) \cong \Zz$, $K_0(\cK) \cong \Zz$, the index maps for both extensions \eqref{ExtE+} and \eqref{ExtE-} coincide. Now \cite[Theorem~2]{lgbmd:ecq} implies that \eqref{ExtE+} and \eqref{ExtE-} give the same class in $\Ext_s(Q^+)$. The reason is that $\operatorname{Ext}(K_0(Q^+),[1_{Q^+}])$ is the trivial group as $K_0(Q^+) \cong \Zz$ and $[1_{Q^+}]$ is a generator of $K_0(Q^+) \cong \Zz$. So the short exact sequence in \cite[Theorem~2]{lgbmd:ecq} tells us that two extensions of $Q^+$ by $\cK$ give the same class in $\Ext_s(Q^+)$ if their index maps coincide. But this is the case for \eqref{ExtE+} and \eqref{ExtE-} by construction. Hence $E_1^0 \otimes E_n^{+1} \cong E_1^0 \otimes E_n^{-1}$ by \cite[\S~3.2]{kkjkt:ek}.
\eproof

For $m, n \geq 2$, the C*-algebra $E_m^{+1} \otimes E_n^{-1}$ contains the ideal $\cK \otimes \cK \cong \cK$, and we denote the corresponding quotient by $Q^{+-}$. Obviously, the primitive ideals of $Q^{+-}$ are given by $\cK \otimes \cO_n$, $\cO_m \otimes \cK$ and $\cK \otimes \cO_n \oplus \cO_m \otimes \cK$. From the six term exact sequence in K-theory for $0 \to \cK \to E_m^{+1} \otimes E_n^{-1} \to Q^{+-} \to 0$, we obtain $K_0(Q^{+-}) \cong \Zz / (m-1)(n-1) \Zz$, with the class of the unit being a generator, and $K_1(Q^{+-}) \cong \gekl{0}$. All this also holds for the quotient $Q^{-+}$ of $E_m^{-1} \otimes E_n^{+1}$ by the ideal $\cK \otimes \cK \cong \cK$.
\blemma
\label{Q+-Q-+}
$Q^{+-}$ and $Q^{-+}$ are isomorphic.
\elemma
\bproof
As Lemma~\ref{Q+Q-}, this is an application of Theorem~\ref{class_pi-orthideals} to $A_1 = Q^{+-}$, $A_2 = Q^{-+}$, $I_1 = \cK \otimes \cO_n \triangleleft \ Q^{+-}$, $J_1 = \cO_m \otimes \cK \triangleleft \ Q^{+-}$, $I_2 = \cK \otimes \cO_n \ \triangleleft \ Q^{-+}$, $J_2 = \cO_m \otimes \cK \ \triangleleft \ Q^{-+}$. Namely, it is straightforward to check that it is possible to choose $\alpha_I$, $\alpha_J$, $\alpha_{I \oplus J}$, $\beta$, $\gamma_I$, $\gamma_J$, and $\gamma_{I \oplus J}$ with all the desired properties in Theorem~\ref{class_pi-orthideals}.
\eproof
\blemma
\label{E+-E-+}
We have $E_m^{+1} \otimes E_n^{-1} \cong E_m^{-1} \otimes E_n^{+1}$.
\elemma
\bproof
By the previous lemma, we can identify $Q^{+-}$ and $Q^{-+}$ (using the same notation as in the previous lemma) so that we can view $E_m^{+1} \otimes E_n^{-1}$ and $E_m^{-1} \otimes E_n^{+1}$ as extensions of $Q^{+-}$:
\bgln
\label{ExtE+-}
  && 0 \to \cK \to E_m^{+1} \otimes E_n^{-1} \to Q^{+-} \to 0 \\
\label{ExtE-+}
  && 0 \to \cK \to E_m^{-1} \otimes E_n^{+1} \to Q^{+-} \to 0.
\egln
Since $\operatorname{Hom}(K_1(Q^{+-}),\Zz) = \gekl{0}$, \cite[Theorem~2]{lgbmd:ecq} yields $\operatorname{Ext}((K_0(Q^{+-}),[1]), \Zz) \cong Ext_s(Q^{+-})$. Hence \eqref{ExtE+} and \eqref{ExtE-} give the same class in $\Ext_s(Q^{+-})$. The reason is that the exact sequences in $K_0$ for \eqref{ExtE+-} and \eqref{ExtE-+} clearly give rise to the same class in $\operatorname{Ext}((K_0(Q^{+-}),[1]), \Zz)$. Hence $E_m^{+1} \otimes E_n^{-1} \cong E_m^{-1} \otimes E_n^{+1}$ by \cite[\S~3.2]{kkjkt:ek}.
\eproof

In an entirely analogous way, we get
\blemma
\label{E++E--}
For all $m, n \geq 2$, we have $E_m^{+1} \otimes E_n^{+1} \cong E_m^{-1} \otimes E_n^{-1}$.
\elemma

\section{Classification of semigroup C*-algebras}

We are now ready to address the general classification problem for C*-algebras of the form $C^*(A_{\Gamma}^+)$. We begin with notation:

\begin{definition}\label{not}
Let $\Gamma$ be a graph with co-irreducible components $\Gamma_i = (V_i,E_i)$. We set 
\begin{gather*}
t(\Gamma) = \# \menge{\Gamma_i}{\abs{V_i} = 1}\\
o(\Gamma) = \# \menge{\Gamma_i}{\abs{V_i} = \infty},
\end{gather*}
and for every $n \in \Zz$
\[
N_n(\Gamma) = \# \menge{\Gamma_i}{1 < \abs{V_i} < \infty, \, \chi(\Gamma_i) = n}.
\]
\end{definition}

\btheo
\label{general-theo}
Let $\Gamma$ and $\Lambda$ be two  graphs. The semigroup C*-algebras $C^*(A_{\Gamma}^+)$ and $C^*(A_{\Lambda}^+)$ of the Artin monoids for $\Gamma$ and $\Lambda$ are stably isomorphic if and only if the following conditions hold:
\begin{enumerate}
\item[(i)] $t(\Gamma) = t(\Lambda)$;
\item[(ii)] 
$N_{-n}(\Gamma) + N_n(\Gamma) = N_{-n}(\Lambda) + N_n(\Lambda)$ for any $n\in\Zz$;
\item[(iii)] $\sum_{n \in \Zz} N_n(\Gamma) =\infty$ or $
\min(o(\Gamma),1)=\min(o(\Lambda),1)$.
\end{enumerate}
They are isomorphic if and only if further
\begin{enumerate}
\item[(iv)]
If $\sum_{n \in\Zz} N_n(\Gamma) < \infty$, $o(\Gamma)=0$ and $N_0(\Gamma)=0$,
then
$$\sum_{n=1}^{\infty} N_{-n}(\Gamma) \equiv \sum_{n=1}^{\infty} N_{-n}(\Lambda) \mod 2$$
\end{enumerate}
holds.
\etheo

\bremark
Note that when (ii) holds, all the conditions in (iii) are symmetric in $\Gamma$ and $\Lambda$. Similarly, when (ii) and (iii) hold, so are the conditions in (iv).

Moreover, if (ii) holds, then $\sum_{n=1}^{\infty} N_{-n}(\Gamma) \equiv \sum_{n=1}^{\infty} N_{-n}(\Lambda) \mod 2$ is equivalent to $\sum_{n=1}^{\infty} N_n(\Gamma) \equiv \sum_{n=1}^{\infty} N_n(\Lambda) \mod 2$. This explains condition (2) in Theorem~\ref{theo_finite-graphs}.
\eremark

For the proof of Theorem~\ref{general-theo}, we need some preparation. Given a graph $\Gamma$ with co-irreducible components $\Gamma_i = (V_i,E_i)$, let $\Gamma'$ be the graph we get from $\Gamma$ by removing all the co-irreducible components $\Gamma_i$ with $\abs{V_i} = 1$ and the corresponding edges. We then have a canonical isomorphism $C^*(A_{\Gamma'}^+) \cong \bigotimes_{\menge{\Gamma_i}{\abs{V_i} > 1}} C^*(A_{\Gamma_i}^+)$.

\begin{lemma}\label{pret}
There is a primitive ideal $I'$ of $C^*(A_{\Gamma}^+)$ such that $\Prim(C^*(A_{\Gamma}^+) / I')$ does not continuously surject onto $\Prim(\cT)$ and which is minimal among all the primitive ideals having this property, and we have
\[
 C^*(A_{\Gamma}^+) / I' \cong C^*(A_{\Gamma'}^+).
 \]
 \end{lemma}

\begin{proof}
Let $I$ be a primitive ideal of $C^*(A_{\Gamma}^+)$. As seen in Section \ref{preliminaries}, we know that $I$ is generated by $\gekl{\bigotimes_j J_{ij}}_i$, where $J_{ij} = C^*(A_{\Gamma_j}^+)$ for $i \neq j$ and $J_{ii} = I_i$ for primitive ideals $I_i$ of $C^*(A_{\Gamma_i}^+)$. It follows that $C^*(A_{\Gamma}^+) / I \cong \bigotimes_i C^*(A_{\Gamma_i}^+) / I_i$, and hence $\Prim(C^*(A_{\Gamma}^+) / I) \cong \prod_i \Prim(C^*(A_{\Gamma_i}^+) / I_i)$.  We now claim that there exists a continuous surjection $\Prim(C^*(A_{\Gamma}^+) / I) \onto \Prim{\cT}$ if and only if there exists a co-irreducible component $\Gamma_i$ of $\Gamma$ with $\abs{V_i} = 1$ and $I_i = (0)$. The direction \an{$\Larr$} is obvious. For \an{$\Rarr$}, assume that for every co-irreducible component $\Gamma_i$ of $\Gamma$ with $\abs{V_i} = 1$, $I_i$ is a maximal ideal of $C^*(A_{\Gamma_i}^+)$ such that $C^*(A_{\Gamma_i}^+) / I_i \cong \Cz$. Then $\Prim(C^*(A_{\Gamma}^+) / I) \cong \prod_k X_k$ where $X_k = \gekl{x_k,y_k}$ and the open subsets of $X_k$ are given by $\emptyset$, $\gekl{x_k}$ and $X_k$. This means that $\overline{\gekl{x_k}} = X_k$ and $\overline{\gekl{y_k}} = \gekl{y_k}$. Furthermore, we know that $\Prim(\cT) = \gekl{\bullet} \sqcup\Tz$, where $\overline{\gekl{\bullet}} = \Prim(\cT)$. Let $f: \prod_k X_k \to \Prim(\cT)$ be a continuous map. We want to show that $f$ cannot be surjective. Let $y = (y_k)_k$ and $f(y) = z$. For arbitrary $x \in \prod_k X_k$, we always have $y \in \overline{\gekl{x}}$. As $f^{-1}(\overline{\gekl{f(x)}})$ is closed and contains $x$, it must also contain $y$. Hence $z = f(y)$ lies in $\overline{\gekl{f(x)}}$. This implies that $f(x) = z$ or $f(x) = \bullet$. But this holds for every $x$ in $\prod_k X_k$. Hence the image of $f$ contains at most $2$ points, and thus $f$ cannot be surjective. This shows our claim.

Therefore, a primitive ideal $I'$ of $C^*(A_{\Gamma}^+)$ such that $\Prim(C^*(A_{\Gamma}^+) / I')$ does not continuously surject onto $\Prim(\cT)$ and which is minimal among all the primitive ideals with this property is generated by $\gekl{\bigotimes_j J_{ij}}_i$, where for a co-irreducible component $\Gamma_i$ with $\abs{V_i} = 1$, $J_{ii} = I_i$ is a maximal ideal of $C^*(A_{\Gamma_i}^+)$ with $C^*(A_{\Gamma_i}^+) / I_i \cong \Cz$, and for a co-irreducible component $\Gamma_i$ with $\abs{V_i} > 1$, $J_{ii} = (0)$. We conclude that $C^*(A_{\Gamma}^+) / I' \cong C^*(A_{\Gamma'}^+)$. 
\end{proof}

\blemma
\label{t}
Let $\Gamma$ and $\Lambda$ be two graphs. 
\begin{itemize}
\item[(1)] If $C^*(A_{\Gamma}^+)$ and $C^*(A_{\Lambda}^+)$ are isomorphic, then $t(\Gamma) = t(\Lambda)$ and $C^*(A_{\Gamma'}^+) \cong C^*(A_{\Lambda'}^+)$.

\item[(2)]  If $C^*(A_{\Gamma}^+) \otimes \cK$ and $C^*(A_{\Lambda}^+) \otimes \cK$ are isomorphic, then $t(\Gamma) = t(\Lambda)$ and $C^*(A_{\Gamma'}^+) \otimes \cK \cong C^*(A_{\Lambda'}^+) \otimes \cK$.

\end{itemize}
\elemma
\bproof
We first prove (1).  Since an isomorphism $C^*(A_{\Gamma}^+) \cong C^*(A_{\Lambda}^+)$ sends the primitive ideal $I$ to a primitive ideal of $C^*(A_{\Lambda}^+)$ with the analogous property, we conclude that every isomorphism $C^*(A_{\Gamma}^+) \cong C^*(A_{\Lambda}^+)$ induces an isomorphism $C^*(A_{\Gamma'}^+) \cong C^*(A_{\Lambda'}^+)$. To prove that $t(\Gamma) = t(\Lambda)$, we observe that the primitive ideals of $C^*(A_{\Gamma}^+)$ which are contained in $I$ are in one-to-one correspondence with the subsets of $\menge{\Gamma_i}{\abs{V_i} = 1}$. Again, as an isomorphism $C^*(A_{\Gamma}^+) \cong C^*(A_{\Lambda}^+)$ sends the primitive ideal $I$ to a primitive ideal of $C^*(A_{\Lambda}^+)$ with the analogous property, we conclude that the power sets of $\menge{\Gamma_i}{\abs{V_i} = 1}$ and $\menge{\Lambda_j}{\abs{W_j} = 1}$ have the same cardinality. Hence also $\menge{\Gamma_i}{\abs{V_i} = 1}$ and $\menge{\Lambda_j}{\abs{W_j} = 1}$ must have the same cardinality (which is either finite or countably infinite).  This proves (1).

(2) is proved in a similar way as (1) using the observation that every primitive ideal of $B \otimes \cK$ is of the form $I \otimes \cK$, where $I$ is a primitive ideal of $B$.
\eproof

\blemma
Let $A_i$, $i = 1,2,\dotsc$, be a countably infinite family of properly infinite unital C*-algebras. Then $A = \bigotimes_{i=1}^{\infty} A_i$ is purely infinite.
\elemma
\bproof
We have to show that every non-zero positive element $a$ of $A$ is properly infinite. By \cite[Lemma~3.3]{ekmr:incacao}, it suffices to find for every $\ve > 0$ a properly infinite, positive element $b \in A$ with $\norm{a-b} < \ve$ and $b \precsim a$. Since $A = \bigotimes_{i=1}^{\infty} A_i$, there exists a (sufficiently large) natural number $n$ and a positive element $x \in \bigotimes_{i=1}^n A_i$ with $\norm{a - x \otimes 1} < \tfrac{\ve}{2}$. By \cite[Lemma~2.2]{ekmr:incacao}, we have that $b \defeq (x - \tfrac{\ve}{2})_+ \otimes 1 = (x \otimes 1 - \tfrac{\ve}{2})_+$ satisfies $b \precsim a$. Also, we have $\norm{b-a} \leq \norm{b - x \otimes 1} + \norm{x \otimes 1 - a} < \ve$. So it suffices to show that $b$ is properly infinite. By construction, $b$ is of the form $c \otimes 1$ for some positive element $c \in \bigotimes_{i=1}^n A_i$. Since the unit $1 \in A_{n+1}$ is properly infinite, we can find isometries $s$ and $t$ in $A_{n+1}$ with $ss^* \perp tt^*$. So $b = c \otimes 1 = (c^{1/2} \otimes s)^* (c^{1/2} \otimes s) \approx (c^{1/2} \otimes s) (c^{1/2} \otimes s)^* = c \otimes ss^*$. Similarly, $b \approx c \otimes tt^*$. But since $(c \otimes ss^*)(c \otimes tt^*) = 0$, we conclude that $b \oplus b \approx (c \otimes ss^*) \oplus (c \otimes tt^*) \approx c \otimes (ss^* + tt^*) \leq c \otimes 1 = b$.
\eproof
\blemma
\label{infirrcomp-pi}
Let $\Gamma$ be a graph with (countably) infinitely many co-irreducible components $\Gamma_i = (V_i,E_i)$, $i = 1,2,\dotsc$. Assume that $1 < \abs{V_i} < \infty$ for all $i$. Then $C^*(A_{\Gamma}^+)$ is strongly purely infinite, i.e., $C^*(A_{\Gamma}^+) \cong C^*(A_{\Gamma}^+) \otimes \cO_{\infty}$.
\elemma
\bproof
By \cite[Theorem~8.3]{jcml:bqitaag}, we know that $C^*(A_{\Gamma}^+)$ has the ideal property (the definition can be found in \cite[Remark~2.1]{cpmr:picrrz}). Moreover, we know that $C^*(A_{\Gamma}^+) \cong \bigotimes_{i=1}^{\infty} C^*(A_{\Gamma_i}^+)$, and each of the $C^*(A_{\Gamma_i}^+)$ is a properly infinite unital C*-algebra. Hence by the previous lemma, we know that $C^*(A_{\Gamma}^+)$ is purely infinite. Therefore, \cite[Proposition~2.14]{cpmr:picrrz} tells us that $C^*(A_{\Gamma}^+)$ is strongly purely infinite. And finally, if $C^*(A_{\Gamma}^+)$ is strongly purely infinite, then \cite[Theorem~9.1]{ekmr:incacao} implies that $C^*(A_{\Gamma}^+) \cong C^*(A_{\Gamma}^+) \otimes \cO_{\infty}$ because $C^*(A_{\Gamma}^+)$ is nuclear and unital.
\eproof

Finally, we are ready for the proof of Theorem~\ref{general-theo}.
\bproof[Proof of Theorem~\ref{general-theo}]
Let us first of all show that if $C^*(A_{\Gamma}^+)\otimes\cK \cong C^*(A_{\Lambda}^+)\otimes\cK$ holds, then conditions (i), (ii) and (iii) must be satisfied. By Lemma~\ref{t} condition (i) holds and that $C^*(A_{\Gamma'}^+)\otimes\cK \cong C^*(A_{\Lambda'}^+)\otimes\cK$. Hence we may assume that all the co-irreducible components of $\Gamma$ and $\Lambda$ have more than one vertex. 

To prove (ii), we observe that the minimal non-zero primitive ideals of $C^*(A_{\Gamma}^+)$ are of the form $I_i = \otimes_j J_{ij}$, where $J_{ij} = C^*(A_{\Gamma_j}^+)$ if $j \neq i$, and $J_{ii} = \cK \triangleleft C^*(A_{\Gamma_i}^+)$ ($\Gamma_i$ consists of only finitely many vertices). For the corresponding quotient, we get $C^*(A_{\Gamma}^+) / I_i \cong \bigotimes_j Q_{ij}$, where $Q_{ij} = C^*(A_{\Gamma_j}^+)$ if $j \neq i$, and $Q_{ii} = C^*(A_{\Gamma_i}^+) / \cK$. Since $K_0(C^*(A_{\Gamma_i}^+) / \cK) \cong \Zz / \abs{\chi(\Gamma_i)} \Zz$, it follows that $K_0(C^*(A_{\Gamma}^+) / I_i) \cong \Zz / \abs{\chi(\Gamma_i)} \Zz$. Hence, we have shown that $N_0(\Gamma)$ is the number of minimal non-zero primitive ideals $I\otimes\cK$ of $C^*(A_{\Gamma}^+)\otimes\cK$ with the property that $K_0(C^*(A_{\Gamma}^+) / I) \cong \Zz$, and that for every $n = 1, 2, \dotsc$, $N_{-n}(\Gamma) + N_n(\Gamma)$ is the number of minimal non-zero primitive ideals $I\otimes\cK$ of $C^*(A_{\Gamma}^+)\otimes\cK$ with the property that $K_0(C^*(A_{\Gamma}^+) / I) \cong \Zz / n \Zz$. Since these descriptions are invariant under stable isomorphisms of C*-algebras, we conclude that (ii) must hold. 

Let us now prove (iii) under the assumption of stable isomorphism.  If $\sum N_n(\Gamma) = \infty$ we are done, so suppose the contrary and note that in this case,  $C^*(A_{\Gamma}^+)$ is strongly purely infinite if and only if $o(\Gamma) > 0$. The direction \an{$\Rarr$} is clear, since $o(\Gamma) > 0$ implies that $C^*(A_{\Gamma}^+)$ has $\cO_{\infty}$ as a tensor factor. To prove \an{$\Larr$}, we observe that if $o(\Gamma) = 0$, then $C^*(A_{\Gamma}^+)$ contains the algebra of compact operators as an ideal, hence cannot be strongly purely infinite. As a consequence,  $C^*(A_{\Gamma}^+)\otimes\cK \cong C^*(A_{\Lambda}^+)\otimes\cK$ implies that either both $o(\Gamma) > 0$ and $o(\Lambda) > 0$, or $o(\Gamma) = o(\Lambda) = 0$, as desired. 

Finally, we assume that $C^*(A_{\Gamma}^+) \cong C^*(A_{\Lambda}^+)$ and that $\sum N_n(\Gamma) <\infty$, that $N_0(\Gamma)=0$ and that $o(\Gamma)=0$. The algebra $\cK$ of compact operators sits inside $C^*(A_{\Gamma}^+)$ as the (unique) minimal non-zero ideal. The inclusion $\cK \into C^*(A_{\Gamma}^+)$ sends in K-theory the $K_0$-class of a minimal projection to $(\prod_i \chi(\Gamma_i)) \cdot [1]$, where $(\prod_i \chi(\Gamma_i))$ is the product over all co-irreducible components of $\Gamma$ (there are only finitely many by assumption) of the Euler characteristics. As $N_0(\Gamma) = 0$, $(\prod_i \chi(\Gamma_i))$ is a non-zero number, and it is positive if and only if $\sum_{n=1}^{\infty} N_{-n}(\Gamma) \equiv 0 \mod 2$. Since $C^*(A_{\Gamma}^+) \cong C^*(A_{\Lambda}^+)$, we must have $\sum_{n=1}^{\infty} N_{-n}(\Gamma) \equiv \sum_{n=1}^{\infty} N_{-n}(\Lambda) \mod 2$. Therefore, all in all, condition (iv) follows when the C*-algebras are isomorphic.

In the opposite direction, we know from Sections \ref{preliminaries} and \ref{coirr} that
\[
  C^*(A_{\Gamma}^+) 
  \cong
  \cT^{\otimes t(\Gamma)} \otimes \cO_{\infty}^{\otimes o(\Gamma)} \otimes \bigotimes_{n=0}^{\infty} \bigotimes_{\menge{i}{\abs{\chi(\Gamma_i)} = n}} E_{1+n}^{\sgn(\chi(\Gamma_i))} \]
and
\[
C^*(A_{\Lambda}^+) 
  \cong 
\cT^{\otimes t(\Lambda)} \otimes \cO_{\infty}^{\otimes o(\Lambda)} \otimes \bigotimes_{n=0}^{\infty} \bigotimes_{\menge{i}{\abs{\chi(\Lambda_i)} = n}} E_{1+n}^{\sgn(\chi(\Lambda_i))}
\]
We note from the outset  that whenever $o(\Gamma)>0$ or $N_0(\Gamma)>0$ then by repeated application of  either Lemma \ref{OinfxE} or Lemma \ref{E10xE} we may simplify these expressions to 
\begin{equation}\label{Gamstd}
  C^*(A_{\Gamma}^+) \cong  \cT^{\otimes t(\Gamma)} \otimes \cO_{\infty}^{\otimes o(\Gamma)}
  \otimes (E_1^0)^{\otimes N_0(\Gamma)} \otimes \bigotimes_{n=1}^{\infty} (E_{1+n}^{+1})^{\otimes (N_{-n}(\Gamma) + N_n(\Gamma))} .
\end{equation}

Assume that (i), (ii) and (iii) hold. 
We begin by noting that in the case $\sum N_{n} (\Gamma) < \infty$ if either $o(\Gamma)>0$ or $N_0(\Gamma)>0$, we also have either $o(\Lambda)>0$ or $N_0(\Lambda)>0$, and we get $C^*(A_{\Gamma}^+)\cong C^*(A_{\Lambda}^+)$ by reducing to the form given in \eqref{Gamstd} and applying (i) and (ii).

When $\sum N_n(\Gamma) = \infty$ then we have  by (ii) and Lemma \ref{infirrcomp-pi} that both $C^*(A_{\Gamma}^+)$ and $C^*(A_{\Lambda}^+)$ are strongly purely infinite, and hence we have
\bglnoz
  C^*(A_{\Gamma}^+) 
  &\cong&
  \cT^{\otimes t(\Gamma)} \otimes \cO_{\infty} 
  \otimes (E_1^0)^{\otimes N_0(\Gamma)} \otimes \bigotimes_{n=1}^{\infty} (E_{1+n}^{+1})^{\otimes (N_{-n}(\Gamma) + N_n(\Gamma))} \\
  &=&  
  \cT^{\otimes t(\Lambda)} \otimes \cO_{\infty}
  \otimes (E_1^0)^{\otimes N_0(\Lambda)} \otimes \bigotimes_{n=1}^{\infty} (E_{1+n}^{+1})^{\otimes (N_{-n}(\Lambda) + N_n(\Lambda))} \\
  &\cong&
   C^*(A_{\Lambda}^+),
\eglnoz
since Lemma \ref{OinfxE} may be applied as above.

It remains to treat the case that $o(\Gamma)=N_0(\Gamma)=0$ and 
 $\sum N_n(\Gamma) < \infty$. Again, by (ii) and (iii), we must have 	$o(\Lambda)=N_0(\Lambda)=0$ and $\sum  N_n(\Lambda) < \infty$ as well, and we get
\bglnoz
  C^*(A_{\Gamma}^+)\otimes\cK
  &\cong&  
  \cT^{\otimes t(\Gamma)} 
  \otimes \bigotimes_{n=1}^{\infty} (E_{1+n}^{+1})^{\otimes (N_{-n}(\Gamma) + N_n(\Gamma))} \otimes\cK\\
  &=&  
  \cT^{\otimes t(\Lambda)} 
  \otimes \bigotimes_{n=1}^{\infty} (E_{1+n}^{+1})^{\otimes (N_{-n}(\Lambda) + N_n(\Lambda))}\otimes\cK\\
 & \cong& C^*(A_{\Lambda}^+)\otimes\cK
\eglnoz
this time appealing to the second half of Lemma \ref{notconfuse}.

Assuming further (iv), we now aim for exact isomorphism, noting that we have already established it when  $o(\Gamma)>0$, $N_0(\Gamma)>0$ or 
 $\sum N_n(\Gamma) = \infty$. We hence assume that $o(\Gamma)=N_0(\Gamma)=0$ and note that also $o(\Lambda)=N_0(\Lambda)=0$
 
Consider first the case where both  $\sum_{n=1}^{\infty} N_{-n}(\Gamma)$ and  $\sum_{n=1}^{\infty} N_{-n}(\Lambda)$ are even. We have
\bglnoz
  C^*(A_{\Gamma}^+) 
  &\cong& \cT^{\otimes t(\Gamma)} 
  \otimes \bigotimes_{n=1}^{\infty} (E_{1+n}^{+1})^{\otimes N_n(\Gamma)} \otimes \bigotimes_{n=1}^{\infty} (E_{1+n}^{-1})^{\otimes N_{-n}(\Gamma)} \\
  &\cong& \cT^{\otimes t(\Gamma)} 
  \otimes \bigotimes_{n=1}^{\infty} (E_{1+n}^{+1})^{\otimes N_n(\Gamma)} \otimes \bigotimes_{n=1}^{\infty} (E_{1+n}^{+1})^{\otimes N_{-n}(\Gamma)} \\
  &\cong& 
  \cT^{\otimes t(\Lambda)} \otimes \bigotimes_{n=1}^{\infty} (E_{1+n}^{+1})^{\otimes (N_{-n}(\Gamma) + N_n(\Gamma))} \\
  &=& 
  \cT^{\otimes t(\Lambda)} \otimes \bigotimes_{n=1}^{\infty} (E_{1+n}^{+1})^{\otimes (N_{-n}(\Lambda) + N_n(\Lambda))}
  \cong C^*(A_{\Lambda}^+)
\eglnoz
by Lemma~\ref{E++E--}.
Now assume both $\sum_{n=1}^{\infty} N_{-n}(\Gamma)$ and $\sum_{n=1}^{\infty} N_{-n}(\Lambda)$ are odd. If there exists $\chi < 0$ such that there are co-irreducible components $\Gamma_k$ and $\Lambda_l$ with $\chi(\Gamma_k) = \chi = \chi(\Lambda_l)$, then we deduce from the previous case that
\bglnoz
  C^*(A_{\Gamma}^+) 
  &\cong& \rukl{\bigotimes_{\Gamma_i \neq \Gamma_k} C^*(A_{\Gamma_i}^+)} \otimes C^*(A_{\Gamma_k}^+) \\
  &\cong& \rukl{\bigotimes_{\Gamma_i \neq \Gamma_k} C^*(A_{\Gamma_i}^+)} \otimes E_{1 + \abs{\chi}}^{-1} \\
  &\cong& \rukl{\bigotimes_{\Lambda_j \neq \Lambda_l} C^*(A_{\Lambda_j}^+)} \otimes C^*(A_{\Lambda_l}^+)
  \cong C^*(A_{\Lambda}^+).
\eglnoz
If there exists no such $\chi$, then by (ii) there must be $\chi < 0$, $\psi < 0$ and co-irreducible components $\Gamma_{k_-}$, $\Gamma_{k_+}$, $\Lambda_{l_-}$, $\Lambda_{l_+}$ with $\chi(\Gamma_{k_-}) = \chi$, $\chi(\Lambda_{l_+}) = - \chi$, $\chi(\Gamma_{k_+}) = - \psi$ and $\chi(\Lambda_{l_-}) = \psi$. Hence
\bglnoz
  C^*(A_{\Gamma}^+) 
  &\cong& \rukl{\bigotimes_{\Gamma_i \neq \Gamma_{k_+}, \, \Gamma_{k_-}} C^*(A_{\Gamma_i}^+)} \otimes C^*(A_{\Gamma_{k_+}}^+) \otimes C^*(A_{\Gamma_{k_-}}^+) \\
  &\cong& \rukl{\bigotimes_{\Gamma_i \neq \Gamma_{k_+}, \, \Gamma_{k_-}} E_{1+ \abs{\chi(\Gamma_i)}}^{\sgn(\chi(\Gamma_i))}} 
  \otimes E_{1 + \abs{\psi}}^{+1} \otimes E_{1 + \abs{\chi}}^{-1} \\
  &\cong& \rukl{\bigotimes_{\Gamma_i \neq \Gamma_{k_+}, \, \Gamma_{k_-}} E_{1+ \abs{\chi(\Gamma_i)}}^{\sgn(\chi(\Gamma_i))}}
  \otimes E_{1 + \abs{\psi}}^{-1} \otimes E_{1 + \abs{\chi}}^{+1} \\
  &\cong& \rukl{\bigotimes_{\Gamma_i \neq \Gamma_{k_+}, \, \Gamma_{k_-}} C^*(A_{\Gamma_i}^+)} \otimes C^*(A_{\Lambda_{l_-}}^+) \otimes C^*(A_{\Lambda_{l_+}}^+)
  \cong C^*(A_{\Lambda}^+).
\eglnoz
In the third step, we used Lemma~\ref{E+-E-+}, and in the fourth step, we used our argument in the previous case. \eproof

\section{The isomorphism problem from the perspective of classification of non-simple C*-algebras}

We give an interpretation of Theorem~\ref{general-theo} from the point of view of classifying non-simple C*-algebras.

We let $\mathcal{O} ( \mathrm{Prim} ( A ) )$ denote the set of open subsets in $\mathrm{Prim}(A)$, and $\mathbb{I} (A)$ the lattice of ideals. The map $\ftn{ \psi_{A} }{ \mathcal{O} ( \mathrm{Prim} ( A ) ) }{ \mathbb{I} (A) }$ given by $\psi_{A} ( U) = \bigcap_{ \rho \notin U } \rho$ is a lattice isomorphism which preserves arbitrary suprema and finite infima.  We denote $\psi_{A} (U)$ by $A [ U ]$.  For every C*-algebra $A$, we denote the pair 
\begin{align*}
\left( \mathrm{Prim} ( A ) , \{ \ksix ( A / A[U] ; A[V] / A[U] ) \}_{ \substack{ V, U \in \mathbb{O} ( \mathrm{Prim} (A) ) \\ U \subseteq V} } \right)
\end{align*}
by $\mathfrak{F} ( A )$, where $\ksix(B,J)$ denotes the standard six-term exact sequence associated to an ideal $J$ of a C*-algebra $B$, considering each $K_0$-group as an ordered group. 

An isomorphism from $\mathfrak{F} (A)$ to $\mathfrak{F} (B)$ thus consists of a homeomorphism
$$
  \ftn{ \phi }{ \mathrm{Prim} (A) }{ \mathrm{Prim} (B) }
$$
and isomorphisms
$$
  \ftn{ \alpha_{U, V} }{ K_{*} ( A[V] / A[U] ) }{ K_{*} ( B[ \phi(V) ] / B [ \phi(U) ] ) }
$$
for each $U, V \in \mathbb{O} ( \mathrm{Prim} (A) )$ with $U \subseteq V$, such that $( \alpha_{U, V}, \alpha_{X, U} , \alpha_{ X, V} )$ is an isomorphism from $\ksix ( A  / A [U] ;  A [ V ] / A [U] )$ to $\ksix ( B / B[ \phi(U) ] ; B [ \phi(V) ] / B[ \phi(U) ] )$ in the sense that it makes all squares commute and is an order isomorphism on all even parts of the $K$-theory. 

 If $A$ and $B$ are unital, we write $( \mathfrak{F}(A) , [1_{A} ] ) \cong ( \mathfrak{F}(B) , [ 1_{B} ] )$ if $\mathfrak{F} (A) \cong \mathfrak{F} (B)$ in such a way that the isomorphism $\alpha_{X, \emptyset}$ sends $[1_{A}]$ in $K_{0} (A)$ to $[1_{B}]$ in $K_{0} (B)$.

Note that if $\ftn{ \phi }{ \mathrm{Prim} (A) }{ \mathrm{Prim} (B) }$ is a homeomorphism, there exists a lattice isomorphism from $\mathbb{I} ( A )$ to $\mathbb{I} (B)$ given by $I \mapsto \psi_{B} ( \phi ( \psi_{A}^{-1} (I) ) )$.  Hence, if $A$ and $B$ are separable and $\ftn{ \phi }{ \mathrm{Prim} (A) }{ \mathrm{Prim} (B) }$ is a homeomorphism, then for all $U \in \mathbb{O} ( \mathrm{Prim} ( A ) )$, we have that $A [U]$ is a primitive ideal of $A$ if and only if $B[ \phi (U) ]$ is a primitive ideal of $B$ (because primitive ideals are precisely given by prime ideals for separable C*-algebras).

The following easy observation is left to the reader.

\begin{lemma}
Let $A$ and $B$ be separable C*-algebras.  Let $U \in \mathbb{O} ( \mathrm{Prim} (A) )$.  
\begin{itemize}
\item[(1)] If $\mathfrak{F} (A) \cong \mathfrak{F} (B)$ via a homeomorphism $\ftn{ \phi }{ \mathrm{Prim} (A) }{ \mathrm{Prim} (B) }$, then
$$
  \mathfrak{F}(A / A[U] ) \cong \mathfrak{F} ( B / B[ \phi (U) ] ).
$$

\item[(2)] If $A$ and $B$ are unital C*-algebras and $( \mathfrak{F} (A) , [ 1_{A} ] ) \cong ( \mathfrak{F} (B) , [ 1_{B} ] )$ via a homeomorphism $\ftn{ \phi }{ \mathrm{Prim} (A) }{ \mathrm{Prim} (B) }$, then 
\begin{align*}
\left( \mathfrak{F}(A / A[U] ) , [ 1_{ A / A[U] } ] \right) \cong \left( \mathfrak{F} ( B / B[ \phi (U) ] ) , [ 1_{ B / B[ \phi (U) ] } ] \right).
\end{align*}
\end{itemize}   
\end{lemma}

\begin{theorem}\label{K-theory-main}
Let $\Gamma$ and $\Lambda$ be two (countable) graphs. Then $C^{*} ( A_{ \Gamma }^{+} ) \otimes \cK \cong C^{*} ( A_{ \Lambda }^{+} ) \otimes \cK$ if and only if $\mathfrak{F} (C^{*} ( A_{ \Gamma }^{+} )) \cong \mathfrak{F}(C^{*} ( A_{ \Lambda }^{+} ))$, and $C^{*} ( A_{ \Gamma }^{+} ) \cong C^{*} ( A_{ \Lambda }^{+} )$ if and only if $( \mathfrak{F} (C^{*} ( A_{ \Gamma }^{+} )) , [ 1_{C^{*} ( A_{ \Gamma }^{+} )} ] ) \cong ( \mathfrak{F}(C^{*} ( A_{ \Lambda }^{+} )) , [ 1_{C^{*} ( A_{ \Lambda }^{+} )} ] )$.
\end{theorem}

\begin{proof}
For both statements, the direction \an{$\Rarr$} is obvious. To prove \an{$\Larr$}, we show that $\mathfrak{F} (C^{*} ( A_{ \Gamma }^{+} )) \cong \mathfrak{F}(C^{*} ( A_{ \Lambda }^{+} ))$ implies (i), (ii) and (iii) from Theorem~\ref{general-theo}, and that $( \mathfrak{F} (C^{*} ( A_{ \Gamma }^{+} )) , [ 1_{C^{*} ( A_{ \Gamma }^{+} )} ] )
  \cong ( \mathfrak{F}(C^{*} ( A_{ \Lambda }^{+} )) , [ 1_{C^{*} ( A_{ \Lambda }^{+} )} ] )$
implies (iv) from Theorem~\ref{general-theo}. We use the notations from Lemma~\ref{t}. The first step is to prove that $\mathfrak{F} (C^{*} ( A_{ \Gamma }^{+} )) \cong \mathfrak{F}(C^{*} ( A_{ \Lambda }^{+} ))$ implies $t(\Gamma) = t(\Lambda)$ and $\mathfrak{F} (C^{*} ( A_{ \Gamma' }^{+} )) \cong \mathfrak{F}(C^{*} ( A_{ \Lambda' }^{+} ))$. $t(\Gamma) = t(\Lambda)$ follows by Lemma~\ref{pret}, because we only use the primitive ideal space and the lattice structure of the set of ideals in this proof. To see that $\mathfrak{F} (C^{*} ( A_{ \Gamma' }^{+} )) \cong \mathfrak{F}(C^{*} ( A_{ \Lambda' }^{+} ))$, let $I'$ be a primitive ideal of $C^*(A_{\Gamma}^+)$ stipulated in Lemma \ref{pret}, and let $U$ be an open set of $\Prim(C^*(A_{\Gamma}^+))$ such that $C^*(A_{\Gamma}^+)[U] = I'$. Then $C^*(A_{\Lambda}^+)[\phi(U)]$ is an ideal with the analogous property. In the proof of Lemma~\ref{t}, we have seen that $C^*(A_{\Gamma}^+) / C^*(A_{\Gamma}^+)[U] \cong C^*(A_{\Gamma'}^+)$. Similarly, we have $C^*(A_{\Lambda}^+) / C^*(A_{\Lambda}^+)[U] \cong C^*(A_{\Lambda'}^+)$. Therefore, (2) from the previous lemma tells us that $\mathfrak{F} (C^{*} ( A_{ \Gamma' }^{+} )) \cong \mathfrak{F}(C^{*} ( A_{ \Lambda' }^{+} ))$, as desired. 

In particular, this implies (i), and we may assume as in the proof of Theorem~\ref{general-theo} that all the co-irreducible components of $\Gamma$ and $\Lambda$ have more than one vertex. Then (ii) follows in exactly the same way as in the proof of Theorem~\ref{general-theo} because we only use primitive ideal spaces, lattice structures of sets of ideals and $K_0$ in this proof. All this can be extracted from the invariant $\mathfrak{F}$. Let us prove (iii). As we have seen in the proof of Theorem~\ref{general-theo}, $o(\Gamma) = 0$ implies that $\cK$ is an ideal of $C^*(A_{\Gamma}^+)$, whereas $o(\Gamma) > 0$ implies that $C^*(A_{\Gamma}^+)$ (and hence also every non-zero ideal) is strongly purely infinite. These two cases can be distinguished by the order on $K_0$. Therefore, we see as in the proof of Theorem~\ref{general-theo} that if $\mathfrak{F} (C^{*} ( A_{ \Gamma }^{+} )) \cong \mathfrak{F}(C^{*} ( A_{ \Lambda }^{+} ))$, then either $o(\Gamma) > 0$ and $o(\Lambda) > 0$ or $o(\Gamma) = 0$ and $o(\Lambda) = 0$. Now assume that $( \mathfrak{F} (C^{*} ( A_{ \Gamma }^{+} )) , [ 1_{C^{*} ( A_{ \Gamma }^{+} )} ] ) \cong ( \mathfrak{F}(C^{*} ( A_{ \Lambda }^{+} )) , [ 1_{C^{*} ( A_{ \Lambda }^{+} )} ] )$. Then the proof of (iv) follows the proof of Theorem~\ref{general-theo}, where we only use lattice structures of sets of ideals, $K_0$ and the $K_0$-classes of the units.
\end{proof}

\section{Graph algebras and the semiprojectivity question}

Apart from semigroup C*-algebras we discussed above, there is another - more traditional - way of constructing a C*-algebra out of a directed graph, possibly allowing for loops. Now we would like to discuss the overlap of these two constructions. In other words, we are interested in the question: Which semigroup C*-algebras for right-angled Artin monoids are isomorphic to graph algebras? We can provide a complete answer to this question.

\subsection{Extensions of C*-algebras}  We first establish some facts about absorbing extensions and the C*-algebras associated to these extensions.  To each injective Busby map $\ftn{ \tau }{ A}{ \mathcal{Q}(B) }$, where $\mathcal{Q}(B) = \mathcal{M}(B) / B$ with $\mathcal{M}(B)$ the multiplier algebra of $B$, associate as usual the extension 
\[
\xymatrix{
e: & 0 \ar[r] & B \ar@{=}[d] \ar@{^{(}->}[r] & E \ar[r]^{\psi} \ar[d] & A \ar[r] \ar[d]^{\tau}& 0 \\
	& 0 \ar[r]    & B		 \ar@{^{(}->}[r]	& \mathcal{M}(B) \ar[r]^{\pi} & \mathcal{Q}(B) \ar[r] & 0
}
\]  
with $E = \pi^{-1} ( \tau (A) )$ and $\psi (x) = \tau^{-1} ( \pi (x) )$.  Note that $\psi$ is a homomorphism since $\tau$ is injective. 

We call $\tau$ (and $e$) \emph{unital} if $A$ is unital and $\tau$ is a unital homomorphism, or, equivalently, if $E$ is a unital C*-algebra.  If $\tau = \pi \circ \alpha$ for some homomorphism $\ftn{ \alpha }{ A}{\mathcal{M}(B)}$, then $\tau$ is called a \emph{trivial extension}.  If $A$ is unital and $\tau = \pi \circ \alpha$ for some unital homomorphism $\ftn{ \alpha }{ A}{\mathcal{M}(B)}$, then $\tau$ is called \emph{strongly unital}.  Not all unital trivial extensions are strongly unital.

Assume that $B$ is stable.  The sum $\tau \oplus \tau'$ of two extensions $\ftn{ \tau, \tau'}{A}{\mathcal{Q}(B)}$ is defined as follows.  Since $B$ is stable, there exist isometries $s_{1}, s_{2} \in \mathcal{M}(B)$ with $1_{\mathcal{M}(B)} = s_{1}s_{1}^{*} + s_{2} s_{2}^{*}$.  Set
\[
(\tau \oplus \tau' )(a) = \pi ( s_{1} ) \tau(a) \pi ( s_{1}^{*} ) + \pi ( s_{2} ) \tau' (a) \pi(s_{2}^{*} )
\]
for all $a \in A$. 

Two extensions $\ftn{ \tau , \tau' }{A}{\mathcal{Q}(B)}$ are said to be \emph{unitarily equivalent}, denoted by $\tau \sim_{u} \tau'$, if there exists a unitary $u \in \mathcal{M}(B)$ such that $\pi(u) \tau ( a) \pi(u)^{*} = \tau'(a)$ for all $a \in A$.  Then two extensions $\ftn{ \tau_{1} , \tau_{2} }{A}{ \mathcal{Q} (B) }$ define the same element in $\mathrm{Ext} (A, B )$ if there exists a unitary $u \in \mathcal{M}(B)$ and there exist trivial extensions $\ftn{ \tau_{1}' , \tau_{2}' }{ A }{ \mathcal{Q}(B) }$ such that $\tau_{1} \oplus \tau_{1}' \sim_{u} \tau_{2} \oplus \tau_{2}'$.  If $\tau_{1}$ and $\tau_{2}$ are unital extensions, then $\tau_{1}'$ and $\tau_{2}'$ can be chosen to be unital extensions (see \cite[Section~5]{mk: extensions kirchberg}).

For a C*-algebra $C$, we let $\widetilde{C}$ be the unitization of $C$ (adding a new unit if $C$ is a unital C*-algebra) and let $\ftn{ \iota_{C} }{C}{\widetilde{C}}$ be the embedding of $C$ into $\widetilde{C}$ as an ideal.  

Recall that an ideal $I$ of a C*-algebra $A$ is an \emph{essential ideal} if every nonzero ideal of $A$ has a nontrivial intersection with $I$.  An extension $0 \to I \overset{ \iota }{ \to } A \to B \to 0$ is \emph{essential} if $\iota(I)$ is an essential ideal of $A$.  It is a well-known fact that an extension $0 \to I \to A \to B \to 0$ is an essential extension if and only if the Busby invariant of the extension is injective.  We now prove in the following proposition that every essential extension $0 \to B \to E \to A \to 0$ with $A$ a non-unital, separable, nuclear C*-algebra and $B$ a C*-algebra that is isomorphic to either $\cK$ or a nuclear, purely infinite simple C*-algebra is absorbing. 

Before proving the proposition, we show that any absorbing extension must be an essential extension.  Hence, the assumption that the extension is essential is necessary.  Note that if $\tau$ or $\tau'$ is injective, then the sum $\tau \oplus \tau'$ is injective.  Since $B$ is stable, there exists a unital embedding from $\mathcal{O}_{2}$ to $\mathcal{M}(B)$ which induces a unital embedding from $\mathcal{O}_{2}$ to $\mathcal{Q}(B)$.  Nuclearity of $A$ gives us an embedding of $A$ into $\mathcal{O}_{2}$, thus the composition gives a trivial essential extension $\ftn{ \tau_{0} }{ A }{ \mathcal{Q}(B) }$.  Therefore, an absorbing extension $\tau$ is an essential extension since $\tau$ is unitarily equivalent to $\tau \oplus \tau_{0}$.
  
\begin{proposition}\label{p:absorbing extensions}
Let $A$ be a non-unital, separable, nuclear C*-algebra and let $B$ be a separable C*-algebra that is isomorphic to either $\cK$ or a nuclear, purely infinite simple C*-algebra.  If $\ftn{ \tau }{A}{ \mathcal{Q} (B) }$ is an essential extension, then for every trivial extension $\ftn{ \tau_{0} }{A}{\mathcal{Q} (B) }$ we have that $\tau \sim_{u} \tau \oplus \tau_{0}$.   Consequently, if $e_{i} : 0 \to B \to E_{i} \to A \to 0$ is an essential extension for $i = 1, 2$ and $[ \tau_{e_{1}} ] = [ \tau_{e_{2}} ]$ in $\mathrm{Ext} ( A, B )$, then $E_{1} \cong E_{2}$.
\end{proposition}

\begin{proof}
Let $\ftn{ \alpha_{0} }{ A }{ \mathcal{M} (B) }$ be a homomorphism with $\tau_{0} = \pi \circ \alpha_{0}$.  Extend $\tau$ and $\alpha_{0}$ to the unitization of $\widetilde{A}$, and denote these extensions by $\ftn{ \widetilde{\tau} }{\widetilde{A} }{ \mathcal{Q} (B) }$ and $\ftn{ \widetilde{ \alpha }_{0} }{ \widetilde{A} } { \mathcal{M} (B) }$ respectively.  

We claim that $\widetilde{\tau}$ is injective.  Let $y \in \ker( \widetilde{\tau} )$.  Then $\tau ( y x ) = \widetilde{\tau} ( y ) \widetilde{ \tau } ( \iota_{A } (x) ) = 0$ and $\tau ( x y ) = \widetilde{ \tau} ( \iota_{A} (x ) ) \widetilde{ \tau } (y) =0$ for all $x \in A$.  Since $\tau$ is injective, we have that $yx = xy = 0$ for all $x \in A$.  Since $A$ is non-unital, $A$ is an essential ideal of $\widetilde{A}$.  Hence, $y = 0$.  This proves our claim.    

Set $E = \pi^{-1} ( \widetilde{\tau}( \widetilde{A} ) ) \subseteq \mathcal{M}(B)$.  Since $\widetilde{\tau}$ is injective, we may define a surjective homomorphism $\ftn{ \psi }{ E }{ \widetilde{A} }$ by $\psi (x) = \widetilde{\tau}^{-1} ( \pi (x) )$.  Define $\ftn{ \eta }{ E }{ \mathcal{M}(B) }$ by $\eta (x ) = \widetilde{\alpha}_{0} \circ \psi(x)$.  Then $\eta$ is a unital homomorphism such that $\eta( E \cap B ) = \{ 0 \}$.  Let $s_{1}$ and $s_{2}$ be isometries such that $1_{ \mathcal{M}(B) } = s_{1} s_{1}^{*} + s_{2} s_{2}^{*}$.  By \cite[Corollary~2]{wa: notes on extensions} and \cite[Proposition~7]{ek: purely infinite simple}, there exists a unitary $u \in \mathcal{M}(B)$ such that $u (s_{1} x s_{1}^{*} + s_{2} \eta(x) s_{2}^{*} ) u^{*} - x \in B$ for all $x \in E$.

We claim $u$ implements a unitary equivalence between $\tau$ and $\tau \oplus \tau_{0}$.  Let $a \in A$.  Choose $x \in E$ such that $\pi(x) = (\widetilde{\tau} \circ \iota_{A})(a)$.  Note that 
\[
\pi \circ \eta (x) = \pi \circ \widetilde{\alpha}_{0} \circ \psi (x) = (\pi \circ \widetilde{\alpha}_{0} )(\widetilde{\tau}^{-1} ( \pi(x) )  )= (\pi \circ \widetilde{\alpha}_{0} \circ \iota_{A} )(a).
\] 
Then 
\begin{align*}
&\pi(u) \left( \tau(a) \oplus \tau_{0} (a)  \right) \pi(u)^{*} \\
&\quad =  \pi( u ) \left( \pi( s_{1} ) \widetilde{ \tau }( \iota_{A}(a) ) \pi( s_{1}^{*} ) + \pi( s_{2}) ( \pi \circ \widetilde{\alpha}_{0} \circ \iota_{A}) (a) \pi( s_{2})^{*} \right)  \pi(u)^{*} \\
&\quad = \pi (u ( s_{1} x s_{1}^{*} + s_{2} \eta(x) s_{2}^{*})u^{*} ) \\
&\quad = \pi(x) \\
&\quad = (\widetilde{\tau} \circ \iota_{A})(a) \\
&\quad = \tau(a). 
\end{align*}
Hence, $\tau \oplus \tau_{0} \sim_{u} \tau$, proving the first part of the proposition.

Suppose $e_{i} : 0 \to B \to E_{i} \to A \to 0$ is an essential extension for $i = 1, 2$ and $[ \tau_{e_{1}} ] = [ \tau_{e_{2}} ]$ in $\mathrm{Ext} ( A, B )$.  By the discussion before the proposition, there exist trivial extensions $\ftn{ \tau_{1}', \tau_{2}' }{A}{\mathcal{Q}(B)}$ such that $\tau_{e_{1}} \oplus \tau_{1}' \sim_{u} \tau_{e_{2}} \oplus \tau_{2}'$.  By the first part of the proposition, we have that $\tau_{e_{1}} \sim_{u} \tau_{e_{1}} \oplus \tau_{1}'$ and $\tau_{e_{2}} \oplus \tau_{2}' \sim_{u} \tau_{e_{2}}'$.  Therefore, $\tau_{e_{1}} \sim_{u} \tau_{e_{2}}$.  By \cite[\S~3.2]{kkjkt:ek}, $E_{1} \cong E_{2}$.
\end{proof}

\subsection{Corners of graph algebras}  We also need some results involving corners of graph algebras.  The general case will be worked out in \cite{seajger:hcgc}.  For the convenience of the reader, we will prove the case that will suit our purposes (see Proposition~\ref{p:  corners of graph algebras}).

Recall that if $E = ( E^{0} , E^{1} , r , s )$ is a graph, the C*-algebra $C^{*} (E)$ associated to $E$ is the universal C*-algebra generated by $\{ p_{v} : v \in E^{0} \} \sqcup \{ s_{e} : e \in E^{1} \}$ subject to the relations
\begin{itemize}
\item[(i)] $p_{v} p_{w} = \delta_{v, w} p_{v}$ for all $v, w \in E^{0}$;

\item[(ii)] $s_{e}^{*} s_{f} = \delta_{e, f} p_{r(e)}$ for all $e, f \in E^{1}$;

\item[(iii)] $s_{e} s_{e}^{*} \leq p_{s(e)}$ for all $e \in E^{1}$; and 

\item[(iv)] $p_{v} = \sum_{ e \in s^{-1} (v) } s_{e} s_{e}^{*}$ for all $v \in E^{0}$ with $0 < | s^{-1} (v) | < \infty$. 
\end{itemize}

A \emph{loop} in $E$ is a path $\alpha = e_{1} \cdots e_{n}$ with $s( e_{1} ) = s( e_{n} )$ and we say that $s(e_{1})$ is the \emph{base point of $\alpha$}.  A \emph{simple loop} in $E$ is a loop $\alpha = e_{1} \cdots e_{n}$ such that $s( e_{i} ) \neq s( e_{j} )$ for $i \neq j$.  We say that $E$ satisfies \emph{Condition~(K)} if every vertex is either the base point of at least two simple loops or is not the base point of a loop.  It is well-known that if $A$ is a Cuntz-Krieger algebra, then $A$ is isomorphic to $C^{*} (E)$, where $E$ is a finite graph with no sinks.  If, in addition, $A$ is purely infinite, then $E$ will also satisfy Condition~(K).

\begin{proposition}\label{p:  corners of graph algebras}
Let $E$ be a graph with finitely many vertices.  Suppose there exists a vertex $w$ in $E$ such that 
\begin{itemize}
\item[(i)] $\{w\}$ is a hereditary and saturated subset of $E^{0}$;

\item[(ii)] $| \{ e \in E^{1} :  s(e) = w \} |$ is either equal to $0$ or $\infty$; 

\item[(iii)] for every $v \in E^{0} \setminus \{w\}$, there are finitely many edges from $v$ to $w$ and there exists at least one $v \in E^{0} \setminus \{w\}$ such that there exists an edge from $v$ to $w$; and 

\item[(iv)] every vertex $v \in E^{0} \setminus \{w\}$ emits finitely many edges and is the base point of at least two loops of length one.
\end{itemize}
Then for every full projection $p \in C^{*} ( E ) \otimes \cK$, we have that $p ( C^{*} (E) \otimes \cK ) p$ is isomorphic to a graph algebra.  Consequently, if $A$ is a unital C*-algebra such that $A \otimes \cK \cong C^{*} (E) \otimes \cK$, then $A$ is isomorphic to a graph algebra.
\end{proposition}

\begin{proof}
Let $\{ e_{ij} \}$ be a system of matrix units for $\cK$.  Throughout the proof, if $p$ is a projection in $C^{*} (E)$ and $n \in \mathbb{N}$, then set $np = \overbrace{p \oplus \cdots \oplus p }^{n}$ in $C^{*} (E) \otimes \cK$.  Let $\{ p_{v} , s_{e} : v \in E^{0} , e \in E^{1} \}$ be a Cuntz-Krieger $E$-family generating $C^{*} (E)$.  Since the only vertex in $E$ that is a singular vertex, i.e., emits no edges or infinitely many edges, is $w$,  by \cite[Theorem~3.4 and Corollary~3.5]{HLMRT: non stable Ktheory}, 
\begin{equation}\label{eq: corner}
p \sim \left( \bigoplus_{v \in S } n_{v} p_{v} \right) \oplus n_{1} \left( p_{w} - \sum_{ e \in T_{1} }  s_{e} s_{e}^{*}  \right) \oplus \cdots \oplus n_{k} \left( p_{w} - \sum_{ e \in T_{k} } s_{e} s_{e}^{*}  \right),
\end{equation}
 where $n_{v} > 0$ for all $v \in S$, $n_{i} \geq 0$ for all $i$, $S \subseteq E^{0} \setminus \{w\}$, and $T_{i}$ is a finite (possibly empty) subset of $s^{-1} ( w )$ for all $i$.  Arguing as in \cite[Lemma~4.6]{seaer: corners of CK algebras}, we have that the projection on the right hand side of \eqref{eq: corner} is Murray-von Neumann equivalent to 
 \[
q = \bigoplus_{ v \in E^{0} } m_{v} p_{v}
\]
where $m_{v} > 0$ for all $v \in E^{0}$.  We use the fact that if $S$ is a finite subset of $s^{-1} (w)$, then 
\[
| S | p_{w} \oplus \left( p_{w} - \sum_{ e \in S } s_{e} s_{e}^{*} \right) \sim \left( \sum_{ e \in S } s_{e} s_{e}^{*} \right) \oplus \left( p_{w} - \sum_{ e \in S } s_{e} s_{e}^{*} \right) \sim p_{w}
\] 
and the fact that if $v_{0} \in E^{0} \setminus \{ w \}$ with $s^{-1} (v_{0} )  \cap r^{-1} (w) \neq \emptyset$, then for any $n$, we have that $p_{v_{0}} \sim n p_{w} \oplus p_{v_{0}} \oplus \left( \bigoplus_{ v \in E^{0} } m_{v}' p_{v} \right)$ for $m_{v}' \geq 0$.  Now, arguing as in  \cite[Proposition~4.7]{seaer: corners of CK algebras}, we have that $q ( C^{*} (E) \otimes \cK ) q$ is isomorphic to a graph algebra.  Since $p \sim q$, we have that $p ( C^{*} (E) \otimes \cK ) p \cong q ( C^{*} (E) \otimes \cK ) q$.  Therefore, $p ( C^{*} (E) \otimes \cK ) p$ is isomorphic to a graph algebra. 

For the last part of the proposition note that $A \cong p ( C^{*} (E) \otimes \cK ) p$, where $p$ is the projection given by the image of $1_{A} \otimes e_{11}$ under some isomorphism from $A \otimes \cK$ to $C^{*} (E) \otimes \cK$.  Since $1_{A} \otimes e_{11}$ is full in $A \otimes \cK$, we have that $p$ is full in $C^{*} (E) \otimes \cK$.
\end{proof}

\subsection{Semigroup C*-algebras and graph algebras}  We now determine when a C*-algebra associated to an Artin monoid is isomorphic to a graph algebra.  To do this, we need to determine when an extension of two graph algebras is isomorphic to a graph algebra.   In spite of substantial effort 
the extension problem for graph algebras has not be completely resolved even for the single non-trivial ideal case.  Moreover, the results in the literature are not sufficient for our purposes.  The following ad hoc result will give us what we need.

\begin{lemma}\label{l: isomorphism ksix}
For each $i$, let $A_{i}$ be a separable, nuclear C*-algebra with an essential ideal $I_{i}$ such that $I_{i}$ is isomorphic to either $\cK$ or a purely infinite simple C*-algebra with trivial $K_{1}$ group, $A_{i} / I_{i}$ satisfies the Universal Coefficient Theorem, and $K_{1} ( A_{i} / I_{i} ) = \gekl{0}$ or $K_{0} (A_{i} / I_{i})$ is a free group (possibly $K_{0} (A_i / I_i) = \gekl{0}$).  Suppose there exist isomorphisms $\ftn{\beta }{ I_{1} \otimes \cK }{ I_{2} \otimes \cK }$, $\ftn{\alpha}{ ( A_{1} / I_{1}) \otimes \cK }{ (A_{2} / I_{2}) \otimes \cK}$, and $\ftn{\eta_{*}}{ K_{*} ( A_{1} \otimes \cK ) }{ K_{*} (A_{2} \otimes \cK) }$ such that $\ftn{ ( K_{*} ( \beta ), \eta_{*} , K_{*} ( \alpha ) ) }{ K_{\mathrm{six} } ( A_{1} \otimes \cK ; I \otimes \cK ) }{ K_{\mathrm{six}} ( A_{2} \otimes \cK ; I_{2} \otimes \cK ) }$ is an isomorphism. Then $A_{1} \otimes \cK \cong A_{2} \otimes \cK$.
\end{lemma}

\begin{proof}
Let $e_{1} : 0 \to I_{2} \otimes \cK \to B_{1} \to (A_{1} / I_{1}) \otimes \cK \to 0$ be the extension obtained by pushing forward the extension $0 \to I_{1} \otimes \cK \to A_{1} \otimes \cK \to (A_{1} / I_{1}) \otimes \cK \to 0$ via the isomorphism $\beta$ and let $e_{2}  : 0 \to I_{2} \otimes \cK \to B_{2} \to (A_{1} / I_{1}) \otimes \cK \to 0$ be the extension obtained by pulling back the extension $0 \to I_{2} \otimes \cK \to A_{2} \otimes \cK \to (A_{2} / I_{2}) \otimes \cK \to 0$ via the isomorphism $\alpha$.  Note that there exist isomorphisms $\ftn{ \phi_{1} }{ A_{1} \otimes \cK }{ B_{1} }$ and $\ftn{ \phi_{2} }{ B_{2} }{ A_{2} \otimes \cK }$ such that $\ftn{ ( K_{*} ( \beta ) , K_{*} ( \phi_{1} ), K_{*} ( \id_{ (A_{1} / I_{1} ) \otimes \cK } ) ) }{ K_{ \mathrm{six}} ( A_{1} \otimes \cK ; I_{1} \otimes \cK ) }{ K_{\mathrm{six}} ( B_{1} ; I_{2} \otimes \cK )}$ and $\ftn{ ( K_{*} ( \id_{ I_{2} \otimes \cK } ), K_{*} ( \phi_{2} ), K_{*} ( \alpha ) ) }{ K_{\mathrm{six}} ( B_{2} ; I_{2} \otimes \cK ) }{ K_{\mathrm{six}} ( A_{2} \otimes \cK ; I_{2} \otimes \cK  ) }$ are isomorphisms.  Then
$$
( K_{*} ( \id_{ I_{2} \otimes \cK } ), K_{*} ( \phi_{2}^{-1} ) \circ\eta_{*} \circ K_{*} ( \phi_{1}^{-1}) , K_{*} ( \id_{ (A_{1} / I_{1} ) \otimes \cK } ) )
$$
is an isomorphism from $K_{\mathrm{six}} ( B_{1} ; I_{2} \otimes \cK )$ to $K_{ \mathrm{six}} ( B_{2} ; I_{2} \otimes \cK )$.

We claim that $[ \tau_{e_{1}} ] = [ \tau_{e_{2}} ]$ in $\mathrm{Ext} ( (A_{1} / I_{1}) \otimes \cK , I_{2} \otimes \cK )$.  Since $A_{1} / I_{1}$ satisfies the Universal Coefficient Theorem, we may identify $\mathrm{Ext} ( (A_{1} / I_{1}) \otimes \cK , I_{2} \otimes \cK )$ with $\kk^{1} ( (A_{1} / I_{1}) \otimes \cK , I_{2} \otimes \cK )$.  Note that $\mathrm{Ext}_{\mathbb{Z}}^{1} ( K_{1} ( (A_{1} / I_{1}) \otimes \cK ) , K_{1} ( I_{2} \otimes \cK ) ) = \gekl{0}$ since $K_{1} ( I_{2} ) = \gekl{0}$.   Suppose $K_{1} ( A_{1} / I_{1} ) = \gekl{0}$.  Then $K_{1} ( \tau_{e_{i}} ) = \gekl{0}$.   Since $K_{1} ( I_{2} ) = \gekl{0}$, we have that $K_{0} ( \tau_{e_{i}} ) = \gekl{0}$.  Thus $K_{*} ( \tau_{e_{i}} ) = \gekl{0}$.  By the Universal Coefficient Theorem, $[ \tau_{e_{i}} ]$ can be identified with the element in $\mathrm{Ext}_{\mathbb{Z}}^{1} ( K_{0} ( (A_{1} / I_{1}) \otimes \cK ) , K_{0} ( I_{2} \otimes \cK ) )$ given by $K_{ \mathrm{six} } ( B_{i} ; I_{2} \otimes \cK )$. As $( K_{*} ( \id_{ I_{2} \otimes \cK } ), K_{*} ( \phi_{2}^{-1} ) \circ\eta_{*} \circ K_{*} ( \phi_{1}^{-1}) , K_{*} ( \id_{ (A_{1} / I_{1}) \otimes \cK } ) )$ is an isomorphism from $K_{\mathrm{six}} ( B_{1} ; I_{2} \otimes \cK )$ to $K_{ \mathrm{six}} ( B_{2} ; I_{2} \otimes \cK )$ we have that $K_{ \mathrm{six} } ( B_{1} ; I_{2} \otimes \cK )$ and $K_{ \mathrm{six} } ( B_{2} ; I_{2} \otimes \cK )$ induce the same element in $\mathrm{Ext}_{\mathbb{Z}}^{1} ( K_{0} ( (A_{1} / I_{1}) \otimes \cK ) , K_{0} ( I_{2} \otimes \cK ) )$. Hence, $[ \tau_{e_{1}} ] = [ \tau_{e_{2}} ]$ in $\mathrm{Ext} ( (A_{1} / I_{1}) \otimes \cK , I_{2} \otimes \cK )$.  Suppose $K_{0} ( A_{1} / I_{1} )$ is a free group (possibly the zero group).  By the Universal Coefficient, $[ \tau_{e_{i}} ]$ is completely determined by $K_{*} ( \tau_{e_{i}} )$. Since $( K_{*} ( \id_{ I_{2} \otimes \cK } ), K_{*} ( \phi_{2}^{-1} ) \circ\eta_{*} \circ K_{*} ( \phi_{1}^{-1}) , K_{*} ( \id_{ (A_{1} / I_{1}) \otimes \cK } ) )$ is an isomorphism from $K_{\mathrm{six}} ( B_{1} ; I_{2} \otimes \cK )$ to $K_{ \mathrm{six}} ( B_{2} ; I_{2} \otimes \cK )$, we have that $K_{*} ( \tau_{e_{1}} ) = K_{*} ( \tau_{e_{2}} )$.  Hence, $[ \tau_{e_{1}} ] = [ \tau_{e_{2}} ]$ in $\mathrm{Ext} ( (A_{1} / I_{1}) \otimes \cK , I_{2} \otimes \cK )$.

In both cases, we have shown that $[ \tau_{e_{1}} ] = [ \tau_{e_{2}} ]$ in $\mathrm{Ext} ( (A_{1} / I_{1} )\otimes \cK , I_{2} \otimes \cK )$, proving our claim.  By Proposition~\ref{p:absorbing extensions}, we have that $B_{1} \cong B_{2}$.  Therefore, $A_{1} \otimes \cK \cong A_{2} \otimes \cK$.
\end{proof}

\begin{lemma}\label{lem:AFOinftyideal}
Let $A$ be a unital, separable, nuclear C*-algebra with an essential ideal $I$ such that $I \cong \cK$ or $I \cong \mathcal{O}_{\infty} \otimes \cK$ and $A / I$ is isomorphic to a purely infinite Cuntz-Krieger algebra.  If $K_{1} ( A / I ) = \gekl{0}$ or $K_{0} (A / I)$ is a free group (possibly $K_{0} (A / I) = \gekl{0}$), then $A$ is isomorphic to a graph algebra.
\end{lemma}

\begin{proof}
By \cite[Theorem~4.4]{abk: range} and \cite[Proposition~8.3]{gr:classification of ck-algebras}, there exists a finite graph $F$ such that each vertex of $F$ is the base point of at least two loop of length one and there exists an isomorphism $\ftn{ \phi }{ C^{*} (F) \otimes \cK }{ A / I \otimes \cK }$.  Let $\ftn{ \psi }{ C^{*} ( G ) \otimes \cK }{ I }$ be an isomorphism such that $K_{*} ( \psi ) = \id$, where $G$ is the graph $\{v\}$ with one vertex and no edges if $I \cong \cK$ and $G$ is the graph with one vertex $\{v\}$ with infinitely many edges when $I \cong \mathcal{O}_{\infty} \otimes \cK$.  By \cite[Lemma~5.4~(r1) and Proposition~5.5]{setkmtjw:rking}, there exists a graph $E$ with the properties that 
\begin{itemize}
\item[(1)] $E^{0} = G^{0} \sqcup F^{0}$,

\item[(2)] $E^{1}$ is the union of $G^{1}$ and $F^{1}$ together with a finite nonzero number of edges from each $w \in F^{0}$ to $v$, and

\item[(3)] there exist an isomorphism $\ftn{ \alpha_{*} }{ K_{*} ( C^{*} (E) ) }{ K_{*} ( A ) }$ with the property  that $( K_{*} ( \psi ), \alpha_{*} , K_{*} ( \phi ) )$ is an isomorphism from $K_{\mathrm{six}} ( C^{*} (E) ; I_{\{v\} } )$ to $K_{\mathrm{six}} ( A; I)$. 
\end{itemize}

Note that $I_{\{v\}} \otimes \cK$ is an essential ideal of $C^{*} (E) \otimes \cK$ and there exist an isomorphism $\ftn{ \overline{\alpha}_{*} }{ K_{*} ( C^{*} (E) \otimes \cK ) }{ K_{*} ( A \otimes \cK ) }$ such that $( K_{*} ( \psi \otimes \id_{\cK} ), \overline{\alpha}_{*} , K_{*} ( \phi \otimes \id_{ \cK} ) )$ is an isomorphism from $K_{\mathrm{six}} ( C^{*} (E)  \otimes \cK ; I_{\{v\} \otimes \cK } )$ to $K_{\mathrm{six}} ( A \otimes \cK ; I \otimes \cK)$.  Also, note that $I \cong I_{\{v\}} = \cK$ or $I \cong I_{\{v\}} \cong \mathcal{O}_{\infty } \otimes \cK$ .  By Lemma~\ref{l: isomorphism ksix}, $A \otimes \cK \cong C^{*} ( E ) \otimes \cK$.  Therefore,  $A$ is isomorphic to a graph algebra by Proposition~\ref{p:  corners of graph algebras}.  
\end{proof}

\begin{lemma}\label{lem:AFidealCKquotient}
For each $m \in \Nz$, for each $n \geq 0$, the smallest nonzero ideal $I$ of $E_{m}^{\pm1} \otimes \bigotimes_{k=1}^{n} E_{2}^{r_{k}} $ is isomorphic to $\cK$ and $( E_{m}^{\pm1} \otimes \bigotimes_{k=1}^{n} E_{2}^{r_{k}} ) / I$ is isomorphic to a Cuntz-Krieger algebra with vanishing $K_{1}$-group.

Consequently, $E_{m}^{\pm1} \otimes \bigotimes_{k=1}^{n} E_{2}^{r_{k}}$ and $E_{m}^{\pm1} \otimes \bigotimes_{k=1}^{n} E_{2}^{r_{k}} \otimes \mathcal{O}_{\infty}$ are isomorphic to graph algebras.
\end{lemma}

\begin{proof}
Note that for each $m \in \Nz$, by \cite[Theorem~7.2]{setkmtjw:rking}, $E_{m}^{+1}$, $E_{m}^{-1}$, $E_{m}^{+1} \otimes \mathcal{O}_{\infty}$, and $E_{m}^{-1} \otimes \mathcal{O}_{\infty}$ are graph algebras with $E_{m}^{\pm 1 } / \cK$ and $(E_{m}^{\pm 1 } \otimes \mathcal{O}_{\infty}) / (\cK \otimes \mathcal{O}_{\infty} ) \cong ( E_{m}^{\pm 1 } / \cK )\otimes \mathcal{O}_{\infty}$ are isomorphic to purely infinite Cuntz-Krieger algebras.  Therefore, we may assume that $n \geq 1$.  

For notational convenience, set $A = E_{m}^{\pm1} \otimes \bigotimes_{k=1}^{n} E_{2}^{r_{k}}$.  Note that $I = \bigotimes_{ k = 1}^{n+1} \cK$.  Let $J = \cK \otimes \bigotimes_{ k = 1}^{n} E_{2}^{ r_{k}}$.  Then $J$ is a primitive ideal and $A / J \cong \mathcal{O}_{m} \otimes \bigotimes_{ k = 1}^{n} E_{2}^{ r_{k}}$.  We will show that $J / I$ is stably isomorphic to an $\mathcal{O}_{2}$-absorbing Cuntz-Krieger algebra, $A / J$ is isomorphic to a Cuntz-Krieger algebra with vanishing boundary maps, and the boundary maps in $K$-theory induced by the extension $0 \to J / I \to A / I \to A / J \to 0$ are zero. 

We will first prove that $J / I$ is $\mathcal{O}_{2}$-absorbing.  Note that it is enough to show that $\left( \bigotimes_{k=1}^{n} E_{2}^{r_{k}} \right) / \left( \bigotimes_{ k = 1}^{n} \cK \right)$ is $\mathcal{O}_{2}$-absorbing since $J / I \cong \cK \otimes \left( \bigotimes_{k=1}^{n} E_{2}^{r_{k}} \right) / \left( \bigotimes_{ k = 1}^{n} \cK \right)$.  Since $E_{2}^{\pm} / \cK \cong \mathcal{O}_{2}$ which is $\mathcal{O}_{2}$-absorbing by \cite[Theorem~3.8]{ekncp:eeccao}, we have that $\left( \bigotimes_{k=1}^{n} E_{2}^{r_{k}} \right) / \left( \bigotimes_{ k = 1}^{n} \cK \right)$ is $\mathcal{O}_{2}$-absorbing for $n = 1$.  Suppose $\left( \bigotimes_{k=1}^{m} E_{2}^{r_{k}} \right) / \left( \bigotimes_{ k = 1}^{m} \cK \right)$ is $\mathcal{O}_{2}$-absorbing for $1 \leq m < n$.  Consider the extension
\begin{align*}
\scalebox{.8}{$0 \to \left( E_{2}^{r_{1}} \otimes  \bigotimes_{k=2}^{n} \cK \right) / \left( \bigotimes_{ k = 1}^{n} \cK \right) \to \left( \bigotimes_{k=1}^{n} E_{2}^{r_{k}} \right) / \left( \bigotimes_{ k = 1}^{n} \cK \right) \to \left( \bigotimes_{k=1}^{n} E_{2}^{r_{k}} \right) / \left(E_{2}^{r_{1}}  \otimes \bigotimes_{ k = 2}^{n} \cK \right) \to 0$}.
\end{align*}
Now,  $( E_{2}^{r_{1}} \otimes  \bigotimes_{k=2}^{n} \cK ) / ( \otimes_{ k = 1}^{n} \cK ) \cong ( E_{2}^{r_{1}} / \cK ) \otimes \bigotimes_{k=2}^{n} \cK \cong \mathcal{O}_{2} \otimes \bigotimes_{k=2}^{n} \cK$ which is $\mathcal{O}_{2}$-absorbing by \cite[Theorem~3.8]{ekncp:eeccao}.  Since $(\bigotimes_{k=1}^{n} E_{2}^{r_{k}} ) / (E_{2}^{r_{1}}  \otimes \bigotimes_{ k = 2}^{n} \cK ) \cong E_{2}^{r_{1}} \otimes \left( \left(\bigotimes_{k=2}^{n} E_{2}^{r_{k}} \right) / \left( \bigotimes_{k=2}^{n} \cK \right) \right)$ and because of the inductive hypothesis, we have that $(\otimes_{k=1}^{n} E_{2}^{r_{k}} ) / (E_{2}^{r_{1}}  \otimes \bigotimes_{ k = 2}^{n} \cK )$ is $\mathcal{O}_{2}$-absorbing. Hence, by \cite[Theorem~3.8]{ekncp:eeccao} and \cite[Corollary~4.3]{atww:ssc}, $\left( \bigotimes_{k=1}^{n} E_{2}^{r_{k}} \right) / \left( \bigotimes_{ k = 1}^{n} \cK \right)$ is $\mathcal{O}_{2}$-absorbing.  This proves our claim.  

Since $J / I$ is $\mathcal{O}_{2}$-absorbing and $J/ I$ has finitely many ideals, by \cite{ek:nkmkna}, $J / I$ is stably isomorphic to a Cuntz-Krieger algebra with vanishing boundary maps.  This is because for any finite $T_{0}$-space $X$, there exists an $\mathcal{O}_{2}$-absorbing Cuntz-Krieger algebra with primitive ideal space $X$.  We also note that the boundary maps in $K$-theory induced by the extension $0 \to J / I \to A / I \to A / J \to 0$ are zero since $K_{*} (J/I) = \gekl{0}$.

We now show that $A / J$ is isomorphic to a Cuntz-Krieger algebra with vanishing boundary maps.  Recall that $A / J \cong \mathcal{O}_{m} \otimes \bigotimes_{ k = 1}^{n} E_{2}^{ r_{k}}$.  Hence, every simple sub-quotient of $A / J$ is isomorphic to $\mathcal{O}_{m} \otimes (\mathcal{I}_{2} / \mathcal{I}_{1} )$ where $\mathcal{I}_{1}, \mathcal{I}_{2}$ are ideals of $\bigotimes_{ k = 1}^{n} E_{2}^{ r_{k}}$ with $\mathcal{I}_{1} \subseteq \mathcal{I}_{2}$ and $\mathcal{I}_{2} / \mathcal{I}_{1}$ simple.  Note that if $\mathcal{I}_{1}, \mathcal{I}_{2}$ are ideals of $\bigotimes_{ k = 1}^{n} E_{2}^{ r_{k}}$ with $\mathcal{I}_{1} \subseteq \mathcal{I}_{2}$ and $\mathcal{I}_{2} / \mathcal{I}_{1}$ simple, then $\mathcal{I}_{2} / \mathcal{I}_{1} \cong \bigotimes_{k = 1}^{n} B_{k}$ where $B_{k}$ is a simple sub-quotient of $E_{2}^{r_{k}}$.  Hence, every simple sub-quotient of $\bigotimes_{ k = 1}^{n} E_{2}^{ r_{k}}$ is either isomorphic to $\bigotimes_{k =1}^{n} \cK$ or is $\mathcal{O}_{2}$-absorbing.  Hence, every simple sub-quotient of $\mathcal{O}_{m} \otimes \bigotimes_{ k = 1}^{n} E_{2}^{ r_{k}}$ is either stably isomorphic to $\mathcal{O}_{m}$ or $\mathcal{O}_{2}$.  So every simple sub-quotient of $A / J$ is stably isomorphic to a Cuntz-Krieger algebra.  Consider the extension $e : 0 \to \mathcal{O}_{m } \otimes \mathcal{I}_{1} \to \mathcal{O}_{m} \otimes \mathcal{I}_{2} \to \mathcal{O}_{m} \otimes ( \mathcal{I}_{2} / \mathcal{I}_{1} ) \to 0$  with $\mathcal{I}_{2} / \mathcal{I}_{1}$ simple.  If $\mathcal{I}_{1} = \gekl{0}$ and $\mathcal{I}_{2} = \bigotimes_{k =1}^{n} \cK$, then $ \mathcal{O}_{m} \otimes \mathcal{I}_{1} = \gekl{0}$ which implies that $e$ has vanishing boundary maps.  If $\mathcal{I}_{2} / \mathcal{I}_{1}$ is $\mathcal{O}_{2}$-absorbing, then $K_{*} ( \mathcal{O}_{m} \otimes ( \mathcal{I}_{2} / \mathcal{I}_{1} ) ) = \gekl{0}$ which also implies that $e$ has vanishing boundary maps.  By \cite[Corollary~3.6]{rb:kxrrzic}, we have that $A / J \cong \mathcal{O}_{m} \otimes \bigotimes_{ k = 1}^{n} E_{2}^{ r_{k}}$ has vanishing boundary maps.  Therefore, by \cite[Corollary~8.2]{rb:kxrrzic}, $A / J$ is isomorphic to a Cuntz-Krieger algebra with vanishing boundary maps.  This finishes the proof of the above claim.

The above claim shows that all the assumptions in \cite[Proposition~3.7, Proposition~3.10, and  Corollary~8.4]{rb:kxrrzic} are satisfied.  Thus, $A /I$ is isomorphic to a purely infinite Cuntz-Krieger algebra.  

We now show that $K_1(( E_{m}^{\pm1} \otimes \bigotimes_{k=1}^{n} E_{2}^{r_{k}} ) / I) \cong \gekl{0}$.  Since $0 \to J / I \to A / I \to A / J \to 0$ has vanishing boundary maps and $J / I$ is $\mathcal{O}_{2}$-absorbing, we have that the surjective map $A / I \to A / J$ induces an injective map $K_{1} ( A / I ) \to K_{1} ( A / J )$.  Since every simple sub-quotient of $A / J$ is stably isomorphic to $\mathcal{O}_{m}$ or $\mathcal{O}_{2}$ and since $A / J$ has finitely many ideals, one can show that $K_{1} ( A / J ) = \{0\}$.  Therefore, $K_{1} ( A/ I ) = \{ 0 \}$. 

Lemma~\ref{lem:AFOinftyideal} implies that $E_{m}^{\pm1} \otimes \bigotimes_{k=1}^{n} E_{m_{k}}^{r_{k}}$ and $E_{m}^{\pm1} \otimes \bigotimes_{k=1}^{n} E_{m_{k}}^{r_{k}} \otimes \mathcal{O}_{\infty}$ are isomorphic to graph algebras.
\end{proof}

\begin{lemma}\label{lem:MEgraph}
Let $m_{1}, m_{2}, \dots m_{n} \in \Nz$.  Then
\begin{itemize}
\item[(1)]  $\bigotimes_{k=1}^{n} E_{ m_{k} }^{ \pm 1}$ is stably isomorphic to unital graph algebra if and only if whenever there exists an $i$ such that $m_{i} \in \{ 1 \} \sqcup \Zz_{ \geq 3}$, we have that $m_{j} = 2$ for all $j \neq i$.

\item[(2)] $\bigotimes_{k=1}^{n} E_{ m_{k} }^{ \pm 1} \otimes \mathcal{O}_{\infty}$ is stably isomorphic to unital graph algebra if and only if whenever there exists an $i$ such that $m_{i} \in \{ 1 \} \sqcup \Zz_{ \geq 3}$, we have that $m_{j} = 2$ for all $j \neq i$.
\end{itemize}
\end{lemma}

\begin{proof}
We prove (1).  (2) is proved in a similar way.

Suppose whenever there exists an $i$ such that $m_{i} \in  \{ 1 \} \sqcup \Zz_{ \geq 3}$, we have that $m_{j} = 2$ for all $j \neq i$.  By Lemma~\ref{lem:AFidealCKquotient}, $\bigotimes_{k=1}^{n} E_{ m_{k} }^{ \pm 1}$ is isomorphic to a graph algebra.  So it is also stably isomorphic to a unital graph algebra.  

Suppose $\bigotimes_{k=1}^{n} E_{ m_{k} }^{ \pm 1}$ is stably isomorphic to a graph algebra.  Note that $E_{m}^{\pm 1} \otimes \cK \cong E_{m}^{+1} \otimes \cK$ for any $m$.  Therefore, it is enough to treat the case $\bigotimes_{k=1}^{n} E_{ m_{k} }^{ + 1}$.  Note that $\bigotimes_{k=1}^{n} E_{ m_{k} }^{ +1}$ has finitely many ideals.  Since $\bigotimes_{k=1}^{n} E_{ m_{k} }^{ +1}$ is stably isomorphic to a unital graph algebra $C^{*} (E)$, we have that $C^{*} (E)$ has finitely many ideals.  Therefore, every sub-quotient of $C^{*} (E)$ is stably isomorphic to a unital graph algebra with finitely many ideals.  Consequently, every sub-quotient of $\bigotimes_{k=1}^{n} E_{ m_{k} }^{ + 1}$ is stably isomorphic to a unital graph algebra with finitely many ideals.  

Suppose there exists $i$ and $j$ such that $m_{i} , m_{j} \in \{ 1 \} \sqcup \Zz_{ \geq 3}$.  Let $I = \bigotimes_{k = 1}^{n} I_{k}$ be the ideal of $\bigotimes_{k=1}^{n} E_{ m_{k} }^{ + 1}$ where $I_{k} = \cK$ if $k \notin \{ i , j \}$, $I_{ i } = E_{ m_{i}}^{+1}$, and $I_{ j } = E_{ m_{j} }^{+1}$.  From the above observation we must have that every sub-quotient of $I$ is stably isomorphic to a unital graph algebra with finitely many ideals.  Note that $I$ is stably isomorphic to $E_{m_{i}}^{+1} \otimes E_{m_{j}}^{+1}$ and $E_{m_{i}}^{+1} \otimes E_{m_{j}}^{+1}$ has a quotient isomorphic to $\mathcal{O}_{m_{i} } \otimes \mathcal{O}_{m_{j}}$.  Therefore, $\mathcal{O}_{m_{i} } \otimes \mathcal{O}_{m_{j}}$ is stably isomorphic to a graph algebra.  

Let $\cK \otimes \cK$ be the smallest non-zero ideal of $E_{m_{i}}^{+1} \otimes E_{m_{j}}^{+1}$.  By the K\"{u}nneth formula, $K_{0} ( E_{m_{i}}^{+1} \otimes E_{m_{j}}^{+1} ) \cong \Zz$ and $K_{1} ( E_{m_{i}}^{+1} \otimes E_{m_{j}}^{+1} ) = \gekl{0}$, and hence the extension $0 \to \cK \otimes \cK \to E_{m_{i}}^{+1} \otimes E_{m_{j}}^{+1} \to (E_{m_{i}}^{+1} \otimes E_{m_{j}}^{+1}) / (\cK \otimes \cK) \to 0$ induces a six-term exact sequence in $K$-theory of the form 
\begin{align*}
\xymatrix{ \Zz \ar[r] & \Zz \ar[r] & K_{0}( (E_{m_{i}}^{+1} \otimes E_{m_{j}}^{+1}) / (\cK \otimes \cK) ) \ar[d] & \\
  K_{1}((E_{m_{i}}^{+1} \otimes E_{m_{j}}^{+1}) / (\cK \otimes \cK) ) \ar[u] & 0 \ar[l] & 0. \ar[l] }
\end{align*}
In particular, $K_{0}((E_{m_{i}}^{+1} \otimes E_{m_{j}}^{+1}) / (\cK \otimes \cK) )$ and $K_{1}((E_{m_{i}}^{+1} \otimes E_{m_{j}}^{+1}) / (\cK \otimes \cK) )$ are cyclic groups.  

Since $(E_{m_{i}}^{+1} \otimes E_{m_{j}}^{+1}) / (\cK \otimes \cK)$ is stably isomorphic to a graph algebra with finitely many ideals, $(E_{m_{i}}^{+1} \otimes E_{m_{j}}^{+1}) / (\cK \otimes \cK)$ has real rank zero.  Therefore, the quotient of $(E_{m_{i}}^{+1} \otimes E_{m_{j}}^{+1}) / (\cK \otimes \cK)$ by the ideal $(\cK \otimes E_{m_{j}}^{+1} + E_{m_{i}}^{+1} \otimes \cK) / ( \cK \otimes \cK )$ induces the following six-term exact sequence 
\begin{align}\label{eq:MEgraph}
\scalebox{.9}{$\vcenter{\xymatrix{ K_{0} ( \mathcal{O}_{m_{i}} ) \oplus K_{0} ( \mathcal{O}_{m_{j}} ) \ar[r] & K_{0}( (E_{m_{i}}^{+1} \otimes E_{m_{j}}^{+1}) / (\cK \otimes \cK) ) \ar[r] & K_{0} ( \mathcal{O}_{m_{i} } \otimes \mathcal{O}_{m_{j}} ) \ar[d]^{0} \\
K_{1} ( \mathcal{O}_{m_{i} } \otimes \mathcal{O}_{m_{j}} ) \ar[u] & K_{1}( (E_{m_{i}}^{+1} \otimes E_{m_{j}}^{+1}) / (\cK \otimes \cK) ) \ar[l] & K_{1} ( \mathcal{O}_{m_{i}} ) \oplus K_{1} ( \mathcal{O}_{m_{j}} ). \ar[l] }}$}
\end{align}
Using the K\"{u}nneth formula, we get 
\begin{align*}
K_{0} ( \mathcal{O}_{m_{i} } \otimes \mathcal{O}_{m_{j}} ) = K_{1} ( \mathcal{O}_{m_{i} } \otimes \mathcal{O}_{m_{j}} ) &= 
\begin{cases}
\Zz_{ \gcd ( m_{i} -1 , m_{j} - 1) } &\text{if $m_{i} , m_{j} \geq 3$} \\
K_{1} ( \mathcal{O}_{m_{i}} ) \oplus K_{0} ( \mathcal{O}_{m_{i}} ) &\text{if $m_{j} =1$} \\
K_{1} ( \mathcal{O}_{m_{j}} ) \oplus K_{0} ( \mathcal{O}_{m_{j}} ) &\text{if $m_{i} =1$}.
\end{cases}
\end{align*}
Since $\mathcal{O}_{m_{i}} \otimes \mathcal{O}_{m_{j}}$ is stably isomorphic to a unital graph algebra, $\gcd ( m_{i} -1 , m_{j} - 1) = 1$ if $m_{i} , m_{j} \geq 3$ and $m_{i} =1$ if and only if $m_{j} =1$.  

Suppose $m_{i} , m_{j} \geq 3$. Exactness of Diagram~(\ref{eq:MEgraph}) implies that $K_{0}( (E_{m_{i}}^{+1} \otimes E_{m_{j}}^{+1}) / (\cK \otimes \cK) ) \cong K_{0} ( \mathcal{O}_{m_{i}} ) \oplus K_{0} ( \mathcal{O}_{m_{j}} ) \cong \Zz_{ m_{i}-1} \oplus \Zz_{ m_{j} - 1}$ which contradicts the fact that $K_{0}(E_{m_{i}}^{+1} \otimes E_{m_{j}}^{+1}) / (\cK \otimes \cK) )$ is a cyclic group.  

Suppose $m_{i}  =1$.  Then $m_{j} = 1$.  Then by the exactness of Diagram~(\ref{eq:MEgraph}), $K_{1}( (E_{m_{i}}^{+1} \otimes E_{m_{j}}^{+1}) / (\cK \otimes \cK) )$ has a sub-group isomorphic to $K_{1} ( \mathcal{O}_{m_{i}} ) \oplus K_{1} ( \mathcal{O}_{m_{j}} ) \cong \Zz \oplus \Zz$.  This cannot happen since $K_{1}( (E_{m_{i}}^{+1} \otimes E_{m_{j}}^{+1}) / (\cK \otimes \cK) )$ is a cyclic group. 
\end{proof}

Let the notation be as in Definition~\ref{not}.

\begin{theorem}\label{isgraph}
Let $\Gamma$ be a countable graph.  Then $C^{*} ( A_{ \Gamma}^{+} )$ is isomorphic to a graph algebra if and only if one of the following holds
\begin{enumerate}
\item $t( \Gamma ) =1$, $o( \Gamma )=0$ and  $N_{k} ( \Gamma ) = 0$ for all $k$
\item $t(\Gamma)=0$, $N_{-1} ( \Gamma )+N_1(\Gamma) < \infty$ and
\[
\sum_{|k|\not=1}N_k(\Gamma)\leq1
\]
\end{enumerate}
\end{theorem}

\begin{proof}
Suppose there exists an isomorphism $\ftn{ \psi }{ C^{*} ( A_{ \Gamma }^{+} ) }{  C^{*} (E) }$ for some countable directed graph $E$.  Since $C^{*} ( A_{ \Gamma }^{+} )$ is unital, $C^{*} (E)$ is unital.  Let $\Gamma_{i} = ( V_{i} , E_{i} )$ be the co-irreducible components of $\Gamma$.  To prove (1), let $I$ be the ideal of $C^{*} ( A_{ \Gamma }^{+})$ generated by $\{ \bigotimes_{j} J_{ij} \}_{i}$ where $J_{ij} = C^{*} ( A_{ \Gamma_{j} }^{+} )$ if $j \neq i$ and 
\begin{align*}
J_{ii} = 
\begin{cases}
\cK &\text{if $1 < | V_{i} | < \infty$} \\
0 &\text{otherwise}.
\end{cases}
\end{align*}
Then $C^{*} ( A_{ \Gamma }^{+} ) / I \cong \bigotimes_{i} C^{*} ( A_{ \Gamma_{i} }^{+} ) / J_{ii}$ where $C^{*} ( A_{ \Gamma_{i}}^{+} ) / J_{ii}$ is a Kirchberg algebra if $|V_{i}| \geq 2$ and $C^{*} ( A_{ \Gamma_{i}}^{+} ) / J_{ii} \cong \mathcal{T}$ otherwise.  In particular, 
\begin{align*}
\mathrm{Prim} ( C^{*} ( A_{ \Gamma }^{+} ) / I  ) &\cong \begin{cases}
 \prod_{ k = 1}^{ t( \Gamma )  } \mathrm{Prim} ( \mathcal{T} ) &\text{if there exists $i$ with $| V_{i} | = 1$} \\
\{ \bullet \} &\text{otherwise}.
\end{cases}
\end{align*}

Note that $I$ is generated by projections.  Therefore, $\psi ( I )$ is generated by projections and hence is a gauge-invariant ideal of $C^{*} (E)$.  Hence, by \cite[Corollary~3.5 and Theorem~3.6]{tbdpirws:crg}, $C^{*} ( E ) / \psi (I)$ is isomorphic to a graph algebra.  Since $C^{*} ( A_{ \Gamma }^{+} ) / I \cong C^{*} ( E ) / \psi (I)$, we have that $C^{*} ( A_{ \Gamma }^{+} ) / I$ is isomorphic to a unital graph algebra.  Note that $C^{*} (A_{ \Gamma}^{+} ) / I$ is $\mathcal{O}_{\infty}$-absorbing (if there exists $i$ such that $| V_{i} | \geq 2$) or $C^{*} (A_{ \Gamma}^{+} ) / I \cong \bigotimes_{ k = 1}^{ t( \Gamma ) } \mathcal{T}$.  

Suppose $C^{*} (A_{ \Gamma}^{+} ) / I$ is $\mathcal{O}_{\infty}$-absorbing.  Since any unital $\mathcal{O}_{\infty}$-absorbing graph algebra has a finite primitive ideal space, we must have that $t( \Gamma ) = 0$.  Suppose $C^{*} (A_{ \Gamma}^{+} ) / I$ is not $\mathcal{O}_{\infty}$-absorbing.  Then $C^{*} (A_{ \Gamma}^{+} ) / I \cong \bigotimes_{ k = 1}^{ t( \Gamma ) } \mathcal{T}$.  Let $J$ be the ideal generated by $\{ \bigotimes_{j} J_{ij} \}_{i}$ where $J_{ij} = \mathcal{T}$ if $j \neq i$ and $J_{ii} = \cK$, then $J$ is an ideal generated by projections such that $\left( \bigotimes_{ k = 1}^{ t( \Gamma ) } \mathcal{T} \right) / J \cong C( \mathbb{T}^{t( \Gamma ) } )$.  Since $C^{*} (A_{ \Gamma}^{+} ) / I$ is isomorphic to a graph algebra and every ideal generated by projections in a graph algebra is gauge invariant, by \cite[Corollary~3.5 and Theorem~3.6]{tbdpirws:crg} every quotient of $C^{*} (A_{ \Gamma}^{+} ) / I$ by an ideal generated by projections is isomorphic to a graph algebra.  Hence, $C( \mathbb{T}^{t( \Gamma ) } ) \cong\left( \bigotimes_{ k = 1}^{ t( \Gamma ) } \mathcal{T} \right) / J$ is isomorphic to a unital graph algebra.  Since the only unital commutative graph algebra is isomorphic to finite direct sums of $\Cz$ and $\mathbb{T}$, we must have that $t( \Gamma ) =1$.

In both cases, we have shown that $t( \Gamma ) \leq 1$.  Suppose $o ( \Gamma ) \neq 0$ or $N_{k} ( \Gamma ) \neq 0$ for some $k$, then there exists an $i$ such that $C^{*} ( A_{ \Gamma_{i}}^{+} ) / J_{ii}$  is a Kirchberg algebra.  Hence, by \cite[Theorem~3.15]{ekncp:eeccao} and \cite[Corollary~3.4]{atww:ssc} $C^{*} ( A_{ \Gamma }^{+} ) / I \cong \bigotimes_{i} C^{*} ( A_{ \Gamma_{i} }^{+} ) / J_{ii}$ is an $\mathcal{O}_{\infty}$-absorbing C*-algebra.  Since every unital graph algebra that is $\mathcal{O}_{\infty}$-absorbing must have finitely many ideals and since
\begin{align*}
\mathrm{Prim} ( C^{*} ( A_{ \Gamma }^{+} ) / I  ) &\cong \begin{cases}
 \prod_{ k = 1}^{ t( \Gamma ) } \mathrm{Prim} ( \mathcal{T} ) &\text{if there exists $i$ with $| V_{i} | = 1$} \\
\{ \bullet \} &\text{otherwise},
\end{cases}
\end{align*}
we have that $t( \Gamma ) = 0$.  Hence, we only get a graph algebra in the case $t(\Gamma)=1$ when all other data vanish.

Suppose $t( \Gamma ) = 0$.  Note that $1 < | V_{i} |$ for all $i$.  Thus, $C^{*} ( A_{ \Gamma_{i} }^{+} )$ is a unital properly infinite C*-algebra, and $\mathrm{Prim} ( C^{*} ( A_{ \Gamma_{i} }^{+} ) ) = \{ x_{i} , y_{i} \}$ with open sets $\{ \emptyset , \{ x_{i} \} , \{ x_{i} , y_{i} \} \}$ when $| V_{i} | < \infty$ and $\mathrm{Prim} ( C^{*} ( A_{ \Gamma_{i} }^{+} ) ) \cong \{ \bullet \}$ when $| V_{i} | = \infty$.  

We claim that $| \menge{k}{N_{k} ( \Gamma ) \neq 0} | < \infty$ and $N_{k}( \Gamma ) < \infty$ for all $k$.  Suppose first $| \menge{k}{N_{k} ( \Gamma ) \neq 0} | = \infty$ or $N_{k}( \Gamma ) = \infty$ for some $k$.  Then $C^{*} ( A_{ \Gamma}^{+} ) \cong \bigotimes_{ i = 1}^{\infty} C^{*} ( A_{ \Gamma_{i} }^{+} )$ and $C^{*} ( A_{ \Gamma}^{+} )$ has infinitely many ideals.  By Lemma \ref{infirrcomp-pi}
, $C^{*} ( A_{ \Gamma}^{+} )$ is $\mathcal{O}_{\infty}$-absorbing.  Again, using the fact that a unital graph algebra that is $\mathcal{O}_{\infty}$-absorbing has finitely many ideals, we have a contradiction.  Therefore, $| \menge{k}{N_{k} ( \Gamma ) \neq 0} | < \infty$ and $N_{k}( \Gamma ) < \infty$ for all $k$, proving the claims in (2).

Note that 
\begin{align*}
C^{*} ( A_{ \Gamma }^{+} ) \cong ( E_{1}^{0} )^{ \otimes N_{0} ( \Gamma ) } \otimes \bigotimes_{ n = 1}^{\infty}  ( E_{1+n}^{-1})^{ \otimes N_{-n} (\Gamma) } \otimes \bigotimes_{ n = 1}^{\infty}  ( E_{1+n}^{+1})^{ \otimes N_{n} (\Gamma) } \otimes ( \mathcal{O}_{\infty} )^{ \otimes o ( \Gamma ) }.
\end{align*}
By Lemma~\ref{lem:MEgraph}, (1) and (2) hold.

In the other direction, we have in case (1) that   $C^{*} ( A_{ \Gamma }^{+} ) \cong \mathcal{T}$ which is isomorphic to a graph algebra.  And in case (2) we have that   either
\begin{align*}
C^{*} ( A_{ \Gamma }^{+} ) \cong ( E_{2}^{-1} )^{ \otimes N_{-1} ( \Gamma ) } \otimes ( E_{2}^{+1} )^{ \otimes N_{1} ( \Gamma ) } \otimes ( \mathcal{O}_{\infty} )^{\otimes o ( \Gamma ) },
\end{align*}
\begin{align*}
C^{*} ( A_{ \Gamma }^{+} ) \cong E_{m}^{+1} \otimes ( E_{2}^{-1} )^{ \otimes N_{-1} ( \Gamma ) } \otimes ( E_{2}^{+1} )^{ \otimes N_{1} ( \Gamma ) } \otimes ( \mathcal{O}_{\infty} )^{\otimes  o ( \Gamma ) }
\end{align*}
for some $m \neq 2$, or 
\begin{align*}
C^{*} ( A_{ \Gamma }^{+} ) \cong E_{m}^{-1} \otimes ( E_{2}^{-1} )^{ \otimes N_{-1} ( \Gamma ) } \otimes ( E_{2}^{+1} )^{ \otimes N_{1} ( \Gamma ) } \otimes ( \mathcal{O}_{\infty} )^{ \otimes o ( \Gamma ) }
\end{align*}
for some $m \neq 2$.  If $ o ( \Gamma ) \geq 1$, then by \cite[Theorem~3.15]{ekncp:eeccao}, $( \mathcal{O}_{\infty} )^{ \otimes o ( \Gamma ) } \cong \mathcal{O}_{\infty}$.  Hence, by Lemma~\ref{lem:AFidealCKquotient}, $C^{*} ( A_{ \Gamma }^{+} )$ is isomorphic to a graph algebra.
\end{proof} 

\begin{remark}
The relation between a (undirected, loop-free) graph $\Gamma$ and a directed graph $G_\Gamma$ with $C^{*} ( A_{ \Gamma }^{+} )\cong C^*(G_\Gamma)$ is somewhat opaque, although the proof given above is in principle constructive. In Figure \ref{somegraphs} below we present eight graphs presenting the C*-algebras given by five-vertex graphs of Figure \ref{fivecases} in the unshaded regions.
\end{remark}

We conclude by establishing semiprojectivity and non-semiprojectivity of $C^* (A_{ \Gamma }^{+} )$ in a number of cases, covering for instance all graphs with 5 or fewer vertices. We note, however, that this theorem does not contain a full answer to the question of which of the C*-algebras under study are semiprojective. The most basic open case has $N_{-2}=2$ and may be represented by a graph with 6 vertices.

\begin{theorem}\label{spnsp}\mbox{}
\begin{enumerate}
\item When $t(\Gamma)>1$, $C^* (A_{ \Gamma }^{+} )$ is not semiprojective.
\item When $t(\Gamma)=1$, $C^*(A_{ \Gamma }^{+} )$ is semiprojective if and only if
\[
o(\Gamma)=\sum_k N_k(\Gamma)=0.
\]
\item When $t(\Gamma)=0$, $C^* (A_{ \Gamma }^{+} )$ is semiprojective when $N_{-1} ( \Gamma )+N_{1} ( \Gamma ) < \infty$ and
\[
\sum_{|k|\not=1}N_k(\Gamma)\leq1.
\]
\end{enumerate}
\end{theorem}

\begin{proof}
We first note that by \cite[Corollary 4.4.16]{de:sccsc},  a C*-algebra of the form $A\otimes \cT$ with $A$ unital, nuclear, infinite-dimensional and in the UCT-class can never be semiprojective. This proves (1) and (2) since $\cT$ itself is trivially semiprojective.

For (3), we first apply Theorem \ref{isgraph} to see that $C^* (A_{ \Gamma }^{+} )$ in this case is a unital graph algebra. We have seen that when $o(\Gamma)>0$, $C^*(A_{ \Gamma }^{+} )$ is strongly purely infinite, and when $o(\Gamma)=0$, there is a minimal ideal $\cK$ in  $C^*(A_{ \Gamma }^{+} )$ so that $C^*(A_{ \Gamma }^{+} )/\cK$ is strongly purely infinite. In either case, \cite{setk:spipgc} applies to guarantee that the C*-algebra is semiprojective.
\end{proof}

\vspace{2cm}

\begin{figure}[h]
\begin{center}
\begin{tabular}{ccccc}
$\xymatrix{
\bullet\ar@(dr,dl)[]\ar@(d,l)\ar@(dl,ul)\ar@(l,u)\ar@(ul,ur)[]\ar[r]\ar@/^/[r]\ar@/^12pt/[r]\ar@/_12pt/[r]\ar@/_/[r]&\circ}$
&&
$\xymatrix{
\bullet\ar@(dr,dl)[]\ar@(d,l)\ar@(l,u)\ar@(ul,ur)[]\ar@/^/[r]\ar@/^12pt/[r]\ar@/_12pt/[r]\ar@/_/[r]&\circ}$&&
$\xymatrix{
\bullet\ar@(d,l)\ar@(dl,ul)\ar@(l,u)\ar[r]\ar@/^/[r]\ar@/_/[r]&\circ}$\\[0.5cm]
$N_{-4}=1$&&$N_{-3}=1$&&$N_{-2}=1$\\[1cm]
$\xymatrix{&\\
\bullet\ar@(d,l)\ar@(dl,ul)\ar@/^/[r]\ar@/_/[r]&\circ}$&&
$\xymatrix{
\bullet\ar[d]&\\
\bullet\ar@/_/[d]\ar@/^/[d]\ar@(l,u)[]\ar@(ld,lu)[]\ar@/^/[r]&\bullet\ar@(d,r)[]\ar@(rd,ru)[]\ar@/^/[l]\\\circ}$
&&
$\xymatrix{\bullet\ar[d]&\\
\bullet\ar@(d,l)\ar@(dl,ul)\ar[r]&\circ}$\\[0.5cm]
$N_{-1}=1$&&$N_0=1$&&$N_1=1$\\[1cm]
$\xymatrix{&\bullet\ar@(r,u)[]\ar@(ul,ur)\ar[dr]\ar[dl]&\\\bullet\ar[dr]\ar@(l,u)[]\ar@(ul,ur)[]&&\bullet\ar[dl]\ar@(r,u)[]\ar@(ul,ur)[]\\&\circ}$&&&&
$\xymatrix{&\bullet\ar@(r,u)[]\ar@(ul,ur)\ar[dr]\ar[dl]&\bullet\ar[d]\\\bullet\ar[dr]\ar@(l,u)[]\ar@(dl,ul)[]&&\bullet\ar[dl]\ar@(r,u)[]\ar@(dr,ur)[]\ar@(d,r)[]\\&\circ}$\\$N_{-1}=2$&&&&$N_{-1}=N_{-2}=1$
\end{tabular}
\end{center}
\caption{Graphs representing cases from Figure \ref{fivecases}}\label{somegraphs}
\end{figure}
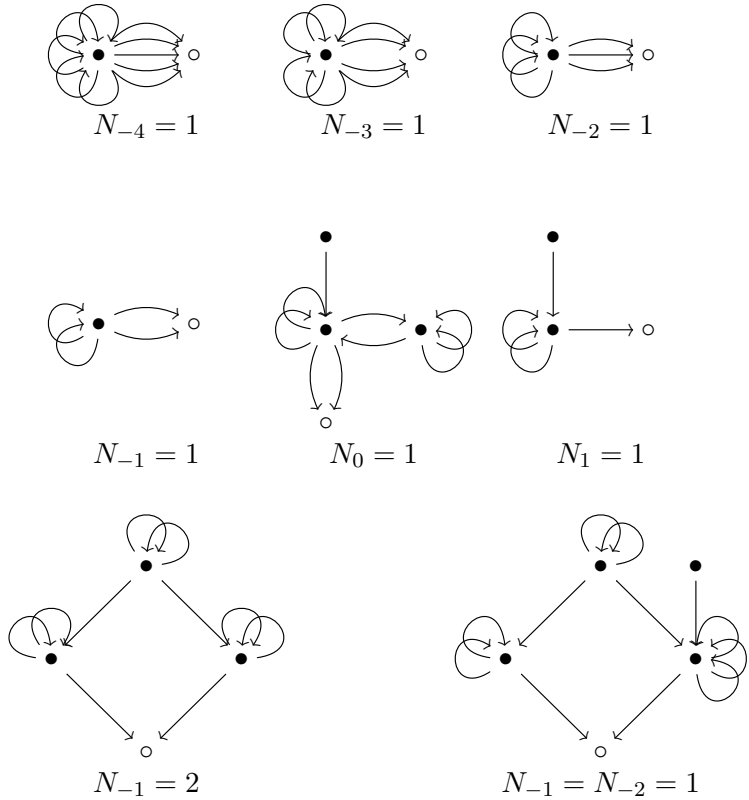

\section{Acknowledgements}
All authors acknowledge support by the Danish National Research Foundation through the Centre for Symmetry and Deformation (DNRF92), and thank the Department of Mathematical Sciences at the University of Copenhagen, where the initial phases of this work were carried out, for providing excellent facilities. The first named author gratefully acknowledges support from the VILLUM Foundation, and the third named author gratefully acknowledges support from the Simons Foundation (\#279369 to Efren Ruiz).


\begin{thebibliography}{BPRS00}

\bibitem[AD]{yadd:hpsgp}
Y.~Antolin and D.~{Dreesen}, \emph{The {Haagerup} property is stable under
  graph products}, arXiv:1305.6748.

\bibitem[ABK14]{abk: range}
S.E. Arklint, R.~Bentmann, T.~Katsura, \emph{The {K}-theoretical range of {C}untz-{K}rieger algebras}, J. Funct. Anal. \textbf{266} (2014), no. 8, 5448--5466.

\bibitem[AGR]{seajger:hcgc}
S.E. {Arklint}, J.~{Gabe}, and E.~{Ruiz}, \emph{Hereditary {C*}-subalgebras
  of graph {C*}-algebras}, in preparation.

\bibitem[AR]{seaer: corners of CK algebras}
S.E. {Arklint} and E.~{Ruiz}, \emph{Corners of {C}untz-{K}rieger algebras},  Trans. Amer. Math. Soc., to appear, arXiv:1209.4336.
 
\bibitem[Arv77]{wa: notes on extensions}
W.~Arveson, \emph{Notes on extensions {C*}-algebras}, Duke Math. J. \textbf{44} (1977), 329--355.  
 
\bibitem[BPRS00]{tbdpirws:crg}
T.~Bates, D.~Pask, I.~Raeburn, and W.~Szymanski, \emph{{C*}-algebras of
  row-finite graphs}, New York J. Math. \textbf{6} (2000), 307--324.

\bibitem[Ben]{rb:kxrrzic}
R.~Bentmann, \emph{Kirchberg {$X$}-algebras with real rank zero and
  intermediate cancellation}, J. Noncommut. Geom., to appear,
  arXiv:1301.6652v1.

\bibitem[BK]{rbmk:uctcfts}
R.~{Bentmann} and M.~{K\"{o}hler}, \emph{Universal coefficient theorems for
  {C*}-algebras over finite topological spaces}, arXiv:1101.5702v3.

\bibitem[Bla77]{bb:itpc}
B.~Blackadar, \emph{Infinite tensor products of {C*}-algebras}, Pacific J.\
  Math. \textbf{72} (1977), 313--334.

\bibitem[Bla85]{bb:stc}
\bysame, \emph{Shape theory for {C*}-algebras}, Math.\ Scand. \textbf{56}
  (1985), 249--275.

\bibitem[Bla04]{bb:ssc}
\bysame, \emph{Semiprojectivity in simple {C*}-algebras}, Adv. Stud. Pure
  Math. \textbf{38} (2004), 1--17.

\bibitem[BD96]{lgbmd:ecq}
L.G. Brown and M.~D\u{a}d\u{a}rlat, \emph{Extensions of {C*}-algebras and
  quasidiagonality}, J. London Math. Soc. (2) \textbf{53} (1996), no.~3,
  582--600.

\bibitem[CL02]{jcml:tarfag}
J.~{Crisp} and M.~{Laca}, \emph{On the {Toeplitz} algebras of right-angled and
  finite-type {Artin} groups}, J. Austral. Math. Soc. \textbf{72} (2002),
  223--245.

\bibitem[CL07]{jcml:bqitaag}
\bysame, \emph{Boundary quotients and ideals of {Toeplitz} algebras of {Artin}
  groups groups}, J. Funct. Anal. \textbf{242} (2007), 127--156.

\bibitem[CEL13]{jcsexl:kcpasa}
J.~{Cuntz}, S.~{Echterhoff}, and X.~{Li}, \emph{On the {$K$}-theory of crossed
  products by automorphic semigroup actions}, Q. J. Math. (2013), doi:
  10.1093/qmath/hat021.

\bibitem[EK]{setk:spipgc}
S.~Eilers and T.~Katsura, \emph{Semiprojectivity and properly infinite
  projections in graph {C*}-algebras}, preprint.

\bibitem[EKTW]{setkmtjw:rking}
S.~Eilers, T.~Katsura, M.~Tomforde, and J.~West, \emph{The ranges of
  {$K$}-theoretical invariants for nonsimple graph algebras}, submitted for
  publication,  arXiv:1202.1989.
  
\bibitem[ELP99]{elp: Morphisms extensions}
S.~Eilers, T. ~A. Loring, and G.~K. Pedersen, \emph{Morphisms of extensions of {C*}-algebras: Pushing forward the {B}usby invariant}, Adv. Math. \textbf{147} (1999), 74--109.  

\bibitem[ERR]{segrer:scecc}
S.~Eilers, G.~Restorff, and E.~Ruiz, \emph{Strong classification of extensions
  of classifiable {C*}-algebras}, In preparation.

\bibitem[ERR09]{segrer:cecc}
\bysame, \emph{Classification of extensions of classifiable {C*}-algebras},
  Adv. Math. \textbf{222} (2009), 2153--2172.

\bibitem[ERR10]{segrer:gclil}
\bysame, \emph{On graph {C*}-algebras with a linear ideal lattice}, Bull.
  Malays. Math. Sci. Soc. \textbf{33} (2010), no.~2, 233--241.

\bibitem[ERR13]{segrer:ccfis}
\bysame, \emph{Classifying {C*}-algebras with both finite and infinite
  subquotients}, J. Funct. Anal. \textbf{265} (2013), 449--468.

\bibitem[End13]{de:sccsc}
D.~Enders, \emph{On the structure of certain classes of semiprojective
  {C*}-algebras}, Ph.D. thesis, {Westf\"{a}lische}
  {Wilhelms-Universit\"{a}t} M{\"u}nster, 2013.

\bibitem[EL91]{retal:iacu}
R.~Exel and T.A. Loring, \emph{Invariants of almost commuting unitaries}, J.\
  Funct.\ Anal. \textbf{95} (1991), 364--376.

\bibitem[FL07]{xfsl:eaca}
X.~{Fang} and S.~{Liu}, \emph{Extension algebras of {Cuntz} algebras}, J. Math.
  Anal. Appl \textbf{329} (2007), 655--663.

\bibitem[HLMRT]{HLMRT: non stable Ktheory}
D.~Hay, M.~Loving, M.~Montgomery, E.~Ruiz, and K.~Todd, \emph{Non-stable {$K$}-theory for Leavitt path algebras}, Rocky Mountain J. Math., to appear, arXiv:1211.1102.

\bibitem[{Iva}10]{ni:ktcrag}
N.~{Ivanov}, \emph{The {$K$}-theory of {Toeplitz} {C*}-algebras of
  right-angled {Artin} groups}, Trans. Amer. Math. Soc. \textbf{362} (2010),
  no.~11, 6003--6027.

\bibitem[JT91]{kkjkt:ek}
K.K.~Jensen and K.~Thomsen, \emph{Elements of {$K\!K$}-theory},
  {Birkh\"{a}user}, Boston, 1991.

\bibitem[Kir]{ek: purely infinite simple}
E.~Kirchberg, \emph{The classification of purely infinite {C*}-algebras using {K}asparov's theory} 3rd draft.  Preprint.

\bibitem[Kir00]{ek:nkmkna}
E.~Kirchberg, \emph{Das nicht-kommutative {M}ichael-{A}uswahlprinzip und die
  {K}lassifikation nicht-einfacher {A}lgebren}, $C\sp *$-algebras (M\"unster,
  1999), Springer, Berlin, 2000, pp.~92--141. 

\bibitem[KP00]{ekncp:eeccao}
E.~Kirchberg and N.~C. Phillips, \emph{Embedding of exact ${C}^*$-algebras in
  the {C}untz algebra $\mathcal {O}_ 2$}, J. Reine Angew. Math. \textbf{525}
  (2000), 17--53.

\bibitem[KR02]{ekmr:incacao}
E.~Kirchberg and M.~R{\o}rdam, \emph{Infinite non-simple ${C}^*$-algebras:
  absorbing the {Cuntz} algebras $\mathcal {O}_ \infty$}, Adv. Math.
  \textbf{167} (2002), 195--264.

\bibitem[{Li}12]{xl:scas} X.~{Li}, \emph{Semigroup {C*}-algebras and amenability of semigroups}, J. Funct. Anal. \textbf{262} (2012), 4302--4340.

\bibitem[{Li}13]{xl:nscca}
X.~{Li}, \emph{Nuclearity of semigroup {C*}-algebras and the connection to
  amenability}, Adv. Math. \textbf{244} (2013), 626--662.

\bibitem[Lor97]{tal:lsppc}
T.A. Loring, \emph{Lifting solutions to perturbing problems in
  {C*}-algebras}, Fields Institute Monographs, vol.~8, American Mathematical
  Society, Providence, RI, 1997.

\bibitem[NR97]{gnlr:gaccc}
G.~{Niblo} and L.~{Reeves}, \emph{Groups acting on {CAT(0)} cube complexes},
  Geom. \& Top. \textbf{1} (1997), 1--7.

\bibitem[{Par}02]{lp:amig}
L.~{Paris}, \emph{Artin monoids inject in their groups}, Comment. Math. Helv.
  \textbf{77} (2002), 609--637.

\bibitem[PR07]{cpmr:picrrz}
C.~{Pasnicu} and M.~{R{\o}rdam}, \emph{Purely infinite {C*}-algebras of real
  rank zero}, J. Reine Angew. Math. \textbf{613} (2007), 51--73.

\bibitem[Res06]{gr:classification of ck-algebras}
G.~Restorff, \emph{Classification of {C}untz-{K}rieger algebras up to stable isomorphism}, J. Reine Angew. Math. \textbf{598} (2006), 185--210.


\bibitem[RR07]{grer:rccconiII}
G.~Restorff and E.~Ruiz, \emph{On {R\o rdam's} classification of certain
  {C*}-algebras with one nontrivial ideal {II}}, {Math. Scand.} \textbf{101}
  (2007), 280--292.
 
\bibitem[Ror97]{mk: extensions kirchberg}
M.~R{\o}rdam, \emph{Classification of extensions of certain {C*}-algebras by their six term exact sequences in {$K$}-theory}, Math. Ann. \textbf{308} (1997), 93--117. 
 
\bibitem[Sor13]{as: geometric class}
A.~S{\o}rensen, \emph{Geometric classification of simple graph algebras}, Ergodic Theory Dynam. Systems \textbf{33} (2013), no. 4, 1199--1220. 
 
\bibitem[TW07]{atww:ssc}
A.~{Toms} and W.~{Winter}, \emph{Strongly self-absorbing {C*}-algebras},
  Trans. Amer. Math. Soc. \textbf{359} (2007), 3999--4029.

\bibitem[Voi83]{dv:acfruoca}
D.~Voiculescu, \emph{Asymptotically commuting finite rank unitary operators
  without commuting approximants}, Acta Sci. Math. (Szeged) \textbf{45} (1983),
  no.~1-4, 429--431.

\end{thebibliography}

\providecommand{\bysame}{\leavevmode\hbox to3em{\hrulefill}\thinspace}
\providecommand{\MR}{\relax\ifhmode\unskip\space\fi MR }
\providecommand{\MRhref}[2]{%
  \href{http://www.ams.org/mathscinet-getitem?mr=#1}{#2}
}
\providecommand{\href}[2]{#2}

\end{document}